\documentclass[12pt,english]{article}
\usepackage[margin=1in]{geometry}
\usepackage[T1]{fontenc}
\usepackage[utf8]{inputenc}

\usepackage{amsmath}
\usepackage{amsthm}
\usepackage{amssymb}

\usepackage{graphicx}
\usepackage{float}
\usepackage{xcolor}
\usepackage{shuffle}
\usepackage{comment}
\usepackage{stmaryrd}
\PassOptionsToPackage{normalem}{ulem}
\usepackage{ulem}
\usepackage{shuffle}
\usepackage{mathrsfs}

\usepackage{mathtools}
\usepackage{colortbl,dcolumn}
\usepackage[backend=biber,style=numeric,sorting=nyt]{biblatex}
\usepackage{enumerate}

\usepackage{hyperref}
\hypersetup{
     colorlinks   = true,
     citecolor    = blue,
     linkcolor    = blue
}

\makeatletter

\numberwithin{equation}{section}

\theoremstyle{plain}
\newtheorem{theorem}{Theorem}[section]
\newtheorem{lemma}[theorem]{Lemma}
\newtheorem{proposition}[theorem]{Proposition}
\newtheorem{corollary}[theorem]{Corollary}

\theoremstyle{definition}
\newtheorem{definition}[theorem]{Definition}
\newtheorem{remark}[theorem]{Remark}
\newtheorem{hypothesis}[theorem]{Hypothesis}

\newtheorem{condition}[theorem]{Condition}
\newtheorem*{definition*}{Definition}

\usepackage{babel}

\date{}

\makeatother

\def\xigeng#1{\textcolor{blue}{Xi: #1}}
\def\sheng#1{\textcolor{red}{Sheng: #1}}

\begin{document}
\title{Hitting Probabilities for Hypoelliptic Differential Equations Driven by Fractional Brownian Motion}
\author{Xi Geng\thanks{School of Mathematics and Statistics, University of Melbourne,  Parkville VIC 3010, Australia.
Email: xi.geng@unimelb.edu.au.}$\ \ $and  Sheng Wang\thanks{School of Mathematics and Statistics, University of Melbourne, Parkville VIC 3010, Australia.
Email: shewang4@student.unimelb.edu.au.}}
\maketitle
\begin{abstract}
The main goal of this article is to derive a two-sided estimate for hitting probabilities of a hypoelliptic stochastic differential equation (SDE) driven by fractional Brownian motion (fBM) with Hurst parameter $H\in(1/4,1)$ in terms of Newtonian-type capacities that are defined with respect to the (sub-Riemannian) control distance associated with the vector fields. As a starting point, we first establish the existence and smoothness of joint densities for the finite-dimensional distributions of the solution in the general context of hypoellitpic SDEs driven by Gaussian rough paths. We then turn to the fBM setting and derive a local upper bound for the joint density in terms of the control distance. As an application of these results, we establish our main estimate on hitting probabilities which generalises a well-known elliptic result of \cite{BNOT} to the hypoelliptic case. 

\tableofcontents{}
\end{abstract}

\section{Introduction}\label{section1}
Consider a stochastic differential equation (SDE) \begin{equation}\label{1aaaa}
dY_t=V(Y_t)~dX_t+V_0(Y_t)~dt,~~Y_0=y_0\in \mathbb{R}^N,
\end{equation}
where $X$ is a certain stochastic process in $\mathbb{R}^d$ and  $V_0, \cdots, V_d$ are suitably regular vector fields on $\mathbb{R}^N$. When $X$ is a semimartingale (e.g. Brownian motion) with the stochastic integral being interpreted in the sense of It\^o, the SDE \eqref{1aaaa} has been well studied within the framework of It\^o's calculus.   
In a non-semimartingale setting, a typical situation which has gained much interest is when $X$ is a fractional Brownian Motion (fBM). If the Hurst parameter $H>1/2$, the stochastic integral in \eqref{1aaaa} can be defined in the sense of Young's integration, and the well-posedness of the SDE was established in \cite{hu2007differential,lyons1994differential}. If $1/4<H<1/2$, the well-posedness of the SDE \eqref{1aaaa} requires techniques from rough path theory, which was originally introduced in the seminal work of Lyons \cite{lyons1998differential} in 1998. Controlled differential equations of the type (\ref{1aaaa}) with driven process $X$ being a general rough path are more commonly known as rough differential equations (RDEs).
Over the past decades, applications of rough path techniques to the study of quantitative properties of RDEs have been extensively developed by various authors in the literature. 

On the other hand, as a counterpart to this analytic development there has been growing interest in the study of probabilistic properties of the SDE \eqref{1aaaa}. An important aspect of this is the nondegeneracy and regularity of solutions. In the 1960s,  H\"ormander \cite{hormander1967hypoelliptic} proved the existence and smoothness of fundamental solution to the heat equation associated with a hypoelliptic second order differential operator. Later on, in the 1970s Malliavin \cite{malliavin1978stochastic} obtained a probabilistic proof of H\"ormander's result by showing that the solution to the SDE (\ref{1aaaa}) driven by Brownian motion admits a smooth density under H\"ormander's hypoellipticity condition on the vector fields. Along this pathway, Malliavin developed a powerful technique known as the Mallivian calculus, which has far-reaching applications to various regularity problems in stochastic analysis. 
A quantitative version of the Doob-Meyer decomposition was obtained by Norris \cite{norris2006simplified} which simplified Malliavin's proof and served as a key ingredient in this methodology. In the context of fBm with $H>1/2$ and elliptic vector fields, Nualart-Hu \cite{hu2007differential} established the existence of a smooth density for the soltuion. This result was extended to the hypoelliptic case by Baudoin-Hairer \cite{baudoin2007version}. More recently, by using rough path techniques Cass-Friz \cite{cass2010densities} obtained the existence of density under hypoelliptic vector fields where the driving process $X$ can be realised as a wide range of Gaussian processes, including fBM with $H>1/4$. The smoothness of density was established by Cass-Hairer-Litterer-Tindel \cite{CHLT} utilising a pathwise Norris' lemma in \cite{HP} together with a novel interpolation inequality.  

A natural stream of questions along this direction are related to the study of quantitative properties of solutions. For instance, in the diffusion context, the seminal paper of Kusuoka-Strook \cite{kusuoka1987applications} established global upper and lower estimates for the transition density under H\"ormander's condition. Cass-Litterer-Lyons \cite{CLL} derived the integrability of the Jacobian flow of the solution to RDEs driven by Gaussian rough paths, resulting in a well-known Weibull-type upper bound on the tail distribution of the solution. In a complementary direction, Boedihardjo-Geng \cite{boedihardjo2024lack} showed that the Cass-Litterer-Lyons estimate is optimal by proving a matching lower bound for a large class of rough integrals. For systems driven by fBm with $H>1/4$ under elliptic vector fields, Baudoin-Nualart-Ouyang-Tindel \cite{BNOT} obtained a global Gaussian upper bound on density of the solution, while a local lower bound was later obtained by Geng-Ouyang-Tindel \cite{geng2022precise} in the hypoelliptic case.  In \cite{BNOT}, the authors also obtained existence, smoothness and quantitative estimates of the joint density (over a pair of times) in the elliptic case, which enabled them to establish a two-sided estimate for hitting probabilities of the solution in terms of Newtonian capacities.


In this article, we consider a hypoelliptic differential equation driven by fBM in the rough path regime $H\in(1/4,1)$. Our main objective is to establish a quantitative relationship between hitting probabilities of the solution and Newtonian-type capacities associated with the sub-Riemannian metric induced by the vector fields. This is a natural generalisation of  the elliptic result obtained by Baudoin-Nualart-Ouyang-Tindel \cite{BNOT}. In what follows, we provide an overview of our main results towards this goal. The main theorem is stated in Theorem \ref{GW3}.

\subsection{Existence and smoothness of joint density} 

Consider an RDE \eqref{1aaaa} driven by fBM with Hurst parameter $H\in(1/4,1)$. While the existence and smoothness of density at a fixed time $t$ is well-known under H\"ormander's condition, a natural question arises: does the joint distribution $(Y_{t_1},\ldots,Y_{t_m})$ admit a (smooth) density? If so, what quantitative estimates can one obtain for the joint density? As seen in \cite{BNOT,dalang2007hitting} for instance, joint density estimates play an important role in the analysis of hitting probabilities, which is often studied using potential-theoretic techniques in the Markovian setting. 

Our first main result provides an affirmative answer to this question. We should point out that the existence of joint density is not an immediate consequence of the existence of marginal density, due to the lack of Markovian properties in the current setting. To formulate the result precisely, we first present the basic assumptions. 

\begin{hypothesis}\label{h1} The vector fields $V_0,\cdots,V_d$ belong to $C_b^\infty(\mathbb{R}^N;\mathbb{R}^N)$ (i.e. smooth with bounded derivatives of all orders).
\end{hypothesis}

To state the key non-degeneracy condition, we need to introduce some notation. 
Given a word $I=(i_1,\cdots,i_k)$ over the letters $\left\{0,1,2,\cdots,d\right\}$, we set
$$
V_I=[V_{i_1},[V_{i_2},\cdots[V_{i_{k-1}},V_{i_k}]\cdots]],
$$
where $[V_i,V_j]$ denotes the Lie bracket between two vector fields; as a   differential operator it is defined by $V_iV_j-V_jV_i$  and one identifies a vector field with a differential operator (taking directional derivatives). The length of a word $I$ is denoted as $|I|$. For each $l\geqslant 1$ and $x\in\mathbb{R}^N$, we introduce the following family of vectors at $x$:
\begin{equation*}  
\mathcal{W}_l(x)=\left\{V_I(x):~|I|=k\leqslant l,~1\leqslant i_k\leqslant d~\text{and}~0\leqslant i_r\leqslant d~\text{for }1\leqslant r\leqslant k-1\right\}.
\end{equation*}
\begin{hypothesis}\label{1.2} There exist an integer $\bar{l}\geqslant1$ and a positive real number $\lambda$ such that
$$
\sum_{W\in\mathcal{W}_{\bar{l}}(x)}\langle W,u\rangle^2_{\mathbb{R}^N}\geqslant\lambda\|u\|^2_{\mathbb{R}^N}
$$
holds uniformly for all $x,u\in \mathbb{R}^N$. The number $\bar{l}$ is referred to as the \textit{hypoellipticity constant}. 
\end{hypothesis}

This is often known as the \textit{uniform H\"ormander condition} which plays a basic role in studying regularity properties of the solution as well as quantitative estimates for the density. 
Our first result is state as follows (see Theorem \ref{GW4} for a more precise formulation).
\begin{theorem}\label{GW1}
    Consider an RDE 
    \begin{equation}\label{1}
d\mathbf{Y}_t=V(Y_t)~d\mathbf{X}_t+V_0(Y_t)~dt,~~Y_0=y_0\in \mathbb{R}^N.
\end{equation}Here $\mathbf{X}$ is a Gaussian rough path including fBM with Hurst parameter $H\in(1/4,1)$, Ornstein-Uhlenbeck process, Brownian bridge as examples (the precise conditions on $\mathbf{X}$ are provided in Section \ref{sec:GRP} below). The vector fields $\mathcal{V}=\{V_i\}_{i=1}^d$ are assumed to satisfy Hypotheses $\ref{h1}$ and $\ref{1.2}$. Let $Y$ be the solution to the the RDE. Then for any $m\geqslant 1$ and $0<t_1<\cdots<t_m\leqslant T$, the distribution of $(Y_{t_1},\cdots,Y_{t_m})$ admits a smooth joint density with respect to the Lebesgue measure on $\mathbb{R}^{mN}$.
\end{theorem}

\subsection{Local upper bound for joint density} 

A crucial ingredient in the analysis of hitting probabilities is a precise local estimate for the joint density of solution over a pair of times $(s,t)$. In what follows, we assume that $V_0 = 0$ and the driving process is a $d$-dimensional fBM with Hurst parameter $H\in (1/4,1)$. In other words, we consider the following RDE
\begin{equation}\label{density1}
d\mathbf{Y}_t=V(Y_t)~d\mathbf{B}_t,~~Y_0=y_0\in \mathbb{R}^N,
\end{equation}
where $\mathbf{B}$ is the canonical rough path lifting of $B$. It is known from Theorem \ref{GW1} that $(Y_s,Y_t)$ admits a smooth joint density. To state our second main result, we need to introduce the following essential definition.

\begin{definition}The \textit{control distance function} associated with the vector fields $\{V_i\}_{i=1}^d$ is defined by\begin{equation}\label{eq:ODE}
d(x,y):=\inf\left\{\big\|\dot{h}\big\|_{L^2([0,1])}: \Pi_1(x,h)=y\right\}.
\end{equation}Here $\Pi_1(x,h)$ denotes the solution to the equation (\ref{density1}) at time $t=1$ with driving path $h$ and initial location $x$. The infimum ranges over all those $h:[0,1]\rightarrow\mathbb{R}^d$ which are absolutely continuous with $L^2$-derivatives and satisfy the condition in \eqref{eq:ODE}.
\end{definition}
We use $B_d(x,r)=\left\{z\in \mathbb{R}^N: d(x,z)<r\right\}$ to denote the ball centered at $x$ with radius $r$ under the control distance $d$ and  $|B_d(x,r)|_{\text{Vol}}$ to denote its Lebesgue measure. Our second main result provides a (sharp) local upper estimate of the joint density of $(Y_s,Y_t)$. 
\begin{theorem}\label{GW2}
Suppose that the vector fields $\mathcal{V}=\{V_i\}_{i=1}^d$ satisfy Hypotheses \ref{h1} and \ref{1.2}.  Let $Y$ be the solution to the RDE (\ref{density1}) on $[0,T]$ and let $p(s,t,x,y)$ denote the joint density of $(Y_s,Y_t)$. For any $\delta>0$, there exist constants $C,\varepsilon>0$ depending on $H,\mathcal{V},T,\delta$, such that
\begin{equation}\label{upper123}
p(s,t,x,y)\leqslant \frac{C}{|B_d(x,(t-s)^H)|_{\mathrm{Vol}}}
\end{equation}
holds uniformly for all $s<t\in (\delta,T]$, $x,y\in\mathbb{R}^N$ with $d(x,y)\leqslant (t-s)^H$ and $t-s<\varepsilon$.
\end{theorem}

\begin{remark}\label{remark}
    If the vector fields are elliptic (e.g. $\bar{l}=1$ in Hypothesis \ref{1.2}), the control distance is locally equivalent to the Euclidean distance: there exist $C,\delta>0$ such that 
    \begin{equation}\label{equ}
      \frac{1}{C}\|x-y\|_{\mathbb{R}^N}\leqslant d(x,y) \leqslant C\|x-y\|_{\mathbb{R}^N},  
    \end{equation}
    for all $x,y\in \mathbb{R}^N$ with $\|x-y\|_{\mathbb{R}^N}<\delta$.
\end{remark}
Some related  quantitative density estimates in the literature are discussed as follows. 
\begin{itemize}
    \item In the diffusion context ($B$ is the Brownian motion), Kusuoka-Strook \cite{kusuoka1987applications} obtained a global upper and lower bound for the transition density $p(t,x,y)$ under the uniform H\"ormander condition:
    $$
    \frac{1}{M|B_d(x,t^{1/2})|_{\mathrm{Vol}}} e^{-Md(x,y)^2/t}\leqslant p(t,x,y)\leqslant\frac{M}{|B_d(x,t^{1/2})|_{\mathrm{Vol}}} e^{-d(x,y)^2/Mt}. 
    $$
    \item For the fBM driven case with elliptic vector fields, 
    Baudoin-Nualart-Ouyang-Tindel \cite{BNOT} established the following upper estimate for the joint density $p(s,t,x,y)$ of $(Y_s,Y_t)$:
    $$
    p(s,t,x,y)\leqslant\frac{C_M}{(t-s)^{nH}}\left(\frac{|t-s|^H}{|x-y|}\wedge 1\right)^p
    $$for all $x,y\in [-M,M]^N$.
    \item For the fBM driven case with hypoelliptic vector fields (i.e. under the uniform H\"ormander condition), Geng-Ouyang-Tindel \cite{geng2022precise} established the following local lower bound for the density:
    $$
    p(t,x,y)\geqslant \frac{C}{|B_d(x,t^H)|_{\mathrm{Vol}}}
    $$
    for all $(t,x,y)\in (0,1]\times \mathbb{R}^N\times \mathbb{R}^N$ satisfying $d(x,y)\leqslant t^H$ and $t$ small.
\end{itemize}

If one attempts to replicate the argument in \cite{BNOT} for the hypoelliptic case, one would only obtain a non-sharp upper bound for the joint density which does not yield optimal capacity estimates for hitting probabilities. We will use Kusuoka-Stroock's method in \cite{kusuoka1987applications}, which respects the sub-Riemannian metric associated with the vector fields and hence yields sharp estimates.

\subsection{Hitting probabilities and hypoelliptic capacities}

Finally, as an application of Theorem~\ref{GW2} we establish a quantitative relationship between hitting probabilities of sets in $\mathbb{R}^N$ by the solution to the RDE \eqref{density1} and the associated Newtonian-type capacities. A classical problem in potential theory is to derive conditions under which a process hits a given set $A$ with positive probability. Typically, such conditions are described in terms of certain capacities. In our setting, we shall look for the following type of relations:
\[
\mathbb{P}(Y([0,T]) \cap A \neq \emptyset) > 0 \quad \text{iff} \quad \mathrm{Cap}_\alpha(A) > 0.
\]Here $\mathrm{Cap}_\alpha(A)$ is a suitable notion of $\alpha$-dimensional capacity, which is defined in terms of the control distance function $d$ (see \eqref{eq:ODE}).

Let us now define the capacity precisely. Given any real number $\alpha$, the \textit{$\alpha$-dimensional Newtonian kernel} is defined by \begin{equation}\label{Kalpha}
K_\alpha(r) :=
\begin{cases}
r^{-\alpha}, & \text{if } \alpha > 0; \\
\log(N_0/r), & \text{if } \alpha = 0; \\
1, & \text{if } \alpha < 0,
\end{cases}
\end{equation}
where $N_0>0$ is a suitable constant whose value is not particularly important. Let $m(\cdot,\cdot)$ be a given metric on $\mathbb{R}^N$. For any probability measure $\mu$ on $\mathbb{R}^N$ with compact support,  its \textit{$\alpha$-dimensional Newtonian energy with respect to the metric $m$} is defined by
\begin{equation}\label{Kam}
\mathcal{E}_\alpha(\mu) := \int_{\mathbb{R}^N\times\mathbb{R}^N} K_\alpha(m(x,y))\,\mu(dx)\,\mu(dy).
\end{equation} 
\begin{definition}\label{def:Cap}
   Let $A$ be a Borel subset of $\mathbb{R}^N$. The \textit{$\alpha$-dimensional capacity of $A$ with respect to the metric $m$} is the reciprocal of the minimal energy:\[
\mathrm{Cap}_\alpha(A) := \left[\inf_{\mu \in \mathcal{P}(A)} \mathcal{E}_\alpha(\mu)\right]^{-1},
\]where $\mathcal{P}(A)$ denotes the set of probability measures that are compactly supported within $A$.
\end{definition}

To relate the capacity to dynamical properties of the process, the choice of the metric $m(\cdot,\cdot)$ is connected with both the dimensionality of the state space and the way the process propagates. 
For elliptic systems, the process diffuses uniformly in all directions, and therefore the Euclidean distance accurately measures its propagation. Taking $m$ to be the Euclidean distance, i.e. $m(x,y)=\|x-y\|_{\mathbb{R}^N}$, a considerable amount of work has been done to characterise hitting probabilities for a wide range of stochastic processes (see \cite{Fuk84} and references therein). 

In the Markovian setting, this is effectively studied by means of techniques from classical potential theory. In a non-Markovian setting, Baudoin-Nualart-Ouyang-Tindel \cite{BNOT} proved the following quantitative estimate. Let $Y$ be the solution to the RDE \eqref{density1} driven by a $d$-dimensional fBM with Hurst parameter $H \in (1/4,1)$. Suppose that the vector fields are uniformly elliptic (i.e. satisfying Hypothesis \ref{1.2} with $\bar{l}=1$). Fix $0 < a < b \leqslant 1$, $M > 0$ and $\eta > 0$. Then there exist positive constants $C_1, C_2$ depending on $a, b, M,\eta$, such that the estimate
\begin{equation}\label{equ(10)}
C_1\,\mathrm{Cap}_{N - 1/H}(A) \leqslant \mathbb{P}(Y([a,b]) \cap A \neq \emptyset) \leqslant C_2\,\mathrm{Cap}_{N - 1/H - \eta}(A)
\end{equation}holds uniformly for all compact sets $A \subseteq [-M, M]^N$. Here the capacity is defined with respect to the Euclidean distance. This generalises the well-known result that ``a $d$-dimensional fBM hits points iff $d < 1/H$'' (see \cite{Xia09} and references therein). 

On the other hand, for hypoelliptic systems, the geometry induced by the dynamics is not uniform in all directions since those directions along Lie brackets of the vector fields are harder to explore. The control distance $d$ defined by \eqref{eq:ODE} captures the underlying sub-Riemannian geometry through which the process propagates. We now define the intrinsic dimension associated with this geometry (see \cite[Section 20.1]{subRiemanbook2019}). 

Let $\mathcal{V}=\{V_1,\cdots,V_d\}$ be a family of smooth vector fields which satisfies Hypothesis \ref{1.2}. Let $\mathcal{D}$ denote the $C^\infty(\mathbb{R}^N;\mathbb{R})$-module generated by $\mathcal{V}$. For each $k\geqslant 1$, we define the submodule $\mathcal{D}^k$ of $C^\infty(\mathbb{R}^N;\mathbb{R}^N)$ recursively by 
\[
\mathcal{D}^1:= \mathcal{D},\ \mathcal{D}^{k+1}:= \mathcal{D}^k + [\mathcal{D},\mathcal{D}^k]. 
\]We also set $\mathcal{D}_x^k := \{W_x:W\in\mathcal{D}^k\}$. The uniform H\"ormander condition ensures that for each $x\in\mathbb{R}^N$ there is a smallest integer $r(x)\leqslant \bar{l}$ such that $\mathcal{D}^{r(x)}_x = \mathbb{R}^N$. This gives rise to a \textit{canonical flag of} $\mathcal{D}$ at $x$: 
\[
\{0\}=:\mathcal{D}^0_x\subseteq \mathcal{D}^1_x \subseteq \cdots \subseteq \mathcal{D}^{r(x)}_x = \mathbb{R}^N.
\]The list $\{\dim \mathcal{D}_x^k \}_{1\leqslant k\leqslant r(x)}$ of integers is called the \textit{growth vector} of $\mathcal{V}$ at $x$. The vector fields in $\mathcal{V}$ are said to be \textit{equiregular} if the growth vector is constant when $x$ varies over $\mathbb{R}^N$. 

\begin{definition}\label{defofhomodim}
Suppose that $\mathcal{V}$ is equiregular. The \textit{homogeneous dimension} of $\mathcal{V}$ is defined by
\[
Q := \sum_{k=1}^{r(x)} k \left( \dim \mathcal{D}^k_x - \dim \mathcal{D}^{k-1}_x \right),
\]which is independent of $x$ by assumption.
\end{definition}

We now state our final result which is also the main theorem of the present article. In what follows, we choose $m(\cdot,\cdot)$ to be the control distance $d$ in Definition \ref{def:Cap} of the capacity.
\begin{theorem}\label{GW3}
Suppose that the family $\mathcal{V}$ of vector fields satisfies Hypotheses \ref{h1}, \ref{1.2} and is equiregular with homogeneous dimension $Q$.
Let \( Y \) be the solution to the RDE \eqref{density1} driven by a \( d \)-dimensional fractional Brownian motion \( B \) with Hurst parameter \( H \in (1/4,1) \). For any fixed real numbers \( 0 < a < b \), \( M > 0 \), \( \eta_1 > 0 \) and \( 0 < \eta_2 < 1/H \), there exist  positive constants
\[
C_1 = C_1(a, b, \mathcal{V}, H, M, \eta_1, Q), \quad C_2 = C_2(a, b, \mathcal{V}, H, M, \eta_2, Q)
\]
such that the estimate
\begin{equation}\label{gcupplowboun123}
C_2 \, \mathrm{Cap}_{Q - 1/H + \eta_2}(A)
\ \leqslant\
\mathbb{P}(Y([a,b]) \cap A \neq \emptyset)
\ \leqslant\
C_1 \, \mathrm{Cap}_{Q - 1/H - \eta_1}(A)
\end{equation}holds uniformly for all compact sets \( A \subseteq B_d(o,M) \). 
In addition, if $Q=N$ one can take $\eta_2=0$.
\end{theorem}

\begin{remark}
Theorem \ref{GW3} is a natural extension of the elliptic result (\ref{equ(10)}) of \cite{BNOT}. In view of (\ref{equ(10)}), the estimate \eqref{gcupplowboun123} is nearly optimal since $\eta_1,\eta_2$ are arbitrarily small numbers.
\end{remark}

\subsection{Summary of contributions}
In what follows, we summarise the novelty and contributions of the current work.

\vspace{2mm}\noindent (A) To prove the existence and smoothness of joint density for $(Y_{t_1},Y_{t_2},\cdots,Y_{t_m})$, instead of successively conditioning on  $\mathcal{F}_{t_m},\cdots,\mathcal{F}_{t_1}$ with step-by-step applications of Malliavin calculus, we analyse the joint Malliavin covariance matrix directly. Then, by a suitable change of variable, the required nondegeneracy of the joint Malliavin matrix is seen as a consequence of the nondegeneracy of the vector fields over finitely many points, which is ensured by Hypothesis~\ref{1.2}.

\vspace{2mm}\noindent (B) By invoking a disintegration formula from Riemannian manifold, we are able to relate the joint density of $(Y_s,Y_t)$ with the distribution of the truncated signature process $\mathbf{B}_{s,t}$ on free nilpotent Lie groups. The problem thus reduces to establishing a smooth upper bound for the conditional density of the signature of fBM which is more manageable due to its explicit structure.

\vspace{2mm}\noindent (C) We employ the control distance \eqref{eq:ODE} (rather than Euclidean) in Definition \ref{def:Cap} of the capacity. It accurately captures the hypoelliptic behaviour of the dynamics and therefore effectively controls the hitting probabilities through the main estimate \eqref{gcupplowboun123}.


\vspace{2mm}\noindent (D) In the hypoelliptic case, the Newtonian kernel \( K(d(x,y)) \) lacks translation invariance  which prevents the use of Fourier analysis and invalidates the elliptic argument. To address this, we slightly increase the energy dimension to \( \alpha + \eta \) under which one can effectively estimate the energy of smoothened measures. 

\vspace{2mm}\noindent\textbf{Organisation}. 
In Section \ref{section2}, we discuss some basic notions and preliminary tools that are needed for our later analysis. In Section \ref{section3}, we prove the existence of smooth joint densities for the finite-dimensional distributions of the solution. In Section \ref{section4}, we derive a local upper estimate for the joint density, which is 
a crucial ingredient for establishing our main hitting probability estimate. In Section \ref{section5}, we develop the proof of the main Theorem \ref{GW3}. 

\section{Preliminaries}\label{section2}
In this section, we first review some basic notions on Gaussian rough paths and then restrict to the fractional Brownian motion setting where more tools are discussed. 

\subsection{Gaussian rough paths}\label{sec:GRP}

We begin by recalling the following definition. Given a mean-zero, real-valued Gaussian process $\{Z_t:t\in[0,T]\}$ in $\mathbb{R}$, let $R(s,t):=\mathbb{E}[Z_sZ_t]$ denote its covariance function.
\begin{definition}
Let $1 \leqslant \rho < 2$. The covariance function $\phi$ is said to have \textit{finite (2D) $\rho$-variation} if
\begin{equation*}
    V_{\rho}(R;[0,T]\times[0,T])^{\rho} 
    := \sup_{\mathcal{D},\mathcal{D}'} 
    \sum_{\substack{[s,t]\in \mathcal{D} \\ [u,v]\in \mathcal{D}'}} 
    \big| R \left(s,t;v,u\right) \big|^{\rho} < \infty,
\end{equation*}
where $R \left(s,t;v,u\right):=\mathbb{E}[(Z_t-Z_s)(Z_u-Z_v)]$ and $\mathcal{D},\mathcal{D}'$ are any finite partitions of $[0,T]$. The function $\phi$ is said to have \textit{finite (2D) H\"older-controlled $\rho$-variation}, if there exists $C>0$ such that 
$$
V_\rho(R;[s,t] \times [s,t]) \leqslant C (t-s)^{1/\rho}
$$for all $0 \leqslant s \leqslant t \leqslant T$. 
\end{definition}
\begin{theorem}[see \cite{FV1}]\label{thm:GLift}
    Let $(X_t)_{0\leqslant t\leqslant T}=(X^1_t,\cdots,X^d_t)$ be a mean-zero Gaussian process with i.i.d. components. Let $R$ denote the covariance function of any of its components. Suppose that $R$ is of finite 2D $\rho$-variation for some $\rho\in[1,2)$. Then the process $X$ admits a canonical lifting to a geometric $p$-rough path $\mathbf{X}$ for all $p>2\rho$. In addition, its (intrinsic) Cameron-Martin space $\bar{\mathcal{H}}$ can be embedded into the space $C^{\rho\text{-var}}([0,T],\mathbb{R}^d)$, more precisely, one has
    $$
    \|h\|_{\bar{\mathcal{H}}}\geqslant\frac{|h|_{\rho\text{-var};[0,T]}}{\sqrt{V_\rho(R;[0,T]\times[0,T])}}.
    $$
\end{theorem}
\begin{remark}
    Writing $\bar{\mathcal{H}}_B$ for the Cameron-Martin space of fBm for $H\in(1/4,1/2)$, the variation embedding in \cite{FV2} gives the stronger result that
    $$
    \bar{\mathcal{H}}_B\hookrightarrow C^{q\text{-var}}([0,T],\mathbb{R}^d)\quad\quad\forall~q>(H+1/2)^{-1}.
    $$
 \end{remark}

 In what follows, we always assume that $(X_t)_{0\leqslant t\leqslant T}=(X^1_t,\cdots,X^d_t)$ is a mean-zero Gaussian process with i.i.d. components. Let $R$ be the covariance function of any of its components. To study regularity properties of rough differential equations driven by a Gaussian process, one needs the following standard assumption.
\begin{condition}\label{con1} We assume that $R$ has finite H\"older controlled $\rho$-variation for some $\rho\in[1,2)$. In particular, $X$ admits a canonical lifting to a geometric $p$-rough path $\mathbf{X}$. We also assume that $\bar{\mathcal{H}}$, the Cameron-Martin space associated with $X$, has \textit{Young complementary regularity} with respect to $X$ in the following sense: one has the continuous embedding
  $$
  \bar{\mathcal{H}}\hookrightarrow C^{q\text{-var}}([0,T],\mathbb{R}^d)
  $$for some $q\geqslant1$ satisfying $1/p+1/q>1$.
\end{condition}
\begin{remark}
Complementary Young regularity for fBM holds when $H>1/4$ by taking $\rho\in[1,3/2)$.
\end{remark}

The following two conditions are fundamental in the study of non-degeneracy (existence and smoothness of density) for RDEs driven by a  Gaussian process (see \cite{CHLT}). Let $Z_t$ denote any component of $X_t$ (recall that they are i.i.d. Gaussian).

\begin{condition}\label{coninf}
    We assume that there exists $\alpha>0$ such that
    \begin{equation}\label{nonde}
     \inf_{0\leqslant s<t\leqslant T}\frac{1}{(t-s)^\alpha}\mathrm{Var}(Z_{s,t}|\mathcal{F}_{0,s}\vee\mathcal{F}_{t,T})>0.
    \end{equation}
    The \textit{index of nondeterminism} is the smallest $\alpha$ for which (\ref{nonde}) is true.
\end{condition}

\begin{condition}\label{concov}
For every $[u,v]\subseteq[s,t]\subseteq[0,S]\subseteq[0,T]$, one has
$$
\textcolor{black}{\mathrm{Cov}(Z_{s,t},Z_{u,v}|\mathcal{F}_{0,s}\vee\mathcal{F}_{t,S})\geqslant0.}
$$ 
\end{condition}

\begin{remark}\label{rconinf}
    It was shown in \cite[Lemma 4.1]{CHLT} that for fBm with Hurst parameter $H\in(0,1/2)$, the relation (\ref{nonde}) holds for any $\alpha\in(1,2]$. The case when $H\in(1/2,1)$ was handled in \cite{PP} where $\alpha$ is shown to be  $H+1$. For the Ornstein-Uhlenbeck process and the Brownian bridge, one has $\alpha=1$. 
\end{remark}

\begin{definition}\label{holderroughness}
    Let $\theta\in(0,1)$. A path $X:[0,T]\to \mathbb{R}^d$ is called \textit{$\theta$-H\"older rough} if there exists a constant $c>0$ such that for every $s\in[0,T]$, every $\varepsilon$ in $(0,T/2]$, and every unit vector $\zeta$ in $\mathbb{R}^d$, there exists $t$ in $[0,T]$ such that $\varepsilon/2<|t-s|<\varepsilon$ and 
    $$
    |\langle\zeta,X_{s,t}\rangle|>c\varepsilon^\theta.
    $$
    The largest such constant is called the \textit{modulus of $\theta$-H\"older roughness} and is denoted as $L_\theta(X)$.
\end{definition}
\begin{remark}\label{holderintegrability}
    It was proven in \cite[Corollary 5.10]{CHLT} that under Condition \ref{coninf}, one has  $L_\theta(X)\in L^p$ for all $p>1$.
\end{remark}

\subsection{The Mandelbrot-van Ness representation}

Let \( B = (B_t)_{0 \leqslant t \leqslant T} \) be a one-dimensional fBM with Hurst parameter \( H \in (0,1) \). Its \textit{canonical Hilbert space} (non-intrinsic Cameron-Martin space) \( \mathcal{H}_B \) is defined to be the completion of the linear span of \( \{ \mathbf{1}_{[0,t]} : t \in [0,T] \} \) with respect to the inner product
\begin{equation}\label{eq:HBInner}
\langle \mathbf{1}_{[0,t]}, \mathbf{1}_{[0,s]} \rangle_{\mathcal{H}_B} = R_H(t,s),
\end{equation}
where $$R_H(t,s) := \tfrac{1}{2}(t^{2H} + s^{2H} - |t - s|^{2H})$$ denotes the covariance function of \( B \). Note that the canonical Hilbert space of the standard Brownian motion \( W = (W_t)_{t \in \mathbb{R}} \) is \( \mathcal{H}_W = L^2(\mathbb{R}) \) with
\[
\langle \mathbf{1}_{[0,t]}, \mathbf{1}_{[0,s]} \rangle_{\mathcal{H}_W} = \min(t, s).
\]

The fBM admits the so-called \textit{Mandelbrot-van Ness representation} which will be useful to us. More precisely, the covariance \( R_H \) admits the integral form
\begin{equation}\label{RHdecom}
R_H(t,s) = \int_{-\infty}^{t \wedge s} k_H(t,u)\,k_H(s,u) \, du,
\end{equation}
where \( k_H \) is the  kernel defined by
\[
k_H(t,u) := c_H \left((t-u)^{H - 1/2}_+ - (-u)^{H - 1/2}_+ \right)
\]
and \( c_H \) is an explicit constant ensuring \eqref{RHdecom} to hold. This leads to the representation
\begin{equation}\label{deffBmH}
B_t = \int_{-\infty}^t k_H(t,s) \, dW_s,\quad0\leqslant t\leqslant T
\end{equation}
which expresses \( B_t \) as a Wiener integral with respect to \( W \). This representation defines an isometric embedding of $\mathcal{H}_B$ into $\mathcal{H}_W$ via
\[
\mathscr{K}(\mathbf{1}_{[0,t]}) := k_H(t,\cdot)\,\mathbf{1}_{(-\infty,t]}(\cdot),\quad 0\leqslant t\leqslant T.
\]
Indeed, according to (\ref{RHdecom}) one has
\[
\langle \mathbf{1}_{[0,t]}, \mathbf{1}_{[0,s]} \rangle_{\mathcal{H}_B} = \langle \mathscr{K}(\mathbf{1}_{[0,t]}), \mathscr{K}(\mathbf{1}_{[0,s]}) \rangle_{\mathcal{H}_W}.
\]

\textcolor{black}{For any step function $g=\sum_{i=1}^na_i\mathbf{1}_{[t_{i-1},t_i)}$ on $[0,T]$, we define
\[
\mathcal{I}(g)=\sum_{i=1}^na_i(B_{t_i}-B_{t_{i-1}}).
\]
By the definition of the inner product (\ref{eq:HBInner}), one has
\begin{equation}
    \mathbb{E}[\mathcal{I}(f)\mathcal{I}(g)]=\langle f, g \rangle_{\mathcal{H}_B}
\end{equation}
for any step functions $f,g$. As a result, the map $\mathcal{I}$ can be extended to  an isometry $\mathcal{I}:\mathcal{H}_B\to L^2(\Omega)$. The (intrinsic) \textit{Cameron-Martin} space $\bar{\mathcal{H}}_B$ is defined to be 
\[
\bar{\mathcal{H}}_B:=\left\{\bar{h}~\middle|~\bar{h}_t=\mathbb{E}[B_t\mathcal{I}(h)],~0\leqslant t\leqslant T, ~h\in\mathcal{H}_B\right\},
\]
where the inner product in $\bar{\mathcal{H}}_B$ is induced by 
$
\langle\bar{f},\bar{g}\rangle_{\bar{\mathcal{H}}_B}=\langle f,g\rangle_{\mathcal{H}_B}.
$
For any $h\in\mathcal{H}_B$, we use $\bar{h}$ to denote its corresponding element in $\bar{\mathcal{H}}_B$.}

To avoid ambiguity, we adopt the notation \( \mathbf{D}_W \) (resp. \( \mathbf{D}_B \)) for the Malliavin derivative with respect to \( W \) (resp. \( B \)), and denote the corresponding Sobolev spaces by \( \mathbb{D}^{k,p}_W \) (resp. \( \mathbb{D}^{k,p}_B \)), where $k$ is the differentiability index and $p$ is the integrability index. Here $F\in\mathbb{D}^{k,p}_W$ means
$$
\sum_{j=0}^k\left(\mathbb{E}\left[ \|\mathbf{D}^j_W F\|^p_{(\mathcal{H}_{W})^{\otimes j}}\right] \right)^{1/p}<\infty,
$$
and analogously for $\mathbb{D}^{k,p}_B$. The connection between \( \mathbb{D}^{k,p}_W \) and \( \mathbb{D}^{k,p}_B \) is described by the following result.

\begin{proposition}\label{MS}
Let \( (W_u)_{-\infty< u\leqslant T} \) be a Brownian motion on the canonical Wiener space \( (\Omega, \mathcal{F}^W, \mathbb{P}) \). Given fixed $H\in(0,1)$, we define the process \( (B_u)_{0\leqslant u\leqslant T} \) by
\begin{equation*}
B_u = \int_{-\infty}^u k_H(u,s) \, dW_s,\quad 0\leqslant u\leqslant T,
\end{equation*}
where \( k_H \) is the Volterra kernel associated with the Hurst parameter $H$. Then the following relations hold true:
\begin{itemize}
    \item[(i)] \( \mathbb{D}^{1,2}_B = (\mathscr{K})^{-1}(\mathbb{D}^{1,2}_W) \);
    \item[(ii)] For all \( F \in \mathbb{D}^{1,2}_B \), one has \( \mathbf{D}_W F = \mathscr{K} \mathbf{D}_B F \).
\end{itemize}
\end{proposition}

We shall consider the Malliavin calculus with respect to the filtration
\[
\mathcal{F}^W_s := \sigma(W_u,~ -\infty < u \leqslant s)
\]
and define the following conditional Malliavin-Sobolev norm:
\begin{equation}\label{conditionalkp0}
\|F\|_{k,p,s} := \left( \sum_{j=0}^{k} \mathbb{E}\left[ \|\mathbf{D}^j_W F\|^p_{(\mathcal{H}_{W,s})^{\otimes j}} \mid \mathcal{F}^W_s \right] \right)^{1/p},
\end{equation}where \( \mathcal{H}_{W,s}:=L^2([s,T]\)  for any \( 0 \leqslant s \leqslant T \).

To ensure compatibility with the Hilbert space \( \mathcal{H}_{W,0} \), we introduce the kernel
\[
L_H(t,s) := c_H (t-s)^{H - 1/2}, \quad \text{for all } 0 \leqslant s \leqslant t \leqslant T.
\]
This defines an operator
\[
\mathscr{L}\left( \mathbf{1}_{[0,t]} \right) := L_H(t, \cdot) \mathbf{1}_{[0,t]}(\cdot), \quad 0 \leqslant t \leqslant T,
\]
and the associated Volterra-type process 
\[
L_t := c_H \int_0^t (t - r)^{H - 1/2} dW_r, \quad 0 \leqslant t \leqslant T.
\]
Note that for any \( 0 \leqslant s \leqslant t \leqslant T \), one has the identities
\begin{equation}\label{innerprodKtoL}
\begin{aligned}
\mathscr{K}(\mathbf{1}_{[0,t]})=&~\mathscr{L}(\mathbf{1}_{[0,t]})\quad\text{on  }[0,T]\\
\left\langle \mathscr{K}(\mathbf{1}_{[0,t]}), \mathscr{K}(\mathbf{1}_{[0,s]}) \right\rangle_{\mathcal{H}_{W,0}}
=&~
\left\langle \mathscr{L}(\mathbf{1}_{[0,t]}), \mathscr{L}(\mathbf{1}_{[0,s]}) \right\rangle_{\mathcal{H}_{W,0}}.
\end{aligned}
\end{equation}
This immediately leads to the following result.

\begin{proposition}\label{DWLDB}
For all \( F \in \mathbb{D}^{1,2}_B \), one has
\[
\mathbf{D}_W F = \mathscr{L} \left( \mathbf{D}_B F \right)
\]
as a \(L^2([0,T]) \)-valued random variable.
\end{proposition}

\begin{proof}

This follows directly from Proposition~\ref{MS} and \eqref{innerprodKtoL}.
\end{proof}

The following two lemmas are taken from \cite[Proposition 9 and Lemma 2 respectively]{BN}. Two processes are said to be \textit{equivalent} if their laws are mutually absolutely continuous.
\begin{lemma}\label{6.7}
For any $H\in(0,1)$, the process \((B_{t})_{0\leqslant t\leqslant T}\) is equivalent to \((L_{t})_{0\leqslant t\leqslant T}\).
\end{lemma}
\begin{lemma}\label{6.6}
Let $X^1_t,~X^2_t$ be Gaussian processes and assume that they are equivalent on $[0,T]$. Then there exists $C_T>0$ such that
$$
E\left[\left(\int_0^Tf(s)dX^1_s\right)^2\right]\leqslant C_TE\left[\left(\int_0^Tf(s)dX^2_s\right)^2\right]
$$for any step function $f$.
\end{lemma}

\subsection{Conditional decomposition of fractional Brownian motion}
Let \( W_t = (W^1_t, \dots, W^d_t) \) be a \( d \)-dimensional Brownian motion defined on \((-\infty, T] \). The Mandelbrot-van Ness representation defines an fBM with Hurst parameter \( H \in (0,1) \) as a Volterra-type process given by
\[
B_t = \int_{-\infty}^t k_H(t,u)\, dW_u, \quad 0 \leqslant t \leqslant T.
\]
For any \( 0 \leqslant s \leqslant u \leqslant T \), we decompose the increment \( B_{s,u} := B_u - B_s \) into two parts according to their \( \mathcal{F}^W_s \)-measurability:
\[
B_{s,u} = L_{s,u} + Q_{s,u},
\]
where
\begin{equation}
\begin{aligned}\label{defXY}
L_{s,u} &:= c_H \int_s^u (u - r)^{H - 1/2}\, dW_r,\\  
Q_{s,u} &:= c_H \int_{-\infty}^s \left[ (u - r)^{H - 1/2} - (s - r)^{H - 1/2} \right]\, dW_r.
\end{aligned}
\end{equation}
Clearly, \( (L_{s,u})_{s\leqslant u\leqslant T} \) is independent of \( \mathcal{F}^W_s \) while \( (Q_{s,u})_{s\leqslant u\leqslant T} \) is \( \mathcal{F}^W_s \)-measurable.

The following simple observation will be used for several times in the sequel.

\begin{lemma} \label{decom}
For any $\varepsilon > 0$ and $s\geqslant 0$, the joint distribution of the process
\[
\left( L_{s, s+\varepsilon u}, Q_{s, s+\varepsilon u} \right)_{0 \leqslant u \leqslant 1}
\]
coincides with that of
\[
\left( \varepsilon^H L_{0,u}, \varepsilon^H Q_{0,u} \right)_{0 \leqslant u \leqslant 1}.
\]
\end{lemma}

\begin{proof}
This is a simple consequence of scaling invariance.
\end{proof}

To compare the Malliavin derivative of functionals of Gaussian processes, we introduce the following notion of equivalence.

\begin{definition}\label{defequalm}
Let $X$ and $Y$ be random variables defined on the Wiener space $(\Omega,\mathcal{F}^W,\mathbb{P})$. We say that $X$ and $Y$ are \textit{equivalent in the sense of Malliavin} and simply write
$
X \overset{M}{=} Y,
$
if one has
\[
\left\| \mathbf{D}_W^k X \right\|_{\mathcal{H}_W^{\otimes k}} \overset{d}{=} \left\| \mathbf{D}_W^k Y \right\|_{\mathcal{H}_W^{\otimes k}}
\]for every integer $k \geqslant 0$ provided that both sides are well-defined.
Here \( \mathbf{D}_W^k \) denotes the \( k \)-th Malliavin derivative with respect to the Brownian motion \( W \).
\end{definition}

The following lemma provides a sufficient condition under which two random variables are equivalent in the above sense.

\begin{lemma}\label{equnorm}
Let \( \Phi^1, \Phi^2 \colon C([0,1]) \to \mathbb{R} \) be measurable functionals such that \( \Phi^1(L_{0,\cdot}) \) and \( \Phi^2(Q_{0,\cdot}) \) are smooth in the sense of Malliavin, i.e.,
\[
\mathbb{E}\left[ \left\| \mathbf{D}_W^k \Phi^1(L_{0,\cdot}) \right\|_{\mathcal{H}_W^{\otimes k}}^p \right] < \infty,\quad
\mathbb{E}\left[ \left\| \mathbf{D}_W^k \Phi^2(Q_{0,\cdot}) \right\|_{\mathcal{H}_W^{\otimes k}}^p \right] < \infty
\]for all $k \geqslant 0$ and $\ p \geqslant 1$.
Then for any \( s \geqslant 0 \) and \( \varepsilon > 0 \), one has
\[
\Phi^1(L_{0,\cdot}) \overset{M}{=} \Phi^1\left( \varepsilon^{-H} L_{s,s+\varepsilon \cdot} \right),\quad
\Phi^2(Q_{0,\cdot}) \overset{M}{=} \Phi^2\left( \varepsilon^{-H} Q_{s,s+\varepsilon \cdot} \right).
\]
\end{lemma}
\begin{proof}
To ease notation, we write \( L_u = L_{0,u} \) and \( K_u = \varepsilon^{-H} L_{s,s+\varepsilon u} \) for \( u \geqslant 0 \). It suffices to show that, for any integers \( l \geqslant 1 \), \( k \geqslant 0 \), any  \( f \in C_c^\infty(\mathbb{R}^l; \mathbb{R}) \), and any finite sequence \( 0 < t_1 < t_2 < \cdots < t_l \leqslant 1 \), the following two random variables
\[
\left\| \sum_{1 \leqslant i_1,\dots,i_k \leqslant l} \partial_{x_{i_1} \cdots x_{i_k}} f(L_{t_1}, \dots, L_{t_l}) \, \mathbf{D}_W L_{t_{i_1}} \otimes \cdots \otimes \mathbf{D}_W L_{t_{i_k}} \right\|_{\mathcal{H}_W^{\otimes k}}
\]
and
\[
\left\| \sum_{1 \leqslant i_1,\dots,i_k \leqslant l} \partial_{x_{i_1} \cdots x_{i_k}} f(K_{t_1}, \dots, K_{t_l}) \, \mathbf{D}_W K_{t_{i_1}} \otimes \cdots \otimes \mathbf{D}_W K_{t_{i_k}} \right\|_{\mathcal{H}^{\otimes k}}
\]are equal in distribution. But
this follows from the facts that both processes \( L_t \) and \( K_t \) are Wiener integrals with respect to the same Brownian motion \( (W_u)_{-\infty<u\leqslant T} \), and that
\[
(L_{t_1}, \dots, L_{t_l}) \overset{d}{=} (K_{t_1}, \dots, K_{t_l}).
\]
The same argument applies to the processes \( (Q_{0,u})_{0 \leqslant u \leqslant 1} \) and \( (\varepsilon^{-H} Q_{s,s+\varepsilon u})_{0 \leqslant u \leqslant 1} \), which  completes the proof of the lemma.
\end{proof}

\subsection{Conditional small ball estimates}

Integrability of the modulus of $\theta$-H\"older roughness plays an important role in the study of the existence of a smooth density. In our study, we need a conditional version of this property and in this subsection, we first derive some related conditional small ball estimates.

\begin{lemma}\label{6.4}
Let \( (B_t)_{t \geqslant 0} \) be a \( d \)-dimensional fBM with Hurst parameter \( H \in (0,1) \). Then there exists a positive constant \( C \) such that for any \( \delta > 0 \), any unit vector \( \phi \in \mathbb{R}^d \), and any \( x \in (0,1) \), one has
\[
\mathbb{P}\left( \sup_{u,v \in [0,\delta]} \left| \langle \phi, B_{u,v} \rangle \right| \leqslant x \,\big|\, \mathcal{F}_0^W \right) \leqslant  \exp\left( -C \delta x^{-1/H} \right),
\]
where \( B_{u,v} := B_v - B_u \) denotes the increment of the fBM.
\end{lemma}
\begin{proof}
Firstly, by the scaling property of fBM, one has
\begin{equation*}
\mathbb{P}\left( \sup_{u,v \in [0,\delta]} \left| \langle \phi, B_{u,v} \rangle \right| \leqslant x \,\big|\, \mathcal{F}_0^W \right)
= \mathbb{P}\left( \sup_{u,v \in [0,1]} \left| \langle \phi, B_{u,v} \rangle \right| \leqslant \frac{x}{\delta^H} \,\big|\, \mathcal{F}_0^W \right).
\end{equation*}
Secondly, for any \( \phi \in \mathbb{R}^d \) with \( \|\phi\|_{\mathbb{R}^d} = 1 \), one notes that
\[
\left( \langle \phi, B_t \rangle \right)_{t \geqslant 0} \overset{d}{=} \left( B_t^1 \right)_{t \geqslant 0},
\]
where \( (B_t^1)_{t \geqslant 0} \) denotes the first component $B$. Therefore, it suffices to establish the desired bound in the case \( \delta = 1 \) and \( \phi = e_1 \), i.e.
\[
\mathbb{P}\left( \sup_{u,v \in [0,1]} \left| B^1_{u,v} \right| \leqslant x \,\big|\, \mathcal{F}_0^W \right) \leqslant  \exp\left( -C x^{-1/H} \right)
\]
for some constant \( C > 0 \).

Fix \( x \in (0,1) \). We define a partition of the interval \( I = [0,1] \) in the follow way. Let
\(n := \left\lfloor x^{-1/H} \right\rfloor \geqslant 1\)
and define the partition sequence \(0=t_0<t_1<\cdots<t_{n+1}=1 \) by
\[
t_i := i x^{1/H}, \quad \text{for } i = 0,1,\ldots,n.
\]
Since $B_0=0$, it is obvious that
\begin{equation}\label{33}
\begin{aligned}
\mathbb{P}\left( \sup_{u,v \in [0,1]} \left| B^1_{u,v} \right| \leqslant x \,\big|\, \mathcal{F}_0^W \right) \leqslant&~ \mathbb{P}\left( \sup_{u \in [0,1]} \left| B^1_{u} \right| \leqslant x \,\big|\, \mathcal{F}_0^W \right)\\
\leqslant&~\mathbb{P}\left(\max_{i=1,\cdots,n}\left| B^1_{t_i} \right| \leqslant x \,\big|\, \mathcal{F}_0^W \right)
\end{aligned}
\end{equation}

To estimate \eqref{33}, we successively condition on the following filtrations in order:
$$
\mathcal{F}_{t_{n-1}}^W,~\mathcal{F}_{t_{n-2}}^W,\cdots,~\mathcal{F}_0^W.
$$
According to Lemma \ref{decom}, the distribution of $B^1_{t_{i-1},t_i}$ conditional on $\mathcal{F}_{t_{i-1}}^W$ is equal to the distribution of $(t_i-t_{i-1})^H B^1_{0,1}$ conditional on $\mathcal{F}_{0}^W$. Furthermore, one has $B^1_{0,1}=L^1_{0,1}+Q^1_{0,1}$. Here $L^1_{0,1},Q^1_{0,1}$ are Gaussian variables, where $L^1_{0,1}$ is independent of $\mathcal{F}_{0}^W$ and $Q^1_{0,1}$ is measurable with respect to $\mathcal{F}_{0}^W$. If we denote $\sigma^2$ as the variance of $L^1_{0,1}$, then one has 
\begin{align*}
\mathbb{P}\left( \left| B^1_{t_{i-1},t_i} \right| \leqslant x \,\big|\, \mathcal{F}_{t_{i-1}}^W \right)&\overset{M}{=}\mathbb{P}\left( (t_i-t_{i-1})^H\left| B^1_{0,1} \right| \leqslant x \,\big|\, \mathcal{F}_0^W \right) \\
&\leqslant\int_{-\frac{x}{\sigma(t_i-t_{i-1})^H}}^{\frac{x}{\sigma(t_i-t_{i-1})^H}}\frac{1}{\sqrt{2\pi}}e^{-y^2/2}dy.
\end{align*}
Here we used the fact that for a Gaussian random variable with fixed variance, the probability \( \mathbb{P}(|Z| \leqslant x) \) is maximised when the mean is zero. Note that the right hand side also serves as an upper bound for \( \mathbb{P}( | B^1_{t_i} | \leqslant x \,|\, \mathcal{F}_{t_{i-1}}^W) \), since \( B^1_{t_i} \) can be viewed as the increment \( B^1_{t_{i-1}, t_i} \) shifted by \( B^1_{t_{i-1}} \), and the bound depends only on the variance. By iterating the above estimate, one obtains that
\begin{align*}
&\mathbb{P}\left(\max_{i=1,\cdots,n}\left| B^1_{t_i} \right| \leqslant x \,\big|\, \mathcal{F}_0^W \right)\\
&\leqslant\mathbb{P}\left(\max_{i=1,\cdots,n-1}\left| B^1_{t_i} \right| \leqslant x \,\big|\, \mathcal{F}_0^W \right)\int_{-\frac{x}{\sigma(t_n-t_{n-1})^H}}^{\frac{x}{\sigma(t_n-t_{n-1})^H}}\frac{1}{\sqrt{2\pi}}e^{-y^2/2}dy\\
&\cdots\\
&\leqslant~\prod_{i=1}^n\int_{-\frac{x}{\sigma(t_n-t_{n-1})^H}}^{\frac{x}{\sigma(t_n-t_{n-1})^H}}\frac{1}{\sqrt{2\pi}}e^{-y^2/2}dy.
\end{align*}
Since $\frac{x}{\sigma(t_n-t_{n-1})^H}\leqslant\frac{1}{\sigma}$ by the construction of the partition, one has
\begin{equation}\label{A}
\mathbb{P}\left(\max_{i=1,\cdots,n}\left| B^1_{t_{i-1},t_i} \right| \leqslant x \,\big|\, \mathcal{F}_0^W \right)\leqslant e^{-Cn}\leqslant e^{-Cx^{-1/H}/2},
\end{equation}
where $C:=-\log\int_{-1/\sigma}^{1/\sigma}\frac{1}{\sqrt{2\pi}}e^{-y^2/2}dy$. The desired estimate thus follows.
\end{proof}

\begin{proposition}\label{6.5}
Let \( (B_t)_{t \geqslant 0} \) be a \( d \)-dimensional fBM with Hurst parameter \( H \in (0,1) \). Then there exist  positive constants \( C_1,C_2 \), such that one has
\[
\mathbb{P}\left(\inf_{\|\phi\|_{\mathbb{R}^d}=1}\sup_{u,v\in [s,s+\delta]}\left|\langle\phi,B_{u,v}\rangle\right|\leqslant x\,\middle|\,\mathcal{F}_s^W\right)\leqslant C_1\exp\left( -C_2 \delta x^{-1/H} \right)
\]for all $\delta> 0$, $s\geqslant0$ and $ x\in(0,1)$.
\end{proposition}
\begin{proof}
As in the proof of Lemma~\ref{6.4}, it suffices to prove the desired  bound in the case \( \delta = 1 \) and \( s = 0 \), i.e.
\[
\mathbb{P}\left( \inf_{\|\phi\|_{\mathbb{R}^d}=1} \sup_{u,v \in [0,1]} \left| \langle \phi, B_{u,v} \rangle \right| \leqslant x \,\middle|\, \mathcal{F}_0^W \right) \leqslant C_1 \exp\left( -C_2 x^{-1/H} \right).
\]
Since $B_0=0$, one first has
$$
\mathbb{P}\left( \inf_{\|\phi\|_{\mathbb{R}^d}=1}\sup_{u,v \in [0,1]} \left| \langle \phi, B_{u,v} \rangle \right| \leqslant x \,\middle|\, \mathcal{F}_0^W \right)\leqslant \mathbb{P}\left( \inf_{\|\phi\|_{\mathbb{R}^d}=1}\sup_{u\in [0,1]} \left| \langle \phi, B_{u} \rangle \right| \leqslant x \,\middle|\, \mathcal{F}_0^W \right).
$$
The right hand side is what we aim to control in the following argument.

According to \eqref{defXY}, the process \( (B_t)_{t \geqslant 0} \) can be decomposed as  
\[
B_t = L_{0,t} + Q_{0,t},
\]  
where \( (L_{0,t})_{t \geqslant 0} \) and \( (Q_{0,t})_{t \geqslant 0} \) are Gaussian processes, with \( (L_{0,t})_{t \geqslant 0} \) being independent of \( \mathcal{F}_0^W \) and \( (Q_{0,t})_{t \geqslant 0} \) being measurable with respect to \( \mathcal{F}_0^W \). For any unit vector \( \phi \in \mathbb{R}^d \), we define
\[
\langle \phi, B_t \rangle = \langle \phi, L_{0,t} \rangle + \langle \phi, Q_{0,t} \rangle := L_t^\phi + Q_t^\phi.
\]
It follows from  \cite[Equation (3.12)]{HP} that
\[
\mathbb{P}\left( \sup_{u \in [0,1]} \left| \langle \phi, B_u \rangle \right| \leqslant x \,\middle|\, \mathcal{F}_0^W \right) \leqslant C_1 \exp\left( -C_2 \sup_{u \in [0,1]} \left| Q_u^\phi \right| \right) \mathbb{P}\left( \sup_{u \in [0,1]} \left| L_u^\phi \right| \leqslant x \right),
\]
where \( C_1,C_2 > 0 \) do not depend on $B$ and $\phi$. Due to the scaling property of \( L^\phi \), the same argument leading to  \eqref{A} shows that
\[
\sup_{\|\phi\|_{\mathbb{R}^d}=1}\mathbb{P}\left( \sup_{u \in [0,1]} \left| L_u^\phi \right| \leqslant x \right) \leqslant \exp\left( -C_3 x^{-1/H} \right)
\]
for some \( C_3 > 0 \).

Combining the above estimates, it follows that
\begin{equation} \label{compactestimate}
\mathbb{P}\left( \sup_{u \in [0,1]} \left| \langle \phi, B_u \rangle \right| \leqslant x \,\middle|\, \mathcal{F}_0^W \right) \leqslant C_1 \exp\left( -C_2 \sup_{u \in [0,1]} \left| Q_u^\phi \right| - C_3 x^{-1/H} \right).
\end{equation}
One can now apply the compactness argument as \cite[Lemma 2]{HP}, specifically the steps following Equation (3.13) there, to bound the right hand side of \eqref{compactestimate} uniformly over all unit vectors \( \phi \). This yields the estimate
\[
\mathbb{P}\left( \inf_{\|\phi\|_{\mathbb{R}^d} = 1} \sup_{u \in [0,1]} \left| \langle \phi, B_u \rangle \right| \leqslant x \,\middle|\, \mathcal{F}_0^W \right) \leqslant C_4 \exp\left( -C_5 x^{-1/H} \right),
\]
which completes the proof of the result.
\end{proof}

\section{Existence and smoothness of joint density}\label{section3}

Our goal of this section is to prove the existence of a smooth joint density for the solution to the RDE \eqref{1} over multiple times under the hypoellipticity assumption. The main result is stated as follows. 

\begin{theorem}\label{GW4}
Let $(X_t)_{t\in[0,T]}=(X^1_t,\cdots,X^d_t)_{t\in[0,T]}$ be a continuous Gaussian process with i.i.d. components. Assume that every component of $X$ satisfies Condition $\ref{con1}$ for some $\rho\in[1,2)$, Condition $\ref{coninf}$ with index $\alpha<2/\rho$ and Condition $\ref{concov}$. Consider the RDE \begin{equation*}
d\mathbf{Y}_t=V(Y_t)~d\mathbf{X}_t+V_0(Y_t)~dt,~~Y_0=y_0\in \mathbb{R}^N,
\end{equation*}where $\mathbf{X}$ is viewed as geometric rough path with roughness $p>2\rho$. Suppose that the vector fields $\{V_i\}_{i=1}^d$ satisfy Hypothesis $\ref{h1}$ and $\ref{1.2}$. Then for any $m\geqslant 1$ and $0<t_1<\cdots<t_m\leqslant T$, the joint distribution of $(Y_{t_1},\cdots,Y_{t_m})$ admits a smooth joint density with respect to the Lebesgue measure.
\end{theorem}

The rest of this section is devoted to the proof of Theorem \ref{GW4}.

As a standard route from the perspective of Malliavin's calculus, the crucial point is to investigate the non-degeneracy of the Malliavin covariance matrix which we first define. 

\begin{definition}
Let $F\in \mathbb{D}^{1,2}_X(\mathbb{R}^N)$ be a random vector in $\mathbb{R}^N$. The \textit{Malliavin covariance matrix} (or simply \textit{Malliavin matrix}) of $F$ is the $N\times N$ random matrix $\gamma(F)$ whose $(i,j)$-entry is defined by 
$$
\gamma_{ij}(F):=\langle~\mathbf{D}_XF^i,\mathbf{D}_XF^j~\rangle_{\mathcal{H}_X}.
$$
\end{definition}

The Malliavin matrix of the solution $Y_t$ is closely related to the Jacobian flow. Let us first recall that $\mathbf{J}_t := \frac{\partial Y_t^x}{\partial{x}}$ ($Y_t^x$ denotes the solution to (\ref{1}) with initial value $x$) is the unique solution to the linear equation
\begin{equation*}
\mathbf{J}_t = \mathrm{Id}_n + \int_0^t DV_0(Y_s^x)\mathbf{J}_s\,ds + \sum_{j=1}^d \int_0^t DV_j(Y_s^x)\mathbf{J}_s\,dX_s^j,
\end{equation*}
and that the following result holds true (see \cite{CHLT}).

\begin{proposition}\label{3.1}
    Let $Y_t$ be the solution to equation \eqref{1} and suppose that the vector fields $\{V_i\}_{i=1}^d$ satisfy Hypothesis \ref{h1}. Then one has $Y_t\in \mathbb{D}^\infty_{X}(\mathbb{R}^N)$ and
    $$
    (\mathbf{D}^{(k)}_{X} Y_t)(s)=\mathbf{J}_{t, s}V_k(Y_s),~\quad\text{for all } k=1,\cdots,d\text{ and }~~0\leqslant s\leqslant t,
    $$
    where $\mathbf{D}^{(k)}_X$ means the $k$-th component of $\mathbf{D}_X$ and $\mathbf{J}_{t, s}=\mathbf{J}_t\mathbf{J}^{-1}_s$.
\end{proposition}

We are now in a position to introduce the Malliavin covariance matrix.


Given $0< t_1<\cdots<t_m\leqslant T$, let $F=(Y_{t_1},\cdots,Y_{t_m})$ being viewed as a $N\times m$-dimensional column vector. We write its Malliavin matrix as a block matrix
\begin{equation}\label{defgammaF}
\gamma(F)
=
\begin{bmatrix} 
\langle\mathbf{D}_XY_{t_1},\mathbf{D}_XY_{t_1}\rangle_{\mathcal{H}_X} & \dots  & \langle\mathbf{D}_XY_{t_1},\mathbf{D}_XY_{t_m}\rangle_{\mathcal{H}_X}\\

\vdots & \ddots & \vdots\\

\langle\mathbf{D}_XY_{t_m},\mathbf{D}_XY_{t_1}\rangle_{\mathcal{H}_X} & \dots  & \langle\mathbf{D}_XY_{t_m},\mathbf{D}_XY_{t_m}\rangle_{\mathcal{H}_X} 
\end{bmatrix}
\end{equation}
with blocks
$$
\langle\mathbf{D}_XY_{t_i},\mathbf{D}_XY_{t_j}\rangle_{\mathcal{H}_X}~=~\sum_{k=1}^d~\int_{[0,t_i]\times[0,t_{j}]}\mathbf{J}_{t_i, u}V_k(Y_u)\otimes \mathbf{J}_{t_j, v}V_k(Y_v)d\phi_{u,v}.
$$
According to Proposition $\ref{3.1}$, $(\mathbf{D}^{(k)}_XY_{t_i})(s)=\mathbf{J}_{t_i, s}V_k(Y_s)1_{[0,t_i]}(s)$, and the inner product within $\mathcal{H}_X$ is understood as a $2D$-Young integral with respect to the covariance function $R(u,v):=\mathbb{E}[X^i_uX^i_v]$.
The following result establishes the non-degeneracy of the Malliavin matrix $\gamma(F)$.

\begin{proposition}\label{nonden}
For any $q>1$, there exists  $C_q>0$ such that 
$$
\mathbb{P}\left(\inf_{\|a\|=1}\langle a,\gamma(F)a\rangle<\varepsilon\right)\leqslant C_q \varepsilon^q
$$for all $\varepsilon>0$.
\end{proposition}

\begin{proof}
We only consider the case $m=2$ (the general case only involves more notation). Let the random variable $R_X$ be the corresponding quantity $\mathcal{R}$ in the pathwise Norris' lemma (see \cite[Theorem 5.6, Equation (5.8)]{CHLT}).
According to  Remark~\ref{holderintegrability} and Cass--Litterer--Lyons~\cite{CLL}, $R_X$ possesses finite moments of all orders. Our proof is divided into the following four steps.  

\textit{Step 1: Lower bound for $\langle a,\gamma(F)a\rangle$}. We first derive a lower bound for $\langle a,\gamma(F)a\rangle$ in terms of the supremum norm.  
Let $a^*=(a_1^*,a_2^*)$, where each $a_i$ ($i=1,2$) is a column vector in $\mathbb{R}^N$ and the superscript $*$ means transposition.  
By the definition of $\gamma(F)$ in~\eqref{defgammaF}, it is  readily checked that  
\begin{equation*}
    \langle a,\gamma(F)a\rangle = \sum_{k=1}^d \Lambda^k, 
    \quad \text{where} \quad 
    \Lambda^k:=\int_{[0,T]^2} f^k_u f^k_v\, dR(u,v),
\end{equation*}
with  
\begin{equation*}
f^k_s:=
\begin{cases}
\left(a_1^*\mathbf{J}_{t_1, 0}+a_2^*\mathbf{J}_{t_2, 0}\right)\mathbf{J}_{0, s}V_k(Y_s), & 0\leqslant s<t_1, \\
a_2^*\mathbf{J}_{t_2, t_1}\mathbf{J}_{t_1,s}V_k(Y_s), & t_1\leqslant s\leqslant t_2.
\end{cases}
\end{equation*}
To simplify notation, we write $(\hat{a}^*_1,\hat{a}^*_2):=\left(a_1^*\mathbf{J}_{t_1, 0}+a_2^*\mathbf{J}_{t_2, 0},a_2^*\mathbf{J}_{t_2,t_1}\right)$ and  one has
\begin{equation}\label{newfks}
f^k_s=
\begin{cases}
\hat{a}^*_1\mathbf{J}_{0, s}V_k(Y_s), & 0\leqslant s<t_1, \\
\hat{a}^*_2\mathbf{J}_{t_1,s}V_k(Y_s), & t_1\leqslant s\leqslant t_2.
\end{cases}
\end{equation}
By applying \cite[Corollary 6.10]{CHLT}, one sees that $\|f^k\|_{\infty,[t_{i-1},t_i]}$ is upper bounded by the maximum of 
\begin{equation*}
 2\mathrm{Var}\!\left(X_{t_{i-1},t_i}\,\middle|\,
\mathcal{F}_{0,t_{i-1}}\vee\mathcal{F}_{t_i,T}\right)^{-1/2}
\left(\int_{[0,T]^2} f^k_uf^k_v\,dR(u,v)\right)^{1/2}
\end{equation*}
and 
\begin{equation*}
C\left(\int_{[0,T]^2} f^k_uf^k_v\,dR(u,v)\right)^{\tfrac{\gamma}{2\gamma+\alpha}}
\left\|f^k\right\|_{\gamma,[t_{i-1},t_i]}^{\tfrac{\alpha}{2\gamma+\alpha}}
\end{equation*}
where $C>0$ is some constant and $\gamma+1/\rho>1$. In particular, to obtain a upper bound on $\left\|f^k\right\|_{\infty,[t_{i-1},t_i]}$ one needs to control the following three terms separately: 
$$\left\| f^k\right\|_{\gamma,[t_{i-1},t_i]},\quad\Lambda^k,\quad\text{Var}\left(X_{t_{i-1},t_{i}}|\mathcal{F}_{0,t_{i-1}}\vee\mathcal{F}_{t_i,T}\right)^{-1}.
$$
\begin{enumerate}
    \item It is easy to see that $\left\| f^k\right\|_{\gamma,[t_{i-1},t_i]}$ is bounded by $R_X$ according to the definition of $R_X$.
    \item $\Lambda^k$ is bounded above by $\langle a,\gamma(F)a\rangle$.
    \item $\text{Var}\left(X_{t_{i-1},t_{i}}|\mathcal{F}_{0,t_{i-1}}\vee\mathcal{F}_{t_i,T}\right)^{-1}$ is bounded above by a positive constant due to Condition $\ref{coninf}$.
\end{enumerate}
As a result, there exist constants $C,q_1,q_2>0$ such that
$$
\left\| f^k\right\|_{\infty,[t_{i-1},t_i]}\leqslant C \langle a,\gamma(F)a\rangle^{q_1} R_X^{q_2}.
$$
Combining this with \eqref{newfks}, it follows that 
\begin{equation}\label{step1mainresult}
\left\| \hat{a}^*_i\mathbf{J}_{t_{i-1}, \cdot}V_k(Y_\cdot)\right\|_{\infty,[t_{i-1},t_i]}\leqslant C \langle a,\gamma(F)a\rangle^{q_1} R_X^{q_2}
\end{equation}for any $i=1,2$ and $k=1,\cdots,d$.

\textit{Step 2: Iterating the bound (\ref{step1mainresult}) to Lie brackets}. The next observation is that
$$
\hat{a}^*_i\mathbf{J}_{t_{i-1},t}V_I(Y_t)=\hat{a}^*_iV_I(Y_{t_{i-1}}) + \sum_{k=1}^d \int_{t_{i-1}}^t\hat{a}^*_i\mathbf{J}_{t_{i-1},u}V_{(k,I)}(Y_u)dX^k_u + \int_{t_{i-1}}^t\hat{a}^*_i\mathbf{J}_{t_{i-1},u}V_{(0,I)}(Y_u)du
$$
for all $I\in\mathcal{A}_1(\bar{l}-1)$, $i=1,2$ and $t_{i-1}\leqslant t\leqslant t_i$. This is easily obtained through the integration by parts formula. Together with Norris' lemma (see \cite[Theorem 5.6]{CHLT}), the above formula implies that there exist constants $C,q_1,q_2>0$, such that
$$
\left\| \hat{a}^*_i\mathbf{J}_{t_{i-1}, \cdot}V_{(k,I)}(Y_\cdot)\right\|_{\infty,[t_{i-1},t_i]}\leqslant C \left\| \hat{a}^*_i\mathbf{J}_{t_{i-1}, \cdot}V_{I}(Y_\cdot)\right\|_{\infty,[t_{i-1},t_i]}^{q_1} R_X^{q_2}
$$
for any $i=1,2$, $k=0,1,\cdots,d$ and $I\in\mathcal{A}_1(\bar{l}-1)$. Starting from the base estimate \eqref{step1mainresult}, an induction argument shows that
\begin{equation}\label{8}
\left\| \hat{a}^*_i\mathbf{J}_{t_{i-1}, \cdot}V_I(Y_\cdot)\right\|_{\infty,[t_{i-1},t_i]}\leqslant C \langle a,\gamma(F)a\rangle^{q_1} R_X^{q_2}
\end{equation}
for any $i=1,2$ and $I\in\mathcal{A}_1(\bar{l})$.

\textit{Step 3: Nondegeneracy of $\langle a,\gamma(F)a\rangle R_X$}. We are going to apply Hypothesis \ref{1.2} to find a constant lower bound for \eqref{8}. Indeed, if one evaluates $\hat{a}^*_i\mathbf{J}_{t_{i-1}, \cdot}V_I(Y_\cdot)$ at the starting points of every interval $[t_{i-1},t_i]$, the estimate \eqref{8} implies that there are constants $C,q_1,q_2>0$ such that
$$
\sup_{I\in\mathcal{A}_1(\bar{l})}| \hat{a}^*_1V_I(y_0)|^2+\sup_{I\in\mathcal{A}_1(\bar{l})}| \hat{a}^*_2V_I(Y_{t_1})|^2\leqslant C \langle a,\gamma(F)a\rangle^{q_1} R_X^{q_2}.
$$   
According to Hypothesis \ref{1.2}, one has 
$$
\sup_{I\in\mathcal{A}_1(\bar{l})}| \hat{a}^*_1V_I(y_0)|^2+\sup_{I\in\mathcal{A}_1(\bar{l})}| \hat{a}^*_2V_I(Y_{t_1})|^2\geqslant\lambda\left(\|\hat{a}_1\|^2_{\mathbb{R}^N}+\|\hat{a}_2\|^2_{\mathbb{R}^N}\right).
$$
Recalling $(\hat{a}^*_1,\hat{a}^*_2)=\left(a_1^*\mathbf{J}_{t_1, 0}+a_2^*\mathbf{J}_{t_2, 0},a_2^*\mathbf{J}_{t_2,t_1}\right)$, one can write
\begin{equation*}
\begin{bmatrix}
a_1\\
a_2\\
\end{bmatrix} 
=
\begin{bmatrix}
~~~\mathbf{J}_{0,t_1}& 0\\
-I& \mathbf{J}_{t_1,t_2}\\
\end{bmatrix}^*
\begin{bmatrix}
\hat{a}_1\\
\hat{a}_2\\
\end{bmatrix}:=\mathbf{Q}_{t_1,t_2}
\begin{bmatrix}
\hat{a}_1\\
\hat{a}_2\\
\end{bmatrix},
\end{equation*}
and one has
\begin{equation*}
1=\|a_1\|^2_{\mathbb{R}^N}+\|a_2\|^2_{\mathbb{R}^N}\leqslant \|\mathbf{Q}_{t_1,t_2}\|_{\mathbb{R}^{N\times N}}^2\left(\|\hat{a}_1\|^2_{\mathbb{R}^N}+\|\hat{a}_2\|^2_{\mathbb{R}^N}\right)
\end{equation*}
Assembling together the above estimates and taking infimum on both sides, one concludes that 
\begin{equation}\label{NondeagammaaRX}
\begin{aligned}
1\leqslant&~\lambda^{-1}C~\|\mathbf{Q}_{t_1,t_2}\|_{\mathbb{R}^{N\times N}}^2 R_X^{q_1}\inf_{\|a\|=1}\langle a,\gamma(F)a\rangle^{q_2}
\leqslant C_1~ R_X^{q_3}\inf_{\|a\|=1}\langle a,\gamma(F)a\rangle^{q_4}
\end{aligned}
\end{equation}
for some $C_1,q_1,q_2,q_3,q_4>0$, where we also used the fact that $\|\mathbf{Q}_{t_1,t_2}\|_{\mathbb{R}^{N\times N}}\leqslant C'R_X^p$ for some $p\geqslant 1$.

\textit{Step 4: Integrability of $\gamma(F)^{-1}$}. Let $q>1$ be given fixed. As a consequence of \eqref{NondeagammaaRX}, one obtains that
\begin{equation*}
\begin{aligned}
\mathbb{P}\left(\inf_{\|a\|=1}\langle a,\gamma(F)a\rangle<\varepsilon\right)\leqslant&~ \mathbb{P}\left(1\leqslant C_1R_X^{q_3}\varepsilon^{q_4}\right)
\leqslant~C_2\varepsilon^{q}\mathbb{E}[R_X^{q_5}]
\end{aligned}
\end{equation*}
for suitable constants $C_2,q_5>0$. The desired estimate thus follows.
\end{proof}

\begin{proof}[Proof of Theorem \ref{GW4}]
    Proposition \ref{nonden} implies that the Malliavin covariance matrix of $(Y_{t_1},\cdots,Y_{t_m})$ has inverse moments of all orders. The desired result thus follows from standard regularity theorems from Malliavin's calculus (see \cite{Nualart}).
\end{proof}


\section{Local upper estimate for the joint density of ($Y_s, Y_t$)}\label{section4}

In this section, we prove Theorem \ref{GW2}, which provides a precise local upper bound for the joint density of $(Y_s,Y_t)$ and will be essential for establishing capacity estimates for hitting probabilities in the next section.  

Our approach is inspired by Kusuoka-Stroock's technique in \cite{kusuoka1987applications}. The main idea is to approximate the actual solution by its stochastic Taylor expansion, whose nondegeneracy can be obtained from the nondegeneracy of the truncated signature process of $X$ together with the hypoellipticity of the vector fields $\mathcal{V}$.

\subsection{Free nilpotent Lie groups and Lie algebras}

To study the truncated signature process, we shall first recall some basic notions on free nilpotent Lie groups.

The truncated tensor algebra of order $l$ over $\mathbb{R}^d$ is denoted by $T^{(l)}$. Under addition and tensor product, $(T^{(l)}, +, \otimes)$ is an associative algebra. The set of homogeneous Lie polynomials of degree $k$ is denoted as $\mathcal{L}_k$, and the free nilpotent Lie algebra of order $l$ is denoted as $\mathfrak{g}^{(l)}$. The exponential map on $T^{(l)}$ is defined by
\begin{equation}\label{(2.6)}
    \exp(a) := \sum_{k=0}^\infty \frac{1}{k!} a^{\otimes k} \in T^{(l)},
\end{equation}
where the sum is indeed finite and hence well-defined. $T^{(l)}$ is equipped with a natural inner product induced by the Euclidean structure on $\mathbb{R}^d$ and the Hilbert-Schmidt tensor norm on each of the tensor products. To be more precise, the inner product on $(\mathbb{R}^d)^{\otimes k}$ is induced by
\[
\langle v_1 \otimes \cdots \otimes v_k, w_1 \otimes \cdots \otimes w_k \rangle_{(\mathbb{R}^d)^{\otimes k}} := \langle v_1, w_1 \rangle_{\mathbb{R}^d} \cdots \langle v_k, w_k \rangle_{\mathbb{R}^d}.
\]
The components $(\mathbb{R}^d)^{\otimes k}$ and $(\mathbb{R}^d)^{\otimes m}$ ($k \ne m$) are assumed to be orthogonal. The Hilbert-Schmidt tensor norm is denoted as $\|\cdot\|_\text{HS}$.

The following algebraic structure plays a fundamental role  in rough path theory.
\begin{definition}\label{def6.1}
The \textit{free nilpotent Lie group} $G^{(l)}$ of order $l$ is defined by
\[
G^{(l)} := \exp(\mathfrak{g}^{(l)}) \subseteq T^{(l)}.
\]
\end{definition}

Note that the exponential map is a diffeomorphism under which $\mathfrak{g}^{(l)}$ is the Lie algebra of $G^{(l)}$. We define the operation $\star$ to be the multiplication on $\mathfrak{g}^{(l)}$ induced from $G^{(l)}$ through the exponential map, namely
\begin{equation*}
    v \star u := \log(\exp(v) \otimes \exp(u)), \quad v, u \in \mathfrak{g}^{(l)}.
\end{equation*}

It will be useful in the sequel to have some basis elements available for the algebras introduced above. Let $\mathcal{A}(l)$ (respectively, $\mathcal{A}_1(l)$) denote the set of words (respectively, non-empty words) over $\{1, \dots, d\}$ of length at most $l$. We set $e_{\emptyset} := 1$, and for each word $I = (i_1, \dots, i_r) \in \mathcal{A}_1(l)$, we define
\[
e_{(I)} := e_{i_1} \otimes \cdots \otimes e_{i_r}, \quad \text{and} \quad e_{I} := [e_{i_1}, \cdots [e_{i_{r-2}}, [e_{i_{r-1}}, e_{i_r}]]\cdots],
\]
where $\{e_1, \dots, e_d\}$ denotes the canonical basis of $\mathbb{R}^d$. Then $\{e_{(I)} : I \in \mathcal{A}(l)\}$ forms an orthonormal basis of $T^{(l)}$ under the Hilbert-Schmidt tensor norm. In addition, one also has
\[
\mathcal{L}_k = \mathrm{Span}\{e_{I} :~|I|=k\},\quad\mathfrak{g}^{(l)} = \mathrm{Span}\{e_{I} : I \in \mathcal{A}_1(l)\}.
\]

As a closed subspace, $\mathfrak{g}^{(l)}$ inherits a canonical Hilbert structure from $T^{(l)}$ which makes it into a flat Riemannian manifold. The associated volume measure $du$ (the Lebesgue measure) on $\mathfrak{g}^{(l)}$ is left invariant with respect to the previously defined product $\star$. In addition, for each $\lambda > 0$, there is a dilation operation $\delta_\lambda : T^{(l)} \to T^{(l)}$ induced by $\delta_\lambda(a) \triangleq \lambda^k a$ if $a \in (\mathbb{R}^d)^{\otimes k}$, which satisfies the relation $\delta_\lambda \circ \exp = \exp \circ \delta_\lambda$.  One can easily show that
\begin{equation*}
    d u \circ \delta^{-1}_\lambda = \lambda^{-\nu} d u, \quad \text{where} \quad \nu := \sum_{k=1}^l k \dim(\mathcal{L}_k).
\end{equation*}
The number $\nu$ is known as the \textit{homogeneous dimension} of $G^{(l)}$.

We always fix the Euclidean norm on $\mathbb{R}^d$ in the remainder of the article. As far as the free nilpotent group $G^{(l)}$ is concerned, there are several useful metric structures. Among them we will use the extrinsic Hilbert-Schmidt metric $\rho_{\mathrm{HS}}$ induced from $T^{(l)}$, namely  we set
\begin{equation*}
    \rho_{\mathrm{HS}}(g_1, g_2) := \| g_2 - g_1 \|_{\text{HS}}, \quad g_1, g_2 \in G^{(l)}.
\end{equation*}Another important metric which we will  use is the intrinsic \textit{Carnot-Carath\'eodory metric} (or just \textit{CC metric}). For a bounded variation (BV) path \( w \colon [0,1] \to \mathbb{R}^d \), its \textit{truncated signature} up to level \( l \) is defined as
\[ 
\mathbf{w}^{(l)}_{0,1} := \sum_{k=0}^l \int_{0<t_1<\cdots<t_k<1} dw_{t_1} \otimes \cdots \otimes dw_{t_k}\in G^{(l)}.
\]It is known that every element in $G^{(l)}$ is the truncated signature of some BV path. The CC-norm of $g\in G^{(l)}$, denoted as $\|g\|_{\mathrm{CC}}$, is the minimial length of all those BV paths whose truncated signature is $g$. Given $g_1,g_2\in G^{(l)}$, their CC-distance is defined by $$\rho_{\mathrm{CC}}(g_1,g_2):=\|g_1^{-1}\otimes g_2\|_{\mathrm{CC}}.$$For $u\in\mathfrak{g}^{(l)}$, we set $\|u\|_{\mathrm{CC}}:=\|\exp(u)\|_{\mathrm{CC}}.$

\subsection{Truncated logarithmic signatures of fractional Brownian motion} 

To define the truncated signature path of fBM, we need some definitions from rough path theory. Recall that the order-$l$ \textit{truncated signature path} of a BV path \( w \colon [0,T] \to \mathbb{R}^d \) is the functional defined by
\[ 
(s,t)\mapsto\mathbf{w}_{s,t} = \sum_{k=0}^l \int_{s<t_1<\cdots<t_k<t} dw_{t_1} \otimes \cdots \otimes dw_{t_k}\in G^{(l)}.
\]It is a \textit{multiplicative functional} in the sense that $\mathbf{w}_{s,u}\otimes\mathbf{w}_{u,t}=\mathbf{w}_{s,t}$ for any $s<u<t$. 

\begin{definition}
Let $\alpha \in (0,1)$. A multiplicative functional $\mathbf{z} : \{(s,t)\in [0,T]^2 : s\leqslant t\} \to T^{([1/\alpha])}$ is called a \textit{weakly geometric $\alpha$-Hölder rough path}
if it takes values in $G^{([1/\alpha])}$ and satisfies 
\begin{equation*}
\|\mathbf{z}\|_{\alpha;\mathrm{HS}} := 
\sum_{i=1}^{[1/\alpha]} 
\sup_{0 \leqslant s < t \leqslant 1} 
\frac{\|z^i_{s,t}\|_{\mathrm{HS}}}{|t-s|^{i\alpha}} < \infty.
\end{equation*}
\end{definition}

It is a basic result in rough path theory (Lyons' extension theorem) that given a weakly geometric rough path $\mathbf{z}$, its truncated signature paths of all orders are well-defined (see \cite{lyons1998differential}).

Now let $(X_t)_{0\leqslant t\leqslant T}$ be a $d$-dimensional fBM with Hurst parameter \( H \in (1/4,1) \) (viewed as a $\gamma$-H\"older weakly geometric rough path for $\gamma<H$; see Theorem \ref{thm:GLift}). Let $l\geqslant 1$ be given fixed. We use $\mathbf{X}^{(l)}_{s,t}$ to denote its truncated signature path of order $l$ and also set
\[
\mathbf{U}_{s,t} := \exp^{-1}(\mathbf{X}^{(l)}_{s,t}) \in \mathfrak{g}^{(l)}.
\]The process $\mathbf{U}_{s,t}$ is known as the \textit{truncated logarithmic signature path} of order $l$. It is well-known that $\mathbf{U}_{s,t}$ satisfies a hypoelliptic RDE on the Lie algebra (see \cite[Remark 7.43]{FV}) $\mathfrak{g}^{(l)}$ and thus admits a smooth density with respect to the Lebesgue measure $du$ on $\mathfrak{g}^{(l)}$ (equivalently, the Haar measure on $G^{(l)}$). 

The aim of this section is to establish the following upper bound for the conditional density of the logarithmic signature 
\( \mathbf{U}_{s,t} \).

\begin{proposition}\label{densiofrohst}
Let 
\[
\rho_{s,t}(u) := \mathbb{P}\left( \mathbf{U}_{s,t} \in du ~\big|~ \mathcal{F}^W_s \right), \quad u \in \mathfrak{g}^{(l)}(\mathbb{R}^d)
\]be the conditional density of \( \mathbf{U}_{s,t} \) given \( \mathcal{F}^W_s \).
There exist a constant \( C > 0 \), depending only on \( d \) and \( H \), and a family of positive random variables \( \{\Psi_{s,t}\}_{0 \leqslant s < t} \) such that
\[
\rho_{s,t}(u) \leqslant \frac{C}{(t-s)^{H\nu}} \exp\left( -\frac{\|u\|_{\mathrm{CC}}^2}{C(t-s)^{2H}} \right) \Psi_{s,t},
\]
for all \( 0 \leqslant s < t \leqslant T \), where the Malliavin-Sobolev norm of \( \Psi_{s,t} \) does not depend on \( s \) or \( t \), and \( \nu=\sum_{k=1}^l k \dim(\mathcal{L}_k) \) is the homogeneous dimension of \(\mathfrak{g}^{(l)}(\mathbb{R}^d) \).
\end{proposition}

The starting point for proving Proposition \ref{densiofrohst} is the following inequality: \begin{equation}\label{generalbound}
\mathbb{P}\left(\mathbf{U}_{s,t}\in d u\middle|\mathcal{F}^W_s \right)\leqslant C_n\,\mathbb{P}\left(\mathbf{U}_{s,t}\geqslant u\middle|\mathcal{F}^W_s \right)^{1/2}\left\|\gamma(\mathbf{U}_{s,t})^{-1}\right\|^n_{n,2^{n+2},s} \left\|\mathbf{D}_W\mathbf{U}_{s,t}\right\|^n_{n,2^{n+2},s}.
\end{equation}This is obtained  in the same way as \cite[Inequality (24)]{BNOT}. Here $n$ is arbitrary and $C_n$ is a positive constant depending on $n$. The event $\{\mathbf{U}_{s,t}\geqslant u\}$ is understood coordinate-wisely with respect to a given basis of $\mathfrak{g}^{(l)}$. Our task is now reduced to estimating each individual terms on the right hand side of (\ref{generalbound}).

\subsubsection{Estimate of $\mathbb{P}\left( \mathbf{U}_{s,t} \geqslant u \,\middle|\, \mathcal{F}_s^W \right)$}

Let \( \mathbf{X} \) denote the canonical lifting of \( X \) as a (random) geometric rough path. We first define its Besov norm which is served as a smooth upper bound for the H\"older norm. 
\begin{definition}\label{defBesovnorm}
Let \( \gamma < H \) and \( q \geqslant 1 \) be given fixed.  The $(\gamma,q)$-Besov norm of $\mathbf{X}$ over $[s,t]$ is defined to be
\[
\mathcal{N}_{\gamma,q}^{s,t}(\mathbf{X}) := \int_s^t \int_s^t \frac{D(\mathbf{X}_{u,v}, \mathbf{X}_{v})^q}{|v - u|^{\gamma q}} \, du \, dv,\quad 0\leqslant s<t\leqslant T,
\]
where 
\[
D(\mathbf{X}_{u,v}) := \sum_{k=1}^{\lfloor 1/H \rfloor} \left\| \int_{u \leqslant s_1 \leqslant \dots \leqslant s_k \leqslant v} dX_{s_1}\otimes\cdots\otimes dX_{s_k}\right\|^{1/k}_{\mathrm{HS}}.
\] 
\end{definition}
A benefit of working with these Besov norms is that they are smooth in the sense of Malliavin. In addition, standard Fernique-type estimate shows that there exists a constant \( \eta_0 > 0 \) such that
\begin{equation}\label{Besovnorminte}
\mathbb{E}\left[ \exp\left(\eta_0 \left( \mathcal{N}_{\gamma,2q}^{0,1}(\mathbf{X}) \right)^{1/q}\right) \right] < \infty.
\end{equation}
By using the Garsia-Rodemich-Rumsey inequality in Carnot groups, for any \( \varepsilon \in(0,H-\gamma) \) one has the following estimate
\begin{equation}\label{Besovnormboun}
\|\mathbf{X}\|_{\gamma\text{-H\"ol};[s,t]} \leqslant C \left( \mathcal{N}^{s,t}_{\gamma+\varepsilon,2q}(\mathbf{X}) \right)^{1/2q}, \qquad 0 \leqslant s \leqslant t \leqslant T.
\end{equation}See \cite{FV} for a discussion of these facts. 

The goal of  this part is to prove the following lemma. 
\begin{lemma}\label{logsig1}
There exist a constant $C > 0$ and a family of positive random variables $(\Psi^1_{s,t})_{0 \leqslant s < t}$, where the Malliavin-Sobolev norm of each $\Psi^1_{s,t}$ does not depend on $s$ or $t$, such that
\begin{equation}\label{eq:logsig1}
\mathbb{P}\left( \mathbf{U}_{s,t} \geqslant u \,\middle|\, \mathcal{F}_s^W \right) \leqslant \exp\left(-\frac{\|u\|_{\mathrm{CC}}^2}{C(t-s)^{2H}}\right) \Psi^1_{s,t}
\end{equation}
for all $u \in \mathfrak{g}^{(l)}$ and $0\leqslant s<t\leqslant T$. 
\end{lemma}

\begin{proof}
For any given $\varepsilon>0,s\geqslant0$, by using the decomposition \eqref{defXY}  the conditional probability $\mathbb{P}\left( \mathbf{U}_{s,s+\varepsilon} \geqslant u \,\middle|\, \mathcal{F}_s^W \right)$ can be viewed as a functional of $(Q_{s,s+\varepsilon u})_{0\leqslant u\leqslant 1}$. According to Lemma \ref{decom}, there exists a functional $F:C([0,1])\to\mathbb{R}$ such that 
$$
\mathbb{P}\left( \mathbf{U}_{s,s+\varepsilon} \geqslant u \,\middle|\, \mathcal{F}_s^W \right)=F(Q_{s,s+\varepsilon\cdot})\quad\text{and}\quad\mathbb{P}\left( \delta_{\varepsilon^H}\mathbf{U}_{0,1} \geqslant u \,\middle|\, \mathcal{F}_0^W \right)=F(\varepsilon^HQ_{0,\cdot})
$$
By applying Lemma \ref{equnorm}, one has 
$$
\mathbb{P}\left( \mathbf{U}_{s,t} \geqslant u \,\middle|\, \mathcal{F}_s^W \right)\overset{M}{=}\mathbb{P}\left( \delta_{(t-s)^H}\mathbf{U}_{0,1} \geqslant u \,\middle|\, \mathcal{F}_0^W \right)
$$for any $0\leqslant s<t$,
where the equivalence notion $\overset{M}{=}$ is introduced in Definition \ref{defequalm}.  

Now it suffices to show
\begin{equation}\label{toprovelemm6.5}
\mathbb{P}\left( \delta_{(t-s)^H}\mathbf{U}_{0,1} \geqslant u \,\middle|\, \mathcal{F}_0^W \right)\leqslant  \exp\left(-\frac{\|u\|_{\mathrm{CC}}^2}{C(t-s)^{2H}}\right) \Psi^1
\end{equation}
for certain random variable $\Psi^1$ with bounded Malliavin-Sobolev norms. We assume without loss of generality that each coordinate $u^\alpha\geqslant 0$. An argument justifying this reduction can be found in the proof of \cite[Theorem 3.13]{baudoin2014upper}. We also assume that $\|u\|_{\mathrm{CC}}>0$, for otherwise there is nothing to prove.

Let us define the homogeneous norm on $\mathfrak{g}^{(l)}$ by $\|u\|\overset{\Delta}{=}\text{max}_{\alpha}|u^\alpha|^{1/q_\alpha}$, where $q_\alpha=k$ if $u^\alpha$ is the coordinate of $u$ along a basis element in $\mathcal{L}_k$. For any random variable $\Xi:\Omega\to\mathfrak{g}^{(l)}$, one has
\[
\left\{\Xi \geqslant u\right\} = \left\{ \delta_{\|u\|_{\mathrm{CC}}^{-1}} \Xi \geqslant \delta_{\|u\|_{\mathrm{CC}}^{-1}} u \right\} \subseteq \left\{ \|\delta_{\|u\|_{\mathrm{CC}}^{-1}} \Xi \| \geqslant \| \delta_{\|u\|_{\mathrm{CC}}^{-1}} u \| \right\}.
\]
The above inclusion uses the assumption $u^\alpha\geqslant0$ for all $\alpha$. Since all homogeneous norms on $\mathfrak{g}^{(l)}$ are equivalent, there are positive constants $C_1,C_2$ such that 
\begin{equation*}
C_1||\cdot||_{\mathrm{CC}}\leqslant \|\cdot\| \leqslant C_2||\cdot||_{\mathrm{CC}}\quad\text{on~}\mathfrak{g}^{(l)}.
\end{equation*}
In particular, $\|\delta_{\|u\|_{\mathrm{CC}}^{-1}}u\|\geqslant C_1$. It follows that 
\begin{equation}\label{generalinclusion}
\begin{aligned}
\big\{\Xi\geqslant u\big\}\subseteq\left\{\|\delta_{||u||_{\mathrm{CC}}^{-1}}\Xi\|\geqslant C_1\right\}
=\left\{||u||_{\mathrm{CC}}^{-1}\|\Xi\|\geqslant C_1\right\}
\subseteq ~\Big\{C_2||\Xi||_{\mathrm{CC}}\geqslant C_1||u||_{\mathrm{CC}}\Big\}.
\end{aligned}
\end{equation}

Now let us apply (\ref{generalinclusion}) to  $\Xi:=\delta_{(t-s)^H}\mathbf{U}_{0,1}$. One finds that
\begin{equation*}
\begin{aligned}
\mathbb{P}\left( \delta_{(t-s)^H}\mathbf{U}_{0,1} \geqslant u \,\middle|\, \mathcal{F}_0^W \right)\leqslant&~\mathbb{P}\left( \|\delta_{(t-s)^H}\mathbf{U}_{0,1}\|_{\mathrm{CC}} \geqslant \frac{C_1}{C_2}\|u\|_{\mathrm{CC}} \,\middle|\, \mathcal{F}_0^W \right)\\
=&~\mathbb{P}\left( \|\mathbf{U}_{0,1}\|_{\mathrm{CC}} \geqslant \frac{C_1\|u\|_{\mathrm{CC}}}{C_2(t-s)^H} \,\middle|\, \mathcal{F}_0^W \right).
\end{aligned}
\end{equation*} 
Recall from  our convention that $\|\mathbf{U}_{0,1}\|_{\mathrm{CC}}=\|\mathbf{X}_{0,1}\|_{\mathrm{CC}}$. By the Chebyshev inequality and the Besov-type bound \eqref{Besovnormboun}, one concludes that 
\begin{align*}
&\mathbb{P}\left( \|\mathbf{U}_{0,1}\|_{\mathrm{CC}}\geqslant \frac{C_1 \|u\|_\mathrm{CC}}{C_2 (t-s)^H} \,\middle|\, \mathcal{F}_0^W \right)\\
&\leqslant C\exp\left( -\frac{\|u\|_{\mathrm{CC}}^2}{C(t-s)^{2H}} \right)
\mathbb{E}\left[ \exp\left( \eta \left( \mathcal{N}^{0,1}_{\gamma,2q}(\mathbf{B}) \right)^{1/q} \right) \bigg| \mathcal{F}_0^W \right],
\end{align*}
for some \( C > 0 \), \( q\geqslant 1 \), \( \eta \in (0, \eta_0] \) and $\gamma<H$. 
The estimate (\ref{toprovelemm6.5}) thus follows by taking
\begin{equation}\label{checksmoothness}
\Psi^1=\mathbb{E}\left[ \exp\left( \eta \left( \mathcal{N}^{0,1}_{\gamma,2q}(\mathbf{B}) \right)^{1/q} \right) \bigg| \mathcal{F}_0^W \right].
\end{equation}

To complete the poof of the lemma, observe that \( \Psi^1 \) is a functional of the process \( (Q_{0,u})_{0 \leqslant u \leqslant 1} \), i.e.
$\Psi^1 = F(Q_{0,\cdot})$ with some path functional \( F : C([0,1]) \to \mathbb{R} \). We then set 
\begin{equation}\label{defPsii}
\Psi^1_{s,t} := F\left( (t-s)^{-H} Q_{s,\,s+(t-s)\cdot} \right).
\end{equation}
According  to Lemma \ref{equnorm},  $\Psi^1$ and $\Psi^1_{s,t}$ have the same Malliavin-Sobolev norms for all $s,t\in[0,T]$. The desired estimate (\ref{eq:logsig1}) thus follows. 
\end{proof}

\subsubsection{Estimate of \(\| \mathbf{D}_W \mathbf{U}_{s,t} \|_{k,p,s} \)}

Our next step is to estimate \( ||\mathbf{D}_W\mathbf{U}_{s,t}||_{k,p,s} \) and the main result is stated as follows.

\begin{lemma}\label{logsig2}
For any integer $k \geqslant 1$ and any real number $p > 0$, there exist a family of random variables $(\Psi^2_{s,t})_{0\leqslant s < t}$ (depending on $k$ and $p$) whose Malliavin-Sobolev norms do not depend on $s$ or $t$, such that
\[
\|\mathbf{D}_W \mathbf{U}_{s,t}\|_{k,p,s} \leqslant (t-s)^H \Psi^2_{s,t}
\]for all $0 \leqslant s < t \leqslant T$.
\end{lemma}

Recall that \( \mathbf{U}_{s,t} = \exp^{-1}(\mathbf{X}^{(l)}_{s,t}) \). In particular, each component of \( \mathbf{U}_{s,t} \) is a polynomial of components of \( \mathbf{B}_{s,t} \). As a result, it suffices to estimate the components of \( \mathbf{B}_{s,t} \). Note that the $m$-th level  component of \( \mathbf{B}_{s,t} \) is
\[
I^m_{s,t}:=\int_{s< t_1<\cdots< t_m< t}d B_{t_1}\otimes\cdots\otimes d B_{t_m}
\]We first need the following lemma. 

\begin{lemma}
For any given $0\leqslant s<t\leqslant T$, one has 
\[
\|I^m_{s,t}\|_{k,p,s} \overset{M}{=} (t-s)^{mH}\|I^m_{0,1}\|_{k,p,0}.
\]
\end{lemma}
\begin{proof}
Clearly, it is equivalent to proving that  
\begin{equation}\label{ImHeHmIm}
\left( \mathbb{E} \left[ \|\mathbf{D}_W^k I^m_{s,s+\varepsilon}\|^p_{(\mathcal{H}_{W,s})^{\otimes k}} \middle| \mathcal{F}^W_s \right] \right)^{1/p} \overset{M}{=} \varepsilon^{Hm} \left( \mathbb{E} \left[ \|\mathbf{D}_W^k I^m_{0,1}\|^p_{(\mathcal{H}_{W,0})^{\otimes k}} \middle| \mathcal{F}^W_0 \right] \right)^{1/p}.
\end{equation}for any \( s \geqslant 0 \) and \( \varepsilon > 0 \). To this end, let us write
\[
\begin{aligned}
\Phi^1(\varepsilon^{-H} Q_{s,s+\varepsilon\cdot}) &= \varepsilon^{-Hmp} \mathbb{E}\left[ \left\|\mathbf{D}_W^k I^m_{s,s+\varepsilon}\right\|^p_{(\mathcal{H}_{W,s})^{\otimes k}} \middle| \mathcal{F}^W_s \right] \\
\Phi^2(Q_{0,\cdot}) &= \mathbb{E}\left[ \left\|\mathbf{D}_W^k I^m_{0,1}\right\|^p_{(\mathcal{H}_{W,0})^{\otimes k}} \middle| \mathcal{F}^W_0 \right]
\end{aligned}
\]with two path functionals  \( \Phi^1, \Phi^2 : C([0,1]) \to \mathbb{R} \), 
where \( Q_t,L_t \) are defined in \eqref{defXY}. We prove \eqref{ImHeHmIm} by showing
\begin{equation}\label{PsiMPsi}
\Phi^1(\varepsilon^{-H} Q_{s,s+\varepsilon\cdot}) \overset{M}{=} \Phi^2(Q_{0,\cdot}).
\end{equation}
Indeed, according to Lemma \ref{equnorm}, it suffices to show that \( \Phi^1 = \Phi^2 \). This is straightforward, since for any \( \eta \in C([0,1]) \),
\begin{equation}\label{secondinstar}
\begin{aligned}
\Phi^1(\eta_\cdot) &= \mathbb{E}\left[ \|\mathbf{D}_W^k (\varepsilon^{-Hm} I^m_{s,s+\varepsilon}) \|^p_{(\mathcal{H}_{W,s})^{\otimes k}} \middle| \varepsilon^{-H} Q_{s,s+\varepsilon\cdot} = \eta_\cdot \right] \\
&= \mathbb{E}\left[ \|\mathbf{D}_W^k I^m_{0,1} \|^p_{(\mathcal{H}_{W,0})^{\otimes k}} \middle| Q_{0,\cdot} = \eta_\cdot \right] = \Phi^2(\eta_\cdot).
\end{aligned}
\end{equation}
The second equality holds because conditional on \( \{Q_{0,\cdot} = \eta_\cdot\} \) and \( \{\varepsilon^{-H} Q_{s,s+\varepsilon\cdot} = \eta_\cdot\} \), there exists a functional \( F_\eta : C([0,1]) \to \mathbb{R} \) such that
\[
F_\eta(\varepsilon^{-H} L_{s,s+\varepsilon\cdot}) = \varepsilon^{-Hm} I^m_{s,s+\varepsilon}, \quad \text{and} \quad F_\eta(L_{0,\cdot}) = I^m_{0,1}.
\]
Since \( (Q_u)_{u \geqslant 0} \) and \( (L_u)_{u \geqslant 0} \) are independent, Lemma \ref{equnorm} implies that
\[
\varepsilon^{-Hm} I^m_{s,s+\varepsilon} = F_\eta(\varepsilon^{-H} L_{s,s+\varepsilon\cdot}) \overset{M}{=} F_\eta(L_{0,\cdot}) = I^m_{0,1}
\]
on the event \( \{Q_{0,\cdot} = \eta_\cdot\} \cap \{\varepsilon^{-H} Q_{s,s+\varepsilon\cdot} = \eta_\cdot\} \). This completes the proof of \eqref{secondinstar}.
\end{proof}
Now we give the proof of Lemma \ref{logsig2}.

\begin{proof}[Proof of Lemma \ref{logsig2}]
It suffices to prove that for any fixed \( k, p \), there exist a family of random variables \( (\Psi^2_{s,t})_{0\leqslant s < t} \) whose Malliavin-Sobolev norms do not depend on $s$ or $t$, such that
\[
\|I^m_{s,t}\|_{k,p,s} \leqslant (t-s)^{mH} \Psi^2_{s,t}.
\]

To this end, we adopt a similar idea as in \cite{I}. Let \( \tilde{W} \) be an independent copy of \( W \), and we use $\tilde{\mathbb{E}}$ and $\mathcal{F}_s^{\tilde{W}}$ to represent the expectation and filtration with respect to $\tilde{W}$. Define \( \tilde{B} \) by
\[
\tilde{B}^i_t = \int_{-\infty}^t k_H(t,u)\, d\tilde{W}^i_u, \quad 0 \leqslant t \leqslant T,\quad 1\leqslant i\leqslant d.
\]
We also set
\[
\Theta_1
=\sum_{i=1}^m \int_{0< t_1<\cdots< t_m< 1}
\cdots dB_{t_{i-1}}\otimes\underset{\substack{\text{$i$-th}}}{d\tilde{B}_{t_i}}\otimes dB_{t_{i+1}}\cdots 
\]
To simplify notation, we denote $I^m:=I^m_{0,1}$. For any \( h\in\mathcal{H}_B \), one has
\[
\left(\mathbf{D}^1_{B}I^m\right)(h)
=\sum_{i=1}^m \int_{0< t_1<\cdots< t_m< 1}
\cdots dB_{t_{i-1}}\otimes\underset{\substack{\text{$i$-th}}}{d\bar{h}_{t_i}}\otimes dB_{t_{i+1}}\cdots =\left(\mathbf{D}^1_{\tilde{B}}\Theta_1\right)(h),
\] 
where $\bar{h}$ is the corresponding element of $h$ in the intrinsic Cameron-Martin space $\bar{\mathcal{H}}_B$.
Therefore, \( \mathbf{D}^1_{B}I^m=\mathbf{D}^1_{\tilde{B}}\Theta_1\). It follows from Proposition \ref{DWLDB} that
\[
\mathbf{D}^1_{W}I^m=\mathscr{L}\left(\mathbf{D}^1_{B}I^m\right)=\mathscr{L}\left(\mathbf{D}^1_{\tilde{B}}\Theta_1\right)=\mathbf{D}^1_{\tilde{W}}\Theta_1
\]
holds almost surely as a functional from $\Omega\to L^2([0,1])$. 

Since all $\mathbb{D}^{k,2}_W$-norms are equivalent on a fixed inhomogeoneous Wiener chaos, one has
\[
||\mathbf{D}^1_{W}I^m||_{\mathcal{H}_{W,0}}=||\mathbf{D}^1_{\tilde{W}}\Theta_1||_{\mathcal{H}_{W,0}}\lesssim \left(\tilde{\mathbb{E}}\left[\Theta_{1}^2\middle|\mathcal{F}^{\tilde{W}}_0\right]\right)^{1/2}.
\]
In addition, by H\"older's inequality, one has
\[
\left(\mathbb{E}\left[||\mathbf{D}^1_{W}I^m||_{\mathcal{H}_{W,0}}^p\middle|\mathcal{F}^W_0 \right]\right)^{1/p}\leqslant \left(\mathbb{E}\tilde{\mathbb{E}}\left[\Theta_1^p\middle|\mathcal{F}^{W}_0\vee\mathcal{F}^{\tilde{W}}_0 \right]\right)^{1/p}
\]for all $p\geqslant 2$.
As in Inahama \cite{I}, one can similarly define \( \Theta_k \) for \( 1\leqslant k\leqslant m \) and show that
\[
\left(\mathbb{E}\left[||\mathbf{D}^k_{W}I^m||_{\mathcal{H}_{W,0}}^p\middle|\mathcal{F}^W_0 \right]\right)^{1/p}\leqslant \left(\mathbb{E}\tilde{\mathbb{E}}\left[\Theta_k^p\middle|\mathcal{F}^{W}_0\vee\mathcal{F}^{\tilde{W}}_0 \right]\right)^{1/p}
\]for all $p\geqslant 2$.
It follows that for any \( k\geqslant1 \), \( p\geqslant2 \), there exists a constant \( C_{k,p} \) such that
\begin{equation}\label{keyboundforIm}
\|I^m\|_{k,p,0}\leqslant C_{k,p}\sum_{r=1}^k\left(\mathbb{E}\tilde{\mathbb{E}}\left[\Theta_r^p\middle|\mathcal{F}^{W}_0\vee\mathcal{F}^{\tilde{W}}_0 \right]\right)^{1/p},
\end{equation}
where 
\begin{equation}\label{defthetar}
\Theta_r
=\sum_{1\leqslant i_1<\cdots < i_r\leqslant m} \int_{0< t_1<\cdots< t_m< 1}
\cdots 
\otimes \underset{\text{$i_1$-th}}{d\tilde{B}_{t_{i_1}}}
\cdots \otimes \underset{\text{$i_2$-th}}{d\tilde{B}_{t_{i_2}}}
\otimes
\underset{\raisebox{-1.0ex}{$\cdots$}}{\cdots}
\otimes \underset{\text{$i_r$-th}}{d\tilde{B}_{t_{i_r}}}
\cdots
\end{equation} and the omitted terms are filled with \( \otimes~dB_{t_i} \).

To eliminate the dependence on \( \tilde{W} \), we take  \( L^p \)-norm with respect to the conditional probability \( \mathbb{P}(~\cdot~| \mathcal{F}_0^{W}) \) on both sides of \eqref{keyboundforIm}. This gives 
\begin{equation}\label{DifferenIm}
\|I^m_{s,t}\|_{k,p,s} \overset{M}{=} (t-s)^{mH}\|I^m_{0,1}\|_{k,p,0} \leqslant C_{k,p} (t-s)^{mH} \sum_{r=1}^k \left( \mathbb{E} \left[ \Theta_r^p \middle| \mathcal{F}^W_0 \right] \right)^{1/p}.
\end{equation}
We thus define \( \Psi^2 := C_{k,p}\left( \mathbb{E} \left[ \Theta_r^p \middle| \mathcal{F}^W_0 \right] \right)^{1/p} \). According to the definition of \( \Theta_i \) in \eqref{defthetar}, each \( \Theta_i \) is in fact a component of the signature lifted from the \( 2d \)-dimensional fractional Brownian motion \( (B_u, \tilde{B}_u)_{u \geqslant 0} \). It is easily seen that \( \Psi^2 \) is smooth in the Malliavin sense. The random variable $\Psi^2_{s,t}$ is now defined in a similar way as in \eqref{defPsii} from $\Psi^2$. This completes the proof of the lemma.
\end{proof}

\subsubsection{Estimate of $\gamma(\mathbf{U}_{s,t})^{-1}$}

The last ingredient for estimating the right hand side of \eqref{generalbound} is the inverse Malliavin matrix of $\mathbf{U}_{s,t}$. The main result for this part is summarised in the lemma below. 

\begin{lemma}\label{logsig3}
 For any integer $k \geqslant 1$ and any real number $p > 0$, there exist a constant $\alpha>0$ and a family of random variables $(\Psi^3_{s,t})_{0\leqslant s < t}$ (depending on $k$ and $p$) such that
\[
\left\|\gamma(\mathbf{U}_{s,t})^{-1} \right\|_{k,p,s} \leqslant \frac{1}{(t-s)^\alpha} \Psi^3_{s,t}, \quad 0 \leqslant s < t \leqslant T,
\]
where the Malliavin-Sobolev norm of each $\Psi^3_{s,t}$ does not depend on $s$ or $t$.
\end{lemma}

The proof of Lemma~\ref{logsig3} is quite technical and therefore postponed to the appendix. 

\subsubsection{Proof of Proposition \ref{densiofrohst}}
We now combine Lemmas~\ref{logsig1}, \ref{logsig2}, and \ref{logsig3} to derive an upper bound for the conditional density of the logarithmic signature process \( \mathbf{U}_{s,t} \).
\begin{proof}[Proof of Proposition \ref{densiofrohst}]
Based on Lemma~\ref{equnorm}, one can apply a similar argument as in the derivation of~\eqref{ImHeHmIm} to show that
\[
\mathbb{P}\left( \mathbf{U}_{s,s+\varepsilon} \in du ~\big|~ \mathcal{F}^W_s \right) \overset{M}{=} \mathbb{P}\left( \delta_{\varepsilon^H} \mathbf{U}_{0,1} \in du ~\big|~ \mathcal{F}^W_0 \right).
\]for any \( \varepsilon > 0 \) and \( s \geqslant 0 \).
By the change of variable \( u = \delta_{\varepsilon^H} \hat{u} \), one obtains
\[
\mathbb{P}\left( \delta_{\varepsilon^H} \mathbf{U}_{0,1} \in du ~\big|~ \mathcal{F}^W_0 \right) = \varepsilon^{-H\nu} \mathbb{P}\left( \mathbf{U}_{0,1} \in d\hat{u} ~\big|~ \mathcal{F}^W_0 \right).
\]
It follows that
\[
\rho_{s,s+\varepsilon}(u) = \mathbb{P}\left( \mathbf{U}_{s,s+\varepsilon} \in du ~\big|~ \mathcal{F}^W_s \right) \overset{M}{=} \varepsilon^{-H\nu} \rho_{0,1}(\hat{u})
\]
for all \( u \in \mathfrak{g}^{(l)}(\mathbb{R}^d) \) with \( \hat{u} = \delta_{\varepsilon^{-H}} u \).

By applying Lemmas~\ref{logsig1}, \ref{logsig2} and~\ref{logsig3} to the inequality (\ref{generalbound}), one finds that
\[
\rho_{0,1}(\hat{u}) \leqslant C \exp\left( -\frac{\|\hat{u}\|_{\mathrm{CC}}^2}{C} \right) \prod_{i=1}^3 (1 + \Psi^i)^k
\]
for some constants \( C, k > 0 \). As a result, for any \( 0 \leqslant s < t \leqslant T \), one has
\[
\rho_{s,t}(u) \leqslant \frac{C}{(t-s)^{H\nu}} \exp\left( -\frac{\|u\|_{\mathrm{CC}}^2}{C(t-s)^{2H}} \right) \prod_{i=1}^3 \left( 1 + \Psi^i_{s,t} \right)^k,
\]
where we define
\[
\Psi := \prod_{i=1}^3 \left( 1 + \Psi^i \right)^k,
\]
and \( \Psi_{s,t} \) is defined in a similar way as in \eqref{defPsii} from $\Psi$. It is clear that the Malliavin-Sobolev norm of \( \Psi \) is finite, which completes the proof of the proposition.
\end{proof}

\subsection{Proof of Theorem \ref{GW2}}

We now proceed to develop the proof of the local upper bound for the joint density of $(Y_s,Y_t)$ (Theorem \ref{GW2}). More specifically, we consider the RDE \eqref{density1}. Given smooth vector fields $\{V_i\}_{i=1}^d$, we  set $V_\emptyset=1$ and for each word $I=(i_1,\cdots,i_r)$,
\[
V_{(I)} := V_{i_1}( \cdots V_{i_{r-1}}(V_{i_r})\cdots), \quad \text{and} \quad V_{I} := [V_{i_1}, \cdots [V_{i_{r-2}}, [V_{i_{r-1}}, V_{i_r}]]\cdots].
\]
Here we identify a vector field with a first order differential operator  (directional derivative) in the obvious way. The following definition plays an essential role in our analysis. 
\begin{definition}
For each $l\geqslant 1$, we define the \textit{Taylor approximation function} \( F_l : \mathfrak{g}^{(l)} \times \mathbb{R}^N \to \mathbb{R}^N \) of order \( l \) associated with the RDE \eqref{density1} by
\begin{equation}\label{TaylorFct}
F_l(u,x) := \sum_{\alpha \in \mathcal{A}_1(l)} V_{(\alpha)}(x) \cdot (\exp u)^\alpha, \quad (u,x) \in \mathfrak{g}^{(l)} \times \mathbb{R}^N,
\end{equation}
where the exponential function is defined on \( T^{(l)} \) by \eqref{(2.6)} and \( (\exp u)^\alpha \) is the coefficient of \( \exp u \) with respect to the tensor basis element \( e_{(\alpha)} \). We  say that \( u \in \mathfrak{g}^{(l)} \) \emph{joins \( x \) to \( y \) in the sense of Taylor approximation} if \( y = x + F_l(u,x) \).
\end{definition}

Let $l\geqslant \bar{l}$ be given fixed, where we recall that $\bar{l}$ is the hypoelliptic constant appearing in Hypothesis \ref{1.2}. According to  \cite[Corollary 4.9]{geng2022precise}, there exists a constant $r>0$ such that for each \( x \in \mathbb{R}^N \) and \( y \in \{x + F_l(u,x) : \|u\|_{\text{HS}} < r\} \), the "bridge space"
\begin{equation*}
M_{x,y} := \{ u \in \mathfrak{g}^{(l)} : \|u\|_{\text{HS}} < r \text{ and } x + F_l(u,x) = y \}
\end{equation*}
is a submanifold of \( \{ u \in \mathfrak{g}^{(l)} : \|u\|_{\text{HS}} < r \} \) with dimension \( \dim \mathfrak{g}^{(l)} - N \). In addition, since both \( \mathfrak{g}^{(l)} \) and \( \mathbb{R}^N \) are oriented Riemannian manifolds, the manifold \( M_{x,y} \) carries a natural orientation and hence a volume form which will be denoted as \( m_{x,y} \). The following result is a standard disintegration formula in Riemannian geometry.

\begin{proposition}
For any \( \varphi \in C_c^\infty(\{ u \in \mathfrak{g}^{(l)} : \|u\|_{\mathrm{HS}} < r \}) \), one has
\begin{equation*}
\int_{\mathfrak{g}^{(l)}} \varphi(u)\, du = \int_{\mathbb{R}^N} dy \int_{M_{x,y}} K(v,x)\varphi(v)\, m_{x,y}(dv),
\end{equation*}
where the kernel \( K \) is given by
\begin{equation}\label{defKvx}
K(v,x) = \left( \det(JF_l(v,x) \cdot JF_l(v,x)^*) \right)^{-1/2},
\end{equation}
$JF_l(v,x):\mathfrak{g}^{(l)}\to \mathbb{R}^N$ is the Jacobian of $F_l(\cdot,x)$ at $v\in\mathfrak{g}^{(l)}$ and we define \( m_{x,y} = 0 \) if \( M_{x,y} = \emptyset \).
\end{proposition}

The above disintegration formula immediately leads to the following formula for the localised conditional density of the Taylor approximation process \( Y_l(s,t,x) := x + F_l(\mathbf{U}_{s,t},x) \).

\begin{lemma}\label{lem:DenTay}
Let $r>0$ be given as before. Let $\eta:\mathfrak{g}^{(l)}\rightarrow[0,1]$ be a smooth bump function such that $\eta(u)=1$ when $\|u\|_{\mathrm{HS}}<r/2$ and $\eta(u)=0$ when $\|u\|_{\mathrm{HS}}>r$. Let \( \mathbb{P}_l^\eta(s,t,\cdot) \) be the conditional measure defined by
\[
\mathbb{P}_l^\eta(s,t,A):= \mathbb{E} \left[ \eta(\mathbf{U}_{s,t}) \mathbf{1}_{\left\{ Y_l(s,t,x) \in A\right\}}\middle|\mathcal{F}^W_s \right]\big|_{x=Y_s},\ \ \ A\in\mathcal{B}(\mathbb{R}^N).
\]Then \( \mathbb{P}_l^\eta(s,t,\cdot) \) is absolutely continuous with respect to the Lebesgue measure and its density is given by
\begin{equation}\label{densitypleta(s,t,y)}
p_l^\eta(s,t,y) = \int_{M_{x,y}} \eta(u) K(u,x) \rho_{s,t}(u) m_{x,y}(du)\big|_{x=Y_s},
\end{equation}
where \( \rho_{s,t} \) is the density of \( \mathbf{U}_{s,t} \) conditional on $\mathcal{F}^W_s$ and \( K \) is given by \eqref{defKvx}.
\end{lemma}

The following result is taken from  \cite[Lemma 3.23]{kusuoka1987applications}.
\begin{lemma}\label{lem:CoVBridge}
There are constants $0<r_1<r_2$ and  $\zeta\in(0,1]$ such that for any $x\in\mathbb{R}^N$ and $y\in\left\{x+F_l(u,x):~u\in \mathfrak{g}^{(l)},~\|u\|_{\mathrm{HS}}<r_1\right\}$, there is a diffeomorphism $S_{x,y}$ between $M_{x,y}\cap\left\{u\in\mathfrak{g}^{(l)}:||u||_{\mathrm{HS}}<r_1\right\}$ and $M_{x,x}\cap\left\{u\in\mathfrak{g}^{(l)}:||u||_{\mathrm{HS}}<r_2\right\}$, which satisfies
\begin{equation}\label{measure}
\zeta\cdot m_{x,x} \leqslant \left(K(\cdot,x)~m_{x,y}\right) \circ S_{x,y}^{-1} \leqslant \frac{1}{\zeta} \cdot m_{x,x} \   \text{on}\ \left\{S_{x,y}(u);~u\in M_{x,y}~\text{and}~||u||_{\mathrm{HS}}<r_1\right\}
\end{equation}
\textit{and}
\begin{equation}\label{distance}
\zeta \cdot (\|S_{x,y}(u)\|_{\mathrm{CC}} + d(x, y)) \leqslant \| u \|_{\mathrm{CC}} \leqslant \frac{1}{\zeta} \cdot (\|S_{x,y}(u)\|_{\mathrm{CC}} + d(x, y))
\end{equation}
for any $u \in M_{x,y}$ with $\|u\|_{\mathrm{HS}}<r_1$.
\end{lemma}

We are now in a position to prove Theorem \ref{GW2}. In fact, we will prove the following more general result which will be used in the next section.
\begin{theorem}\label{thm:JntDenUp}
Suppose that the vector fields $\{V_\alpha\}_{\alpha=1}^d$ satisfy Hypothesis \ref{h1} and Hypothesis \ref{1.2}. Let $p_{s,t}(x,y)$ be the joint density of the random vector $(Y_s,Y_t)$ defined by the RDE (\ref{density1}). For any $\delta>0$, there exist constants $C_1,C_2,\tau>0$ depending only on $\delta,H,N$ and the vector fields $\{V_\alpha\}_{\alpha=1}^d$ such that
\begin{equation}\label{upper}
p_{s,t}(x,y)\leqslant \frac{C_1}{|B_d(x,(t-s)^H)|_{\mathrm{Vol}}}\exp\left(-\frac{d(x,y)^2}{C_1(t-s)^{2H}}\right)+C_2|t-s|
\end{equation}
for all $(s,t,x,y)\in(\delta,T]\times \mathbb{R}^N\times \mathbb{R}^N$. Moreover, if one restricts to the cone region
$$
d(x,y)\leqslant (t-s)^H~~\text{and}~~(t-s)<\tau,
$$
then $C_2$ can be taken to be $0$.
\end{theorem}

\begin{proof}
According to Lemma \ref{lem:DenTay} and Lemma \ref{lem:CoVBridge}, one has
\begin{equation}\label{firstestipletasty}
\begin{aligned}
p_l^\eta(s,t,y)=&\int_{M_{x,y}\cap\left\{u\in\mathfrak{g}^{(l)}:~||u||_{\mathrm{HS}}<r_1\right\}}\eta(u)\rho_{s,t}(u)K(u,x)m_{x,y}(du)\Big|_{x=Y_s}\\
\leqslant&\int_{M_{x,y}\cap\left\{u\in\mathfrak{g}^{(l)}:~||u||_{\mathrm{HS}}<r_1\right\}}\rho_{s,t}(u)K(u,x)m_{x,y}(du)\Big|_{x=Y_s}\\
=&\int_{M_{x,x}\cap\left\{v\in\mathfrak{g}^{(l)}:~||v||_{\mathrm{HS}}<r_2\right\}\bigcap\text{Im}(S_{x,y})}\rho_{s,t}(S_{x,y}^{-1}v)m_{x,x}(dv)\Big|_{x=Y_s},
\end{aligned}
\end{equation}
Thanks to Proposition \ref{densiofrohst}, one has
\begin{equation*}
\rho_{s,t}(u)\leqslant\frac{C}{(t-s)^{Hv}}\text{exp}(-\frac{||u||^2_{\mathrm{CC}}}{C(t-s)^{2H}})\Psi_{s,t}
\end{equation*}
with some $\mathcal{F}^W_{s}$-measurable $\Psi_{s,t}$ which is smooth in the sense of Malliavin and whose Malliavin-Sobolev norms do not depend on $s,t$. By applying $(\ref{distance})$, one obtains that
\begin{equation}\label{secondestipleta}
 \rho_{s,t}(S_{x,y}^{-1}v)\leqslant\frac{C}{(t-s)^{H\nu}}e^{-\frac{\zeta^2 d(x,y)^2}{C(t-s)^{2H}}}e^{-\frac{\zeta^2||v||^2_{\mathrm{CC}}}{C(t-s)^{2H}}}\Psi_{s,t}
\end{equation}for all $v\in \mathrm{Im}(S_{x,y}).$
Combining \eqref{firstestipletasty} and \eqref{secondestipleta}, it follows that
\begin{equation}\label{secondestipleta+1}
p_l^\eta(s,t,y)\leqslant\frac{C}{(t-s)^{H\nu}}e^{-\frac{\zeta^2 d(x,y)^2}{C(t-s)^{2H}}}\Psi_{s,t} \int_{M_{x,x}\cap\left\{v\in\mathfrak{g}^{(l)}:~||v||_{\mathrm{HS}}<r_2\right\}}e^{-\frac{\zeta^2||v||^2_{\mathrm{CC}}}{C(t-s)^{2H}}}m_{x,x}(dv)\Big|_{x=Y_s}.
\end{equation}

The integral on the right hand side can be estimated as
\begin{equation*}
\begin{aligned}
&~\int_{M_{x,x}\cap\left\{v\in\mathfrak{g}^{(l)}:~||v||_{\mathrm{HS}}<r_2\right\}}e^{-\frac{\zeta^2||v||^2_{\mathrm{CC}}}{C(t-s)^{2H}}}m_{x,x}(dv)\Big|_{x=Y_s}\\
 &\leqslant\int_{M_{x,x}\cap\left\{\|v\|_{\mathrm{CC}}\leqslant 1\right\}}e^{-\frac{\zeta^2||v||^2_{\mathrm{CC}}}{C(t-s)^{2H}}}m_{x,x}(dv)\Big|_{x=Y_s}\\
 & \ \ \ \ \ +\int_{M_{x,x}\cap\left\{\|v\|_{\mathrm{CC}}> 1,~||v||_{\mathrm{HS}}<r_2\right\}}e^{-\frac{\zeta^2||v||^2_{\mathrm{CC}}}{C(t-s)^{2H}}}m_{x,x}(dv)\Big|_{x=Y_s}\\
&:=I_1(Y_s)~+~I_2(Y_s).
\end{aligned}
\end{equation*}
By Fubini's theorem, one has
\begin{equation*}
\begin{aligned}
I_1(x)=&~\frac{2\zeta^2}{C(t-s)^{2H}}\int_0^\infty u e^{-\frac{\zeta^2u^2}{C(t-s)^{2H}}}m_{x,x}\left(\{v\in M_{x,x}:||v||_{\mathrm{CC}}\leqslant u\wedge 1\}\right)du\\
=&~\frac{2\zeta^2}{C}\int_0^\infty r e^{-\frac{\zeta^2r^2}{C}}m_{x,x}(\{v\in M_{x,x}:||v||_{\mathrm{CC}}\leqslant (t-s)^Hr\wedge 1\})dr,
\end{aligned}
\end{equation*}
where we used the change of variable $u=(t-s)^Hr$ to obtain the second equality. According to \cite[Equation (3.33)]{kusuoka1987applications}, there exists $K\in(0,\infty)$ such that the estimate
\begin{align*}
&m_{x,x}(\{v\in M_{x,x}:||v||_{\mathrm{CC}}\leqslant (t-s)^Hr\wedge 1\})\\ &\leqslant K^{r\wedge(t-s)^{-H}+1}m_{x,x}(\{v\in M_{x,x}:||v||_{\mathrm{CC}}\leqslant (t-s)^H\})
\end{align*}
holds for all $r\in(0,\infty)$ and $x\in\mathbb{R}^N$. In addition, by \cite[Equation (3.32)]{kusuoka1987applications}, there exists $\delta\in(0,1)$ such that for all $(\eta,x)\in (0,1]\times \mathbb{R}^N$, one has
$$
\eta^{-\nu }m_{x,x}(\{v\in M_{x,x}:||v||_{CC}<\eta\})|B_d(x,\eta)|_{\mathrm{Vol}}\in [\delta,1/\delta].
$$
It follows that
\begin{equation}\label{secondestipleta+2}
\begin{aligned}
    I_1(x)\leqslant&~  C_{\zeta,K}~m_{x,x}(\{v\in M_{x,x}:||v||_{\mathrm{CC}}\leqslant (t-s)^H\})\\
    \leqslant&~C_{\zeta,K,\delta}~(t-s)^{H\nu}|B_d(x,(t-s)^H)|^{-1}_{\mathrm{Vol}}.
\end{aligned}
\end{equation}
On the other hand, note that $e^{-\frac{\zeta^2||v||^2_{\mathrm{CC}}}{C(t-s)^{2H}}}$ is upper bounded by $e^{-\frac{\zeta^2}{C(t-s)^{2H}}}$ on $\left\{\|v\|_{\mathrm{CC}}> 1\right\}$. As a result, $I_2(x)\leqslant e^{-\frac{\zeta^2}{C(t-s)^{2H}}}$ holds for some strictly positive constant $\zeta$, and in particular, $I_2(x)$ is upper bounded by $C_{\alpha}(t-s)^\alpha$ for all $\alpha>0$ and $s,t\in[0,T]$. By choosing a small $\alpha$, it follows that 
\[
I_2(x)\lesssim (t-s)^{H\nu}|B_d(x,(t-s)^H)|^{-1}_{\mathrm{Vol}},
\]
where we used the simple relation $r^\alpha\lesssim|B_d(x,r)|_{\mathrm{Vol}}$ for some small $\alpha$ when $r$ is small. 
By substituting \eqref{secondestipleta+2} into \eqref{secondestipleta+1}, one concludes that
\begin{equation}\label{eq:TayDenEst}
p_l^\eta(s,t,y)\leqslant\frac{C}{|B_d(x,(t-s)^H)|_{\mathrm{Vol}}}e^{-\frac{d(x,y)^2}{C(t-s)^{2H}}}\Psi_{s,t}\Big|_{x=Y_s}.
\end{equation}

On the other hand, by applying a similar argument as in \cite[Proposition 4.23]{geng2022precise} but with  $\mathbb{P}(~\cdot~)$ replaced by $\mathbb{P}(~\cdot~|\mathcal{F}^W_s)$, one can show that
\begin{equation}\label{eq:DenApprox}
\left|~\mathbb{P}(Y_t\in dy~|~\mathcal{F}^W_s)-p_l^\eta(s,t,y)\right|\leqslant |t-s|\times\Phi_{s,t},
\end{equation}
where $\Phi_{s,t}$ is smooth in the sense of Malliavin and its Malliavin-Sobolev norms do not depend on $s,t$. 
By combining the estimates \eqref{eq:TayDenEst} and \eqref{eq:DenApprox}, one obtains that
$$
\mathbb{E}[\delta_y(Y_t)|\mathcal{F}_s^W]=\mathbb{P}(Y_t\in dy~|~\mathcal{F}^W_s)\leqslant\frac{C}{|B_d(x,(t-s)^H)|_{\mathrm{Vol}}}e^{-\frac{d(x,y)^2}{C(t-s)^{2H}}}\Psi_{s,t}\Big|_{x=Y_s}+|t-s|\times\Phi_{s,t}.
$$
It follows that
\begin{equation}\label{localuppersmo}
\begin{aligned}
p_{s,t}(x,y)=& \mathbb{E}[\delta_{x}(Y_s)\mathbb{E}[\delta_{y}(Y_t)~|~\mathcal{F}^W_s]]\\
\leqslant &~\frac{C}{|B_d(x,(t-s)^H)|_{\mathrm{Vol}}}e^{-C\frac{d(x,y)^2}{(t-s)^{2H}}}\mathbb{E}[\delta_x(Y_s)\Psi_{s,t}]~+~(t-s)\mathbb{E}[\delta_x(Y_s)\Phi_{s,t}],
\end{aligned}
\end{equation}
where $\Psi_{s,t}$ and $\Phi_{s,t}$ are two random variables whose Malliavin-Sobolev norms do not depend on $s,t$.

Following the notation and the integration by parts formula in  \cite[Propositions 2.1.4, 2.1.5]{Nualart}, one has
\[
\mathbb{E}[\delta_x(Y_s)\Psi_{s,t}]=\mathbb{E}[\mathbf{1}_{\{Y_s>x\}}H_{(1,\cdots,N)}(Y_s,\Psi_{s,t})].
\]
For $1\leqslant p<q<\infty$, there exist constants $\beta, \gamma > 1$ and integers $n, m$ such that
\[
\| H_{\alpha}(F, G) \|_{p}
\leqslant c_{p, q} \, \| \det \gamma_{F}^{-1} \|_{\beta}^{m} \, \| DF \|_{k, \gamma}^{n} \, \| G \|_{k, q}.
\]
Because the Malliavin-Sobolev norms of $\Psi_{s,t}$ are independent of $s,t$, $\| \Psi_{s,t} \|_{k, q}$ only depends on $k,p$. In addition, for any $s\in[\delta,T]$, it is well-known that (see \cite{CHLT}) the following three quantities
$$
\mathbb{P}(Y_s>x),\quad\| \det \gamma_{Y_s}^{-1} \|_{\beta}^{m},\quad\|\mathbf{D}_W Y_s \|_{k, \gamma}^{n},\quad 
$$are uniformly bounded above by a constant depending on $\delta$.  The same argument applies to the other term $\mathbb{E}[\delta_x(Y_s)\Phi_{s,t}]$. Now one can integrate \eqref{localuppersmo} by parts to obtain that
$$
p_{s,t}(x,y)\leqslant\frac{C}{|B_d(x,(t-s)^H)|_{\mathrm{Vol}}}e^{-\frac{d(x,y)^2}{C(t-s)^{2H}}}~+~C(t-s).
$$If $d(x,y)\leqslant (t-s)^H$, the first term grows faster than the second one, hence yielding  the desired  upper estimate \eqref{upper123}.

Now the proof of the theorem is complete. 
\end{proof}

\section{Hitting probabilities and Newtonian-type capacities}\label{section5}

In this section, we develop the proof of Theorem \ref{GW3}, which establishes a two-sided estimate for hitting probabilities of a hypoelliptic SDE driven by fBM with Hurst parameter $H\in(1/4,1)$ in terms of Newtonian-type capacities.  A main novel point here is that the capacities used to control the hitting probabilities effectively are the ones induced by the \textit{sub-Riemannian control distance} $d$ (rather than the Euclidean metric as in the elliptic case). We recall from Definition \ref{def:Cap} that they are defined by \begin{equation}\label{alphadimenenrgy1}
\mathrm{Cap}_\alpha(A) := \left[\inf_{\mu \in \mathcal{P}(A)} \mathcal{E}_\alpha(\mu)\right]^{-1}
\end{equation}for $\alpha\in\mathbb{R}$, where \begin{equation}\label{alphadimenenrgy}
\mathcal{E}_\alpha(\mu) := \int\!\!\int K_\alpha(d(x,y))\,\mu(dx)\,\mu(dy)
\end{equation}and $K_\alpha(d(x,y))$ is the kernel defined by \eqref{Kalpha}.

The strategies for proving the upper and lower bounds in \eqref{gcupplowboun123} are quite different, and we present them in separate subsections. The two main theorems are stated in Theorem \ref{theoremlowerbound}
 and Theorem \ref{theoremupperbound} respectively.
 
As a preliminary tool, here we recall a standard ball-box volume estimate from sub-Riemannian geometry which will be used in both parts of the argument. A detailed proof can be found in \cite[Lemma 20.17]{subRiemanbook2019}. 

\begin{lemma}\label{lemmvolcompari}
Let \( \mathcal{V} = \{V_\alpha\}_{\alpha=1}^d \subseteq C_b^\infty(\mathbb{R}^N; \mathbb{R}^N) \) be a family of equiregular $C^\infty_b$-vector fields on \( \mathbb{R}^N \). For every compact set \( K \subseteq \mathbb{R}^N \), there exist positive constants \( \varepsilon_0\) and \( C_{1,\mathcal{V},K} < C_{2,\mathcal{V},K} \), such that
\begin{equation}\label{volcompari}
    C_{1,\mathcal{V},K} \varepsilon^Q \leqslant |B_d(q,\varepsilon)|_{\mathrm{Vol}} \leqslant C_{2,\mathcal{V},K} \varepsilon^Q
\end{equation}
for all \( q \in K \) and \( \varepsilon < \varepsilon_0 \), where 
$Q$ denotes the homogeneous dimension (see Definition \ref{defofhomodim}). 
\end{lemma}

\subsection{Lower bound for hitting probabilities}

We first demonstrate in a general way how certain estimates on joint densities will lead to a lower bound of hitting probablities in terms of capacities. After that, we show that our hypoelliptic SDE satisfy the presumed density estimates.

Let $(u_{t})_{0\leqslant t\leqslant T}$ be a stochastic process in $\mathbb{R}^N$ with continuous sample paths. Suppose that $(u_s,u_t)$ has a joint density function $p_{s,t}(\cdot,\cdot)$ with respect to the Lebesgue measure for all $s\neq t\in(0,T]$.
We make the following assumptions on the process $u_t$. 

\vspace{2mm}\noindent $\mathbf{(A1)}$ For any $0<a<b\leqslant T$ and $M>0$, there exists a  positive constant $C_1=C(a,b,T,M)$ such that 
$$
\int_a^b p_t(y)dt\geqslant C_1
$$holds for all $y\in B_d(o,M)$.

\vspace{2mm}\noindent $\mathbf{(A2)}$ Let $\alpha < Q$ be given fixed ($Q$ is the homogeneous dimension as before). For any $0<\delta< T$ and $M>0$, there exists $C_2=C(\delta,T,M,\alpha)>0$ such that
$$
\int_a^b \int_a^b p_{s,t}(x,y)\ ds\ dt \leqslant \ C_2\ K_\alpha(d(x,y))
$$holds  for all $\delta\leqslant a<b\leqslant T$ and $x,y\in B_d(o,M)$.
\begin{lemma}\label{lemmaafterA1A2}
Let \( (u_t)_{0 \leqslant t \leqslant T} \) be a stochastic process satisfying assumptions $\mathbf{(A1)}$ and $\mathbf{(A2)}$. Let \( 0 < a < b \leqslant T \) and \( M > 0 \) be given fixed. Then for any \( \eta \in (0, Q - \alpha) \) (if \( \alpha < 0 \), we further require \( \alpha + \eta < 0 \)), there exists a positive constant \( C = C(a, b, T, M, \alpha, \eta, V) \) such that
\begin{equation}\label{lowerboundforhit}
    \mathbb{P}\big(u([a, b]) \cap A \neq \emptyset\big) \geqslant C\, \mathrm{Cap}_{\alpha + \eta}(A)
\end{equation}holds for all compact sets \( A \subseteq B_d(o,M) \)
\end{lemma}
\begin{proof} Let $A\subseteq B_d(o,M)$ be a given compact subset such that \( \mathrm{Cap}_{\alpha+\eta}(A) > 0 \).

\vspace{2mm}\noindent \underline{\textit{Case 1}: $\alpha<0$.} 

\vspace{2mm} In this case, \( \alpha + \eta < 0 \) and by the definition \eqref{Kalpha} of the kernel $K_\alpha$, one has \( \mathrm{Cap}_{\alpha+\eta}(A) = 1 \). Therefore, one only needs to prove that
\begin{equation}\label{hittingprolargerthanc}
\mathbb{P}\big(u([a, b]) \cap A \neq \emptyset\big) \geqslant C
\end{equation}with some positive constant \( C \) depending only on \( a, b, M, \alpha, \eta \). In what follows, we always use $C$ to denote such a constant which may differ from line to line. 

Let us consider the following random variable 
\[
J_\varepsilon(A) := \frac{1}{|A_\varepsilon|_{\mathrm{Vol}}} \int_a^b \mathbf{1}_{A_\varepsilon}(u_t)\,dt,\ \ \ \varepsilon\in(0,1),
\]
where \( A_\varepsilon := \{ x \in \mathbb{R}^N : d(x, A) \leqslant \varepsilon \} \). We claim that the estimate
\[
\mathbb{P}(J_\varepsilon(A) > 0) \geqslant C
\]
holds uniformly for all \( \varepsilon \in (0,1) \) and compact sets \( A \subseteq B_d(o,M) \), where the constant \( C > 0 \) does not depend on \( \varepsilon \) or the set \( A \). Indeed, by Assumption $\mathbf{(A1)}$ one has 
\[
\mathbb{E}[J_\varepsilon(A)] 
= \frac{1}{|A_\varepsilon|_{\mathrm{Vol}}} \int_a^b \int_{A_\varepsilon} p_t(y)\,dy\,dt 
\geqslant C.
\]
On the other hand, Assumption $\mathbf{(A2)}$ (also noting that $\alpha<0$) shows that
\[
\begin{aligned}
\mathbb{E}[J_\varepsilon(A)^2] 
&= \frac{1}{|A_\varepsilon|_{\mathrm{Vol}}^2} \int_a^b \int_a^b \int_{A_\varepsilon} \int_{A_\varepsilon} p_{s,t}(x,y)\,dx\,dy\,ds\,dt \\
&\leqslant \frac{C_2}{|A_\varepsilon|_{\mathrm{Vol}}^2} \int_{A_\varepsilon} \int_{A_\varepsilon} K_\alpha(d(x,y))\,dx\,dy \leqslant C.
\end{aligned}
\]By the Cauchy-Schwarz inequality, one obtains that
\[
\mathbb{P}(J_\varepsilon(A) > 0) = \mathbb{E}[\mathbf{1}_{\{J_\varepsilon(A) > 0\}}^2] 
\geqslant \frac{\mathbb{E}[J_\varepsilon(A)]^2}{\mathbb{E}[J_\varepsilon(A)^2]} 
\geqslant C.
\]
Observe that 
\[
\{ \omega : J_\varepsilon(A) > 0 \} \subseteq \{ \omega : u([a,b]) \cap A_\varepsilon \neq \emptyset \}.
\]
Therefore, one has
\[
\mathbb{P}(u([a,b]) \cap A_\varepsilon \neq \emptyset) 
\geqslant \mathbb{P}(J_\varepsilon(A) > 0) \geqslant C.
\]
The inequality \eqref{hittingprolargerthanc} now follows by taking $\varepsilon\downarrow 0$.

\vspace{2mm}\noindent \underline{\textit{Case 2: $0<\alpha<Q$}.}

\vspace{2mm} For any \( \varepsilon \in (0,1) \) and \( \mu \in \mathcal{P}(A) \), we define the random variable
\[
J_\varepsilon(\mu) := \int_{\mathbb{R}^N} \mu(dz) \int_a^b \frac{\mathbf{1}_{B_d(z,\varepsilon)}(u_t)}{|B_d(z,\varepsilon)|_{\mathrm{Vol}}} \,dt 
= \int_a^b \mu_\varepsilon(u_t) \,dt,
\]
where the mollified density \( \mu_\varepsilon \colon \mathbb{R}^N \to \mathbb{R} \) is given by
\[
\mu_\varepsilon(x) := \int_{\mathbb{R}^N} \frac{\mathbf{1}_{B_d(z,\varepsilon)}(x)}{|B_d(z,\varepsilon)|_{\mathrm{Vol}}} \mu(dz).
\]
It follows from \eqref{volcompari} and assumption $\mathbf{(A1)}$ that there exists a positive constant \( C = C(a, b, M, \mathcal{V}) \) such that for all \( \varepsilon \in (0,1) \),
\begin{equation}\label{case2lowerboun}
\mathbb{E}[J_\varepsilon(\mu)] 
= \int_{\mathbb{R}^N} \frac{\mu(dz)}{|B_d(z,\varepsilon)|_{\mathrm{Vol}}} \int_{B_d(z,\varepsilon)} \left( \int_a^b p_t(y)\,dt \right) dy 
\geqslant C.
\end{equation}
Moreover, one has
\begin{equation}\label{case2capuppboun}
\begin{aligned}
\mathbb{E}[J_\varepsilon(\mu)^2] 
&= \int_{\mathbb{R}^N} \int_{\mathbb{R}^N} \mu_\varepsilon(x)\, \mu_\varepsilon(y) \left( \int_a^b\int_a^b p_{s,t}(x,y)\,ds\,dt \right) dx\,dy \\
&\leqslant C_2 \int_{\mathbb{R}^N} \int_{\mathbb{R}^N} K_\alpha(d(x,y))\, \mu_\varepsilon(dx)\, \mu_\varepsilon(dy) = C_2~\mathcal{E}_\alpha(\mu_\varepsilon),
\end{aligned}
\end{equation}
where the inequality follows from Assumption $\mathbf{(A2)}$ and the last equality is due to the definition of the energy functional (by abuse of notation we write $\mu_\varepsilon$ for both the measure and the density).

Recall that \( \mathrm{Cap}_{\alpha+\eta}(A) > 0 \). The main observation is that there exists a probability measure \( \bar{\mu} \in \mathcal{P}(A) \) such that
\begin{equation}\label{case2capuppboun1}
\mathcal{E}_\alpha(\bar{\mu}_\varepsilon) \leqslant \frac{C_{M,\alpha,\mathcal{V},\eta}}{\mathrm{Cap}_{\alpha+\eta}(A)}
\end{equation}holds for all $\varepsilon\in(0,1)$, 
where $C_{M,\alpha,\mathcal{V},\eta}$ is a positive constant independent of $\varepsilon$ and the set $A$. We postpone the proof of this technical fact in Lemma \ref{lemmcase2capuppboun} below. Presuming the correctness of \eqref{case2capuppboun1}, one can deduce by using the Cauchy-Schwarz inequality along with the estimates (\ref{case2lowerboun}, \ref{case2capuppboun}, \ref{case2capuppboun1}) that
\begin{equation}\label{jmuepsi}
\mathbb{P}(u([a,b]) \cap A_\varepsilon \neq \emptyset)\geqslant \mathbb{P}(J_\varepsilon(\bar{\mu}) > 0) 
= \mathbb{E}[\mathbf{1}_{\{J_\varepsilon(\bar{\mu}) > 0\}}^2] 
\geqslant \frac{\mathbb{E}[J_\varepsilon(\bar{\mu})]^2}{\mathbb{E}[J_\varepsilon(\bar{\mu})^2]} 
\geqslant C\, \mathrm{Cap}_{\alpha+\eta}(A)
\end{equation}
for all $\varepsilon\in(0,1)$ with some positive constant \( C = C(a, b, M, n, \alpha, \mathcal{V}, \eta) \). The desired estimate \eqref{lowerboundforhit} follows by taking $\varepsilon\downarrow 0$.

\vspace{2mm}\noindent \underline{\textit{Case 3: $\alpha=0$}.}

\vspace{2mm} By definition, one has \( K_0(d(x,y)) = \log(N_0 / d(x,y)) \). It is straightforward to see that there exists a positive constant \( C = C_{M,\eta,N_0} \) such that
\[
K_0(d(x,y)) \leqslant C\, d(x,y)^{-\eta/2} = C\, K_{\eta/2}(d(x,y))
\]
for all \( x, y \in B_d(o, M) \). The desired estimate follows from the same argument as in Cast II. 
\end{proof}

Now we come back to prove the existence of $\bar{\mu}$ used in the above proof.

\begin{lemma}\label{lemmcase2capuppboun}
Fix a real number $M>0$ and a compact set $A\subseteq B_d(o,M)$. Suppose that $\alpha\geqslant0,\eta>0$, $\alpha+\eta<Q$ and \( \mathrm{Cap}_{\alpha+\eta}(A) > 0 \). Then there exists a probability measure \( \bar{\mu} \in \mathcal{P}(A) \) such that
\begin{equation}\label{case2capuppbounapp}
\mathcal{E}_\alpha(\bar{\mu}_\varepsilon) \leqslant \frac{C_{M,\alpha,\mathcal{V},\eta}}{\mathrm{Cap}_{\alpha+\eta}(A)}
\end{equation}holds for all $\varepsilon\in(0,1),$
where $C_{M,\alpha,\mathcal{V},\eta}$ is a positive constant independent of $\varepsilon$ and the set $A$.
\end{lemma}
\begin{proof}
Without loss of generality, we only consider the case when $M\geqslant2$ and $A\subseteq B_d(o,M-1)$. According to the definition of the capacity $\mathrm{Cap}_{\alpha+\eta}(A)$, there exists a probability measure $\nu \in \mathcal{P}(A)$ such that
\begin{equation}\label{eq:ExistMuPf}
0 < \mathcal{E}_{\alpha+\eta}(\nu) \leqslant 2 \inf_{\mu \in \mathcal{P}(A)} \mathcal{E}_{\alpha+\eta}(\mu) = \frac{2}{\mathrm{Cap}_{\alpha+\eta}(A)}.
\end{equation}
Let us write
\[
F := \mathcal{E}_{\alpha+\eta}(\nu) = \iint d(x, y)^{-(\alpha+\eta)} \, \nu(dy) \nu(dx).
\]

\textit{Step 1.} In this step, we provide the following sufficient condition for \eqref{case2capuppbounapp} to hold. More precisely, suppose that there exists a probability measure $\bar{\mu} \in \mathcal{P}(A)$ such that
\begin{equation}\label{estiformuepsi}
\bar{\mu}_\varepsilon(B_d(x, r)) \leqslant C_{1,\mathcal{V}}^{-1} \cdot C_{2,\mathcal{V}} \cdot 6F \cdot 4^{\alpha+\eta} r^{\alpha+\eta} 
\quad \text{for all } x \in \mathbb{R}^N,~ r > 0,~ \varepsilon > 0
\end{equation}
and
\begin{equation}\label{uppbounmubar}
\bar{\mu}_\varepsilon(\mathbb{R}^N) \leqslant C_{1,\mathcal{V}}^{-1} \cdot C_{2,\mathcal{V}} 
\quad \text{for all } \varepsilon \in (0, 1),
\end{equation}
where $C_{1,\mathcal{V}}$, $C_{2,\mathcal{V}}$ are given by the volume comparison \eqref{volcompari}. Then we claim that \eqref{case2capuppbounapp} follows.

Indeed, one first notes that 
\begin{equation*}
\begin{aligned}
\int_{\mathbb{R}^N} d(x, y)^{-\alpha} \bar{\mu}_\varepsilon(dy) 
&= \int_0^\infty \bar{\mu}_\varepsilon\left(\left\{ y : d(x, y)^{-\alpha} \geqslant u \right\} \right) du \\
&= \int_0^\infty \bar{\mu}_\varepsilon\left( B_d(x, u^{-1/\alpha}) \right) du \\
&= \alpha \int_0^\infty r^{-\alpha - 1} \bar{\mu}_\varepsilon\left( B_d(x, r) \right) dr \\
&= \alpha \int_0^{2M} r^{-\alpha - 1} \bar{\mu}_\varepsilon\left( B_d(x, r) \right) dr 
  + \alpha \int_{2M}^\infty r^{-\alpha - 1} \bar{\mu}_\varepsilon\left( B_d(x, r) \right) dr \\
&=: I_1(x) + I_2(x),
\end{aligned}
\end{equation*}for all $x\in\mathbb{R}^n$,
where the first equality follows from Fubini's theorem and the third equality is due to the change of variables \( r = u^{-1/\alpha} \). By the assumption \eqref{estiformuepsi}, one has
\[
\frac{I_1(x)}{C_{1,V}^{-1} \cdot C_{2,V} \cdot 6F \cdot 4^{\alpha+\eta} \cdot \alpha} 
\leqslant \int_0^{2M} r^{\eta - 1} dr = \frac{2^\eta M^\eta}{\eta}
\]for all $x\in\mathbb{R^N}$. If one further assumes \( x \in A_\varepsilon \), then
\[
\frac{I_2(x)}{C_{1,V}^{-1} \cdot C_{2,V} \cdot 6F \cdot 4^{\alpha+\eta} \cdot \alpha} 
\leqslant \int_{2M}^\infty r^{-\alpha - 1} (2 M)^{\alpha + \eta} dr = \frac{2^\eta M^\eta}{\alpha}.
\]
To reach the above inequality, we  used the fact that \( A_\varepsilon \subseteq B_d(x, 2M) \) and thus
\[
\bar{\mu}_\varepsilon(B_d(x, r)) = \bar{\mu}_\varepsilon(B_d(x, 2M))
\]for all $r \geqslant 2M$. It follows from \eqref{uppbounmubar} and \eqref{eq:ExistMuPf} that
\begin{equation*}
\begin{aligned}
\mathcal{E}_\alpha(\bar{\mu}_\varepsilon)=&\int_{A_\varepsilon}\left(I_1(x)+I_2(x)\right)\bar{\mu}_\varepsilon(dx)\\
\leqslant &~6\cdot C_{1,V}^{-2} \cdot C_{2,V}^2\cdot 2^{2\alpha+3\eta}\cdot(\alpha\eta^{-1}+1)\cdot M^\eta\cdot F\\
\leqslant &~\frac{12\cdot C_{1,V}^{-2} \cdot C_{2,V}^2\cdot 2^{2\alpha+3\eta}\cdot(\alpha\eta^{-1}+1)\cdot M^\eta}{\mathrm{Cap}_{\alpha+\eta}(A)}
\end{aligned}
\end{equation*}for all $\varepsilon\in(0,1)$. This proves the desired estimate \eqref{case2capuppbounapp}.

\textit{Step 2.} We will verify \eqref{uppbounmubar} and \eqref{estiformuepsi} in this step. Define the set
\[
B := \left\{ x \in A : \int_{\mathbb{R}^N} d(x, y)^{-(\alpha+\eta)} \, \nu(dy) \leqslant 3F \right\}\subseteq A.
\]
We claim that $\nu(B) \geqslant 1/2$. Assuming the contrary, one has 
\begin{align*}
F &= \iint d(x, y)^{-(\alpha+\eta)} \, \nu(dy) \nu(dx) \\
&\geqslant \int_{B^c} \nu(dx) \int_{\mathbb{R}^N} d(x, y)^{-(\alpha+\eta)} \, \nu(dy) \\
&\geqslant \int_{B^c} 3F \, \nu(dx) \geqslant 3F \cdot \frac{1}{2} = \frac{3}{2} F,
\end{align*}
which is a contradiction. Therefore, one must have $\nu(B) \geqslant 1/2$. 

Let $\mu := \nu(\cdot \cap B)$ be the restriction of $\nu$ to the set $B$. For all $x \in B$ and $r > 0$, one has
\begin{equation*} 
r^{-(\alpha+\eta)} \mu(B_d(x, r)) 
\leqslant \int_{B_d(x, r)} d(x, y)^{-(\alpha+\eta)} \, \nu(dy) 
\leqslant 3F.
\end{equation*}
Now we define
\[
\bar{\mu}(\cdot) := \frac{\nu(\cdot \cap B)}{\nu(B)} = \frac{\mu(\cdot)}{\nu(B)}.
\]
Since $\nu(B) \geqslant 1/2$, it follows that $\bar{\mu} \in \mathcal{P}(B)$ and
\begin{equation}\label{eq:bar_mu_bound}
\bar{\mu}(B_d(x, r)) = \nu(B)^{-1} \mu(B_d(x, r)) 
\leqslant 2 \cdot 3F \cdot r^{\alpha+\eta} = 6F r^{\alpha+\eta}
\end{equation}for all $x \in B,~ r > 0$.
In fact, the estimate \eqref{eq:bar_mu_bound} holds for all $x \in \mathbb{R}^N$ and $r > 0$ which is seen as follows.

\begin{itemize}
    \item If $B_d(x, r) \cap B = \emptyset$, then $\bar{\mu}(B_d(x, r)) = 0$ trivially.
    \item Otherwise, there exists some $z \in B_d(x, r) \cap B$. By applying \eqref{eq:bar_mu_bound} at $z$ with radius $2r$ and noting that
    $
    B_d(x, r) \subseteq B_d(z, 2r)$, 
    one has
    \begin{equation} \label{targetedmeasureonB}
    \bar{\mu}(B_d(x, r)) \leqslant \bar{\mu}(B_d(z, 2r)) \leqslant 6F \cdot 2^{\alpha + \eta} r^{\alpha + \eta}.
    \end{equation}
\end{itemize}
Thus, the estimate \eqref{targetedmeasureonB} holds uniformly for all $x \in \mathbb{R}^N$ and $r > 0$. Now we proceed to prove the estimates \eqref{estiformuepsi} and \eqref{uppbounmubar} for the mollified measure $\bar{\mu}_\varepsilon$. 

\begin{itemize}
\item \underline{\textit{Case 1: $0 < r < \varepsilon$.}} In this case, one has
\begin{align}
\bar{\mu}_\varepsilon(B_d(x, r)) 
&= \int_{B_d(x, r)} dy \int_{\mathbb{R}^N} \frac{\mathbf{1}_{B_d(z, \varepsilon)}(y)}{|B_d(z, \varepsilon)|_{\mathrm{Vol}}} \, \bar{\mu}(dz) \nonumber  \\
&\leqslant C_{2,V} r^Q \cdot C_{1,V}^{-1} \varepsilon^{-Q} \cdot 6F \cdot 2^{\alpha+\eta} \varepsilon^{\alpha+\eta} \nonumber \\
&= C_{2,V} \cdot C_{1,V}^{-1} \cdot 6F \cdot 2^{\alpha+\eta} \cdot r^{\alpha+\eta}, \label{eq:small_r_estimate}
\end{align}
where we used the volume comparison \eqref{volcompari} and the fact that $\alpha + \eta < Q$.

\item \underline{\textit{Case 2: $r \geqslant \varepsilon$.}} By using Fubini's theorem and the triangle inequality, one has
\begin{align}
\bar{\mu}_\varepsilon(B_d(x, r)) 
&= \int_{B_d(x, r)} dy \int_{\mathbb{R}^N} \frac{\mathbf{1}_{B_d(z, \varepsilon)}(y)}{|B_d(z, \varepsilon)|_{\mathrm{Vol}}} \, \bar{\mu}(dz) \nonumber \\
&= \int_{B_d(x, 2r)} \frac{\bar{\mu}(dz)}{|B_d(z, \varepsilon)|_{\mathrm{Vol}}} \int_{B_d(x, r)} \mathbf{1}_{B_d(z, \varepsilon)}(y) \, dy \nonumber \\
&\leqslant \bar{\mu}(B_d(x, 2r)) \cdot C_{1,\mathcal{V}}^{-1} \varepsilon^{-Q} \cdot C_{2,\mathcal{V}} \varepsilon^Q \nonumber \\
&\leqslant C_{1,\mathcal{V}}^{-1} \cdot C_{2,\mathcal{V}} \cdot 6F \cdot 4^{\alpha+\eta} r^{\alpha+\eta}, \label{eq:large_r_estimate}
\end{align}
where we used \eqref{targetedmeasureonB} and the volume comparison in \eqref{volcompari} to reach the third line.
\end{itemize}

\noindent
Since $\bar{\mu}$ is a probability measure, one also has from the third line of \eqref{eq:large_r_estimate} that
\[
\bar{\mu}_\varepsilon(\mathbb{R}^N) \leqslant C_{1,\mathcal{V}}^{-1} \cdot C_{2,\mathcal{V}}.
\]
Combining \eqref{eq:small_r_estimate} and \eqref{eq:large_r_estimate}, one obtains the estimate
\begin{equation}\label{eq:smalllarge_r_estimate}
\bar{\mu}_\varepsilon(B_d(x, r)) \leqslant C_{1,\mathcal{V}}^{-1} \cdot C_{2,\mathcal{V}} \cdot 6F \cdot 4^{\alpha+\eta} r^{\alpha+\eta}
\end{equation}for all $x\in\mathbb{R}^N$ and $r>0$, which completes the proof.
\end{proof}

We now specialise in our SDE context and prove the following main result of this section. 

\begin{theorem}\label{theoremlowerbound}
Assume that the family of vector fields \( \mathcal{V}=\{V_\alpha\}_{\alpha=1}^d \) is equiregular on \( \mathbb{R}^N \) and satisfies Hypotheses~\ref{h1} and~\ref{1.2}. Let \( Y_t \) be the solution to the RDE~\eqref{density1} where \( \mathbf{B} \) is the canonical lift of a fractional Brownian motion \( (B_t)_{0 \leqslant t \leqslant T} \) with Hurst parameter $H\in(1/4,1)$. Fix \( 0 < a < b \) and \( M > 0 \). Then for any \( \eta \in \left( 0, 1/H \right) \), there exists a constant \( C = C(a, b, M, Q, H, \mathcal{V}, \eta) > 0 \) such that the estimate
\[
\mathbb{P}\left( Y([a, b]) \cap A \neq \emptyset \right) \geqslant C  \mathrm{Cap}_{Q - 1/H + \eta}(A)
\]holds for all compact sets \( A \subseteq B_d(o, M) \).
Here \( Q \) denotes the homogeneous dimension.
\end{theorem}
\begin{proof}
According to Lemma~\ref{lemmaafterA1A2}, it remains to verify conditions $\mathbf{(A1)}$ and $\mathbf{(A2)}$. The density function \( p_t(\cdot) \) of \( Y_t \) is known to be continuous (see \cite[Theorem 3.5]{CHLT}) and everywhere strictly positive (see \cite[Theorem 1.5]{geng2022precise}). In particular, one has
\[
\int_a^b p_t(y)\,dt \geqslant \inf_{y \in B_d(o, M)} \int_a^b p_t(y)\,dt =: C(a, b, M) > 0.
\]
This verifies condition $\mathbf{(A1)}$.

To establish the upper bound in condition $\mathbf{(A2)}$, recall from the joint density upper estimate (\ref{upper}) that there exist constants \( c_1, c_2, c_3 > 0 \) such that
\begin{equation*}
p_{s,t}(x,y) \leqslant \frac{c_1}{|B_d(x,(t-s)^H)|_{\mathrm{Vol}}} \exp\left(-c_2 \frac{d(x,y)^2}{(t-s)^{2H}}\right) + c_3 |t-s|
\end{equation*}holds for all \( s, t \in [a, b] \) and \( x, y \in \mathbb{R}^N \).
If one further restricts to \( x \in B_d(o, M) \),  by Lemma~\ref{lemmvolcompari} the volume of the ball \( B_d(x,(t-s)^H) \) is comparable to $|t-s|^{HQ}$. Namely, there exist constants \( c_4, c_5 > 0 \) depending only on \( M \) and the vector fields \( \mathcal{V} \), such that
\begin{equation}\label{controlballvolume}
c_4\, (t-s)^{QH} \leqslant |B_d(x,(t-s)^H)| \leqslant c_5\, (t-s)^{QH}.
\end{equation}
It follows that 
\[
p_{s,t}(x,y) \leqslant \frac{C_{1,M,\mathcal{V}}}{(t-s)^{QH}} \exp\left( -C_{2,M,\mathcal{V}} \frac{d(x,y)^2}{(t-s)^{2H}} \right) + C_{3,M,\mathcal{V}} |t-s|
\] holds uniformly for all \( s, t \in [a, b] \) and \( x, y \in B_d(o, M) \), where  \( C_{i,M,\mathcal{V}}\) ($i=1,2,3$) are positive constants depending on $a,b,M,\mathcal{V}$. 

Our aim is to show that
\[
\int_a^b \int_a^b p_{s,t}(x,y)\, ds\, dt \leqslant C\, K_{Q - 1/H}(d(x,y)),
\]
with some positive constant \( C = C(a, b, M, Q, H,\mathcal{V})\).

First of all, one observes that \begin{equation}\label{verity1}
\int_a^b \int_a^b |t - s|\, ds\, dt \leqslant C\, K_{Q - 1/H}(d(x,y)).
\end{equation}This is a trivial consequence of the fact that $K_{Q-1/H}(d(x,y))$ is bounded away from zero for all $x,y\in B_d(o,M)$ (in the case $Q=1/H$, one needs to choose a suitaly large $N_0$ depending on $M$; recall (\ref{Kalpha}) for the definition of the kernel $K$).
Now it suffices to establish the following estimate:
\begin{equation}\label{verity2}
\int_a^b \int_a^b \frac{1}{(t - s)^{QH}} \exp\left( -C_{2,M,\mathcal{V}} \frac{d(x,y)^2}{(t - s)^{2H}} \right) ds\, dt \leqslant C\, K_{Q - 1/H}(d(x,y)),
\end{equation}
where \( C = C(a, b, M, Q, H, \mathcal{V}) > 0 \). To this end, one first obtains by some simple changes of variables that 
\begin{align}
    &\int_a^b \int_a^b \frac{1}{(t-s)^{QH}}\text{exp}(-C_{2,M,\mathcal{V}}\frac{d(x,y)^2}{(t-s)^{2H}})\ ds\ dt\nonumber\\ &\leqslant \ 2\ (b-a)\int_0^{b-a}\frac{1}{u^{QH}}\text{exp}(-C_{2,M,\mathcal{V}}\frac{r^2}{u^{2H}}) \ du\nonumber\\
     &=\ \frac{2(b-a)}{H r^\alpha}\int_0^{\frac{(b-a)^H}{r}} q^{-Q-1+1/H}\text{exp}(-C_{2,M,\mathcal{V}}\frac{1}{q^2})\ dq\nonumber\\
    &= \ \frac{2(b-a)}{H r^\alpha}\int_0^{\frac{(b-a)^H}{r}} q^{-\alpha-1}\text{exp}(-C_{2,M,\mathcal{V}}\frac{1}{q^2})\ dq,\label{eq:HitLowPf}
\end{align}
where we set $r:=d(x,y)\leqslant 2M$ and $\alpha:=Q-1/H$.
We now verify the inequality \eqref{verity2} by considering three separate cases.

\begin{itemize}
\item \underline{\textit{Case 1}: \( Q < 1/H \) (i.e. \( \alpha < 0 \)).} We split the integral \eqref{eq:HitLowPf} as follows:
\begin{equation*}
\begin{aligned}
&\frac{2(b-a)}{H r^\alpha} \int_0^{\frac{(b-a)^H}{r}} q^{-\alpha - 1} \exp\left( -C_{2,M,\mathcal{V}} \frac{1}{q^2} \right) dq \\
&\leqslant\ \frac{2(b-a)}{H r^\alpha} \left[ \int_0^{\frac{(b-a)^H}{2M}} q^{-\alpha - 1} \exp\left( -C_{2,M,\mathcal{V}} \frac{1}{q^2} \right) dq + \int_{\frac{(b-a)^H}{2M}}^{\frac{(b-a)^H}{r}} q^{-\alpha - 1} dq \right] \\
&= \frac{2(b-a)}{H r^\alpha} \left[ \xi_{a,b,M,\alpha,H,\mathcal{V}} + \frac{1}{-\alpha} \left( \frac{r^\alpha}{(b-a)^{\alpha H}} - \frac{(2M)^\alpha}{(b-a)^{\alpha H}} \right) \right] \\
&\leqslant \frac{2(b-a)}{H} \left[ \xi_{a,b,M,\alpha,H,\mathcal{V}} (2M)^{-\alpha} + \frac{1}{-\alpha \cdot (b-a)^{\alpha H}} \right],
\end{aligned}
\end{equation*}
where
\[
\xi_{a,b,M,\alpha,H,\mathcal{V}} := \int_0^{\frac{(b-a)^H}{2M}} q^{-\alpha - 1} \exp\left( -C_{2,M,\mathcal{V}} \frac{1}{q^2} \right) dq.
\]
Since \( K_\alpha(d(x,y)) \equiv 1 \) in this case, the estimate \eqref{verity2} holds with
\[
C = \frac{2(b-a)}{H} \left[ \xi_{a,b,M,\alpha,H,\mathcal{V}} (2M)^{-\alpha} + \frac{1}{(-\alpha)  (b-a)^{\alpha H}} \right].
\]

\item \underline{\textit{Case 2}: \( Q = 1/H \) (i.e. \( \alpha = 0 \)).} In this case, one has
\begin{equation*}
\begin{aligned}
&\frac{2(b-a)}{H} \int_0^{\frac{(b-a)^H}{r}} q^{-1} \exp\left( -C_{2,M,\mathcal{V}} \frac{1}{q^2} \right) dq \\
&\leqslant\  \frac{2(b-a)}{H} \left[ \int_0^{\frac{(b-a)^H}{2M}} q^{-1} \exp\left( -C_{2,M,\mathcal{V}} \frac{1}{q^2} \right) dq + \int_{\frac{(b-a)^H}{2M}}^{\frac{(b-a)^H}{r}} q^{-1} dq \right] \\
&= \frac{2(b-a)}{H} \left[ \zeta_{a,b,M,H,\mathcal{V}} + \log\left( \frac{2M}{r} \right) \right],
\end{aligned}
\end{equation*}
where
\[
\zeta_{a,b,M,H,\mathcal{V}} := \int_0^{\frac{(b-a)^H}{2M}} q^{-1} \exp\left( -C_{2,M,\mathcal{V}} \frac{1}{q^2} \right) dq.
\]
By setting \( N_0 := e^{\zeta_{a,b,M,H,\mathcal{V}}} \cdot 2M \), one obtains that
\[
\int_a^b \int_a^b \frac{1}{(t-s)^{QH}} \exp\left( -C_{2,M,\mathcal{V}} \frac{d(x,y)^2}{(t-s)^{2H}} \right) ds\, dt \leqslant \frac{2(b-a)}{H} \log\left( \frac{N_0}{d(x,y)} \right),
\]
which matches the definition of \( K_0(d(x,y)) \). Therefore, the estimate \eqref{verity2} holds with \( C = \frac{2(b-a)}{H} \).

\item \underline{\textit{Case 3}: \( Q > 1/H \) (i.e. \( \alpha > 0 \)).} In this case, one has
\begin{equation*}
\begin{aligned}
&\frac{2(b-a)}{H r^\alpha} \int_0^{\frac{(b-a)^H}{r}} q^{-\alpha - 1} \exp\left( -C_{2,M,\mathcal{V}} \frac{1}{q^2} \right) dq\\
&\leqslant \frac{2(b-a)}{H r^\alpha} \int_0^{\infty} q^{-\alpha - 1} \exp\left( -C_{2,M,\mathcal{V}} \frac{1}{q^2} \right) dq
= \frac{2(b-a)}{H} \eta_{\alpha,M,\mathcal{V}} \cdot r^{-\alpha},
\end{aligned}
\end{equation*}
where
\[
\eta_{\alpha,M,\mathcal{V}} := \int_0^{\infty} q^{-\alpha - 1} \exp\left( -C_{2,M,\mathcal{V}} \frac{1}{q^2} \right) dq.
\]
Since \( K_\alpha(d(x,y)) = d(x,y)^{-\alpha} \), the estimate \eqref{verity2} holds with \( C = \frac{2(b-a)}{H} \eta_{\alpha,M,\mathcal{V}} \).
\end{itemize}
Now the verification of the condition (\textbf{A2}) and thus the proof of the theorem is complete.
\end{proof}

\subsection{Upper bound for hitting probabilities}

We now proceed to establish the upper bound
in Theorem \ref{GW3}. As in the lower bound argument, we first establish a general result and restrict to the RDE context. 

According to \cite[Theorem 5.12]{BNOT}, the upper bound stated in \eqref{gcupplowboun123} can be obtained under a general criterion involving  Hausdorff measures. We first recall some basic definitions. We continue to assume that the collection of vector fields \( \{V_\alpha\}_{\alpha=1}^d \), each belonging to \( C_b^\infty(\mathbb{R}^N; \mathbb{R}^N) \), is equiregular throughout \( \mathbb{R}^N \), with the corresponding homogeneous dimension denoted by \( Q \). The metric \( d(\cdot, \cdot) \) refers to the control distance defined by \eqref{eq:ODE}. Given any non-negative real number \( \alpha \), the \textit{Hausdorff measure of dimension} \( \alpha \) for a subset \( A \) of \( \mathbb{R}^N \) is defined by
\begin{equation}\label{Hausdorff}
\begin{aligned}
\mathcal{H}_\alpha(A) &:= \lim_{\delta \downarrow 0} \mathcal{H}_\alpha^\delta(A); \\
\mathcal{H}_\alpha^\delta(A) &:= \inf\left\{ \sum_{i} (2r_i)^\alpha : A \subseteq \bigcup_{i} B_d(x_i, r_i),\ \sup_i r_i \leqslant \delta \right\}.
\end{aligned}
\end{equation}
For any Borel subset \( A \subseteq \mathbb{R}^N \), we define the \textit{order-\( \alpha \) weighted Hausdorff measure} as
\begin{equation}\label{WeightedHausdorff}
\begin{aligned}
\lambda_\alpha(A) &= \lim_{\delta \downarrow 0} \lambda_\alpha^\delta(A), \\
\lambda_\alpha^\delta(A) &= \inf\left\{ \sum_i c_i\, d(E_i)^\alpha : 1_A \leqslant \sum_i c_i\, 1_{E_i},\ c_i > 0,\ \sup_i d(E_i) \leqslant \delta \right\}.
\end{aligned}
\end{equation}
where for each set \( E \), the diameter \( d(E) \) refers to the greatest distance between any two points in \( E \) with respect to the metric \( d \).

In order to apply Hausdorff-type criteria to control hitting probabilities, we will need a generalised version of Frostman's lemma adapted to compact metric spaces. We refer the reader to  \cite[Theorem 8.17]{mattila1999geometry} for its proof.
\begin{lemma}[Frostman's Lemma]\label{Frostman}
Let \( U \) be a compact metric space and fix \( 0 \leqslant \alpha < \infty \), \( 0 < \delta \leqslant \infty \). Then there exists a Radon measure \( \mu \) on \( U \) such that \( \mu(U) = \lambda^\delta_\alpha(U) \) and
\begin{equation}\label{cononmu}
\mu(E) \leqslant d(E)^\alpha \quad \text{for all } E \subseteq U \text{ with } d(E) < \delta.
\end{equation}
In particular, if \( \mathcal{H}_\alpha(U) > 0 \), then there exist \( \delta > 0 \) and a Radon measure \( \mu \) satisfying \eqref{cononmu} with \( \mu(U) > 0 \).
\end{lemma}

Let \( (u_t)_{0 \leqslant t \leqslant T} \) be a continuous stochastic process taking values in \( \mathbb{R}^N \). We consider the following assumption.

\vspace{2mm}\noindent $\mathbf{(A3)}$ The  estimate 
\begin{equation}\label{WTS2}
\mathbb{P}\left(u([s, t]) \cap B_d(z, \varepsilon) \neq \emptyset\right) \leqslant C \varepsilon^\beta
\end{equation}
holds uniformly for all $z\in B_d(o,M)$, $\varepsilon>0$, $s,t\in [a,b]$ with $|t-s|=\varepsilon^{1/H}$. Here $0<a<b\leqslant T$, $\beta>0$, $M>0$, $H\in(0,1)$ are suitable constants and $C>0$ possibly depends on them but not on $s,t,z$.

\vspace{2mm}Assuming condition $\mathbf{(A3)}$, we will first derive the desired upper bound for the hitting probabilities. The justification of $\mathbf{(A3)}$ in the RDE context will be addressed in a later part. 


\begin{lemma}\label{8.7}
Let \( (u_t)_{0 \leqslant t \leqslant T} \) be a continuous stochastic process satisfying assumption $\mathbf{(A3)}$. Fix \( \beta > 0 \), \( 0 < a < b \leqslant T \) and \( M > 0 \). Then the following statements hold true.
\begin{itemize}
\item There exists a constant \( C_1 > 0 \) such that
\begin{equation}\label{WTS}
\mathbb{P}\big(u([a, b]) \cap A \neq \emptyset\big) \leqslant C_1\, \mathcal{H}_{\beta - 1/H}(A)
\end{equation}for all Borel sets \( A \subseteq B_d(o, M) \).

\item For every \( \eta > 0 \), there exists a constant \( C_2 > 0 \) depending additionally on 
\( \eta \), such that
\begin{equation}\label{WTS1}
\mathbb{P}\big(u([a, b]) \cap A \neq \emptyset\big) \leqslant C_2\, \textnormal{Cap}_{\beta - 1/H - \eta}(A)
\end{equation}for all compact sets \( A \subseteq B_d(o, M) \).
\end{itemize}
\end{lemma}
\begin{proof}
We assume that \( A \neq \emptyset \) for otherwise there is nothing to prove. We proceed by considering two cases.

\underline{\textit{Case 1: \( \beta \leqslant 1/H \)}.} From the definitions of the Hausdorff measure and capacities in \eqref{Hausdorff} and \eqref{Kalpha}, one has
\[
\textnormal{Cap}_{\beta - 1/H - \eta}(A) = 1 \quad \text{and} \quad \mathcal{H}_{\beta - 1/H}(A) = \infty
\]for any \( A \subseteq B_d(o,M) \) and any \( \eta > 0 \). Hence, both upper bounds in \eqref{WTS} and \eqref{WTS1} are  satisfied in this case, with constants $C_1=C_2=1$.

\underline{\textit{Case 2: \( \beta > 1/H\)}.} We first establish \eqref{WTS}, which is an easier  consequence of $\mathbf{(A3)}$. Fix \( \varepsilon \in (0,1) \) and consider a partition \( \{I_k\}_{k \geqslant 1} \) of the interval \( [a, b] \), where each subinterval \( I_k \) has length \( \varepsilon^{1/H} \) and the intervals have pairwise disjoint interiors. Then
\begin{equation*}
\begin{aligned}
\mathbb{P}\big(u([a, b]) \cap B_d(z, \varepsilon) \neq \emptyset\big)
&\leqslant \sum_{k: I_k \cap [a, b] \neq \emptyset} \mathbb{P}\big(u(I_k) \cap B_d(z, \varepsilon) \neq \emptyset\big) \\
&\leqslant \left( \left\lfloor \frac{b - a}{\varepsilon^{1/H}} \right\rfloor + 1 \right) C_H\, \varepsilon^\beta \leqslant C_{a,b,\beta,H}\, \varepsilon^{\beta - 1/H}.
\end{aligned}
\end{equation*}
It follows that for any countable collection of open balls \( \{B_d(x_i, r_i)\}_{i=1}^\infty \) such that \( A \subseteq \bigcup_{i=1}^\infty B_d(x_i, r_i) \) and \( \sup_i r_i \leqslant \delta \), one has
\begin{equation}\label{Hdelta0}
\begin{aligned}
\mathbb{P}\big(u([a, b]) \cap A \neq \emptyset\big)
&\leqslant \sum_{i=1}^\infty \mathbb{P}\big(u([a, b]) \cap B_d(x_i, r_i) \neq \emptyset\big) \\
&\leqslant C_{a,b,\beta,H} \cdot 2^{1/H - \beta} \sum_{i=1}^\infty (2r_i)^{\beta - 1/H}.
\end{aligned}
\end{equation}
Taking the infimum of both sides of \eqref{Hdelta0} over all such coverings \( \{B_d(x_i, r_i)\}_{i=1}^\infty \) with \( \sup_i r_i \leqslant \delta \), one obtains that
\begin{equation}\label{Hdelta}
\mathbb{P}\big(u([a, b]) \cap A \neq \emptyset\big)\leqslant C_1~\mathcal{H}^\delta_{\beta - 1/H}(A),
\end{equation}
where the constant is given by \( C_1 = 2^{1/H - \beta} C_{a,b,\beta,H} \). Since the left hand side of \eqref{Hdelta} is independent of \( \delta \), sending \( \delta \downarrow 0 \) yields
\[
\mathbb{P}\big(u([a, b]) \cap A \neq \emptyset\big) \leqslant C_1 \, \mathcal{H}_{\beta - 1/H}(A).
\]
This establishes \eqref{WTS}.

Next, we proceed to establish \eqref{WTS1}. We assume that
\[
\gamma := \mathbb{P}\big(u([a, b]) \cap A \neq \emptyset\big) > 0.
\]
To prove \eqref{WTS1}, we propose the following sufficient condition. Suppose that for some $\delta\geqslant 2M$, there exists a probability measure \( \bar{\mu} \in \mathcal{P}(A) \) such that
\begin{equation}\label{suffcond}
\bar{\mu}(E) \leqslant \gamma^{-1} C_1 \cdot 30^{\beta - 1/H} \cdot d(E)^{\beta - 1/H} \quad \text{for all } E \subseteq \mathbb{R}^N \text{ with } d(E) < \delta.
\end{equation}
We claim that the estimate \eqref{WTS1} follows from this assumption.

\begin{itemize}
\item \textbf{Sufficiency of \eqref{suffcond}.}  
Fix \( \eta > 0 \) such that \( \alpha := \beta - 1/H - \eta > 0 \). Recall that the \( \alpha \)-energy of the measure \( \bar{\mu} \) is defined by
\[
\mathcal{E}_\alpha(\bar{\mu}) = \int \int d(x, y)^{-\alpha} \, \bar{\mu}(dx) \, \bar{\mu}(dy).
\]
One has
\begin{equation*}
\begin{aligned}
    &\ \int d(x,y)^{-\alpha}~\bar{\mu}(dy)=\int_0^\infty\bar{\mu}(\{y:d(x,y)^{-\alpha}\geqslant u\})du\\
    &=\ \int_0^\infty\bar{\mu}(B_d(x,u^{-1/\alpha}))\ du\ \overset{u^{-1/\alpha}=r}{=} \alpha \int_0^\infty r^{-\alpha-1}\ \bar{\mu}(B_d(x,r))\ dr\\
    &=\alpha \int_0^{2M} r^{-\alpha-1}\ \bar{\mu}(B_d(x,r))\ dr +\alpha\int_{2M}^\infty r^{-\alpha-1}\ \bar{\mu}(B_d(x,r))\ dr\\
    &:=U_1(x)+U_2(x).
\end{aligned}
\end{equation*}
For the  $U_1$-term, \eqref{suffcond} implies that
$$
\frac{U_1(x)}{\gamma^{-1} C_1 \cdot 30^{\alpha+\eta} \cdot\alpha}\leqslant\int_0^{2M} r^{-\alpha-1}\cdot (2r)^{\alpha+\eta}\ dr = \frac{2^{\alpha+2\eta} M^\eta}{\eta},
$$for all $x\in\mathbb{R}^N$, where we used the fact that $d(B_d(x,r))\leqslant 2r$. For the $U_2$-term, \eqref{suffcond} implies that
$$
\frac{U_2(x)}{\gamma^{-1} C_1 \cdot 30^{\alpha+\eta} \cdot\alpha}\leqslant\int_{2M}^\infty r^{-\alpha-1}\cdot (4M)^{\alpha+\eta}\ dr = \frac{2^{\alpha+2\eta} M^\eta}{\alpha}
$$for all $x\in A$, where we used the fact that  \( A \subseteq B_d(x, 2M) \) for any $x\in A$, so that 
\[
\bar{\mu}(B_d(x, r)) = \bar{\mu}(B_d(x, 2M)) \quad \text{for all } r \geqslant 2M.
\] As a result, one has
\begin{equation*}
\begin{aligned}
\mathcal{E}_\alpha(\bar{\mu})=&\int_{A}\left(U_1(x)+U_2(x)\right)\bar{\mu}(dx)\\
\leqslant &~\gamma^{-1} C_1 \cdot 30^{\alpha+\eta} \cdot(\alpha\eta^{-1}+1)\cdot2^{\alpha+2\eta} M^\eta\\
:=&~\gamma^{-1}C_2,
\end{aligned}
\end{equation*}
where $C_2:=C_1 \cdot 30^{\alpha+\eta} \cdot(\alpha\eta^{-1}+1)\cdot2^{\alpha+2\eta} M^\eta$. It follows that
\[
\mathbb{P}\big(u([a, b]) \cap A \neq \emptyset\big) = \gamma 
\leqslant C_2 \left[ \inf_{\mu \in \mathcal{P}(A)} \mathcal{E}_{\beta - 1/H - \eta}(\mu) \right]^{-1} 
= C_2 \, \mathrm{Cap}_{\beta - 1/H - \eta}(A).
\]
This completes the proof of \eqref{WTS1}.
\end{itemize}

Now we come back to construct a measure \( \bar{\mu} \) that satisfies \eqref{suffcond}. According to \eqref{Hdelta}, one has
\begin{equation}\label{H30delta1}
0 < \gamma = \mathbb{P}\big(u([a, b]) \cap A \neq \emptyset\big) \leqslant C_1\, \mathcal{H}^{\delta}_{\beta - 1/H}(A), \quad \text{for all } \delta > 0.
\end{equation}
A key tool is  \cite[Lemma 8.16]{mattila1999geometry}, which states that for any compact metric space \( U \) and real numbers \( 0<\delta\leq\infty, \ 0 \leqslant \alpha < \infty \), one has
\begin{equation}\label{H30delta2}
\mathcal{H}_\alpha^{30\delta}(U) \leqslant 30^\alpha\, \lambda_\alpha^\delta(U).
\end{equation}
Combining \eqref{H30delta1} and \eqref{H30delta2}, one finds that 
\begin{equation}\label{H30delta3}
0 < \gamma \leqslant C_1\, \mathcal{H}^{30\delta}_{\beta - 1/H}(A) \leqslant 30^{\beta - 1/H} C_1\, \lambda_{\beta - 1/H}^\delta(A), \quad \text{for all } \delta > 0.
\end{equation}

Fix a \( \delta \) sufficiently large (say, \( \delta \geqslant 2M \)). Applying Lemma \ref{Frostman} to the set \( A \) and the measure \( \lambda_{\beta - 1/H}^\delta \), one obtains a Radon measure \( \mu \) such that
\begin{equation*}
\begin{aligned}
&\mu(A) = \lambda_{\beta - 1/H}^\delta(A), \quad \text{and} \quad \text{supp}(\mu) \subseteq A,\\  
&\mu(E) \leqslant d(E)^{\beta - 1/H} \quad \text{for all } E \subseteq \mathbb{R}^N \text{ with } d(E) < \delta.
\end{aligned}
\end{equation*}
Since \( A \subseteq B_d(o, M) \) is a non-empty bounded set, the definition of \( \lambda^\delta_\alpha \) ensures
\[
0 < \mu(A) = \lambda_{\beta - 1/H}^\delta(A) < \infty.
\]
One then define a probability measure $\bar{\mu}\in\mathcal{P}(A)$ by
\[
\bar{\mu}(\cdot) := \frac{\mu(\cdot \cap A)}{\mu(A)}.
\]
Duo to \eqref{H30delta3} and the relation $\mu(A) = \lambda_{\beta - 1/H}^\delta(A)$, one has
\begin{equation*}
\bar{\mu}(E) \leqslant \gamma^{-1} C_1 \cdot 30^{\beta - 1/H} \cdot d(E)^{\beta - 1/H} 
\end{equation*}for all $E\subseteq\mathbb{R}^N$ with $d(E)<\delta$. 
This justifies the condition \eqref{suffcond} and therefore completes the proof of the lemma.
\end{proof}

Now we restrict to the RDE context and establish the main upper bound in Theorem \ref{GW3}. We first present a moment continuity estimate with respect to the control distance $d$ which may also be of independent interest.

\begin{lemma}\label{controlmoment}
Under the assumptions in Theorem \ref{theoremlowerbound}, there exists a constant $C_1=C(H,\mathcal{V},p,T)$ such that
$$
\mathbb{E}\left[d(Y_s,Y_t)^p\right]\leqslant C_1|t-s|^{pH}\,\quad\forall~s,t\in[0,T].
$$
Moreover, for any $\eta\in(0,H)$ and $p\geqslant1$,  there exists a constant $C_2=C(H,\mathcal{V},p,T,\eta)$, such that 
$$
\mathbb{E}\left[\sup_{s,t\in[0,T]}\frac{d(Y_s,Y_t)^p}{|t-s|^{p\eta}}\right]\leqslant C_2 .
$$
\end{lemma}

\begin{proof}
We assume without loss of generality that $s=0$ and aim at proving the following estimate:
\begin{equation}\label{WTSscaling1}
\mathbb{E}\left[d(y_0, Y_t)^p\right] \leqslant C t^{pH}, \quad \text{for all } t \in [0, T],
\end{equation}
where the constant \( C \) depends only on \( H \), \( \mathcal{V} \), \( p \), and \(T\).
To this end, consider the rescaled rough differential equation
\begin{equation}\label{scalinequation}
dZ_r = \varepsilon^{H} \sum_{i=1}^d V_i(Z_r) \, dB_r, \quad Z_0 = z_0 \in \mathbb{R}^N,
\end{equation}
and let \( (Y_r^\varepsilon)_{r \geqslant 0} \) denote its solution. By the scaling property of fBM, it is obvious that 
\[
d(y_0, Y_t) \overset{d}{=} d(y_0, Y_1^t)
\]for all $t\in[0,T]$
where \( Y_1^t \) is the solution at time \( 1 \) to the equation~\eqref{scalinequation} with \( \varepsilon = t \). Therefore, the estimate~\eqref{WTSscaling1} is equivalent to the following bound:
\begin{equation}\label{WTSscaling2}
\mathbb{E}\left[d(y_0, Y_1^t)^p\right] \leqslant C t^{pH} \quad \forall t \in [0, T].
\end{equation}To establish this estimate, we localise the expectation on the events \( \{\omega:\|Y^t_1 - y_0\|_{\mathbb{R}^N} \geqslant 1\} \) and \( \{\omega:\|Y^t_1 - y_0\|_{\mathbb{R}^N} < 1\} \) separately.

\underline{\textit{Case 1 : Localisation on \(  \{\omega:\|Y^t_1 - y_0\|_{\mathbb{R}^N} \geqslant 1\} \).}} To proceed, we need a comparison result between the control distance \( d(x, y) \) and the Euclidean distance \( \|x - y\|_{\mathbb{R}^N} \) in \( \mathbb{R}^N \). According to \cite[Equation (3.22)]{kusuoka1987applications}, there exists a constant \( K \geqslant 1 \), depending only on the vector fields, such that
\begin{equation}\label{(3.22)}
\frac{1}{K}\|x - y\|_{\mathbb{R}^N} \leqslant d(x, y) \leqslant K \|x - y\|^{1/\bar{l}}_{\mathbb{R}^N}
\end{equation}for all $x,y\in\mathbb{R^N}$ with $\|x - y\|_{\mathbb{R}^N} \wedge d(x, y) \leqslant 1$, 
where \( \bar{l} \) is the hypoellipticity constant appearing in Hypothesis \ref{1.2}. 

Using \eqref{(3.22)}, one can derive an upper bound for \( d(x, y) \) when \( \|x - y\|_{\mathbb{R}^N} \geqslant 1 \). Fix such $x,y$ and let $m\geqslant 1$ satisfy \( m \leqslant \|x - y\|_{\mathbb{R}^N} < m + 1 \).
Consider a partition \( \{x_i\}_{i=0}^m \subseteq \overline{xy} \), the line segment between \( x \) and \( y \), such that \( x_0 = x \), \( \|x_k - x_{k-1}\|_{\mathbb{R}^N} = 1 \) for \( 1 \leqslant k \leqslant m \), and \( \|x_m - y\|_{\mathbb{R}^N} < 1 \). It follows from the triangle inequality and the upper bound in \eqref{(3.22)} that
\begin{equation*}
\begin{aligned}
d(x, y) &\leqslant \sqrt{m+1} \left( \sum_{i=0}^{m-1} d(x_{i}, x_{i+1}) + d(x_m, y) \right) \\
&\leqslant \sqrt{m+1} (Km + K) = K (m+1)^{3/2}.
\end{aligned}
\end{equation*}
Since \( m+1 \leqslant 2\|x - y\|_{\mathbb{R}^N} \), the right hand side can be further bounded by \( 2^{3/2} K \|x - y\|^{3/2}_{\mathbb{R}^N} \). It follows that
\[
d(x, y) \leqslant 2^{3/2} K \|x - y\|^{3/2}_{\mathbb{R}^N}
\]for all $x,y$ with $\|x - y\|_{\mathbb{R}^N} \geqslant 1$.
By applying this bound to the random variable \( Y^t_1 \), one obtains that
\begin{equation}\label{star1}
\begin{aligned}
d(y_0, Y^t_1)^p \mathbf{1}_{\{\|Y^t_1 - y_0\|_{\mathbb{R}^N} \geqslant 1\}} 
&\leqslant (2^{3/2} K)^p \|Y^t_1 - y_0\|^{\frac{3p}{2}}_{\mathbb{R}^N} \mathbf{1}_{\{\|Y^t_1 - y_0\|_{\mathbb{R}^N} \geqslant 1\}} \\
&\leqslant (2^{3/2} K)^p t^{\frac{3pH}{2}} \left| \int_0^1 V(Y^t_r) \, dB_r \right|^{\frac{3p}{2}},
\end{aligned}
\end{equation}
where the last inequality is due to the definition of $(Y^t_r)_{r\geqslant 1}$ in \eqref{scalinequation}.

\underline{\textit{Case 2: Localisation on the event \( \{ \omega : \|Y^t_1 - y_0\|_{\mathbb{R}^N} < 1 \} \).}}  
By \eqref{(3.22)}, one has
\[
d(y_0, Y^t_1) \leqslant K\, \|Y^t_1 - y_0\|^{1/\bar{l}}_{\mathbb{R}^N} < K.
\]
Hence,
\begin{equation}\label{(1)}
d(y_0, Y^t_1)^p \mathbf{1}_{\{\|Y^t_1 - y_0\|_{\mathbb{R}^N} < 1\}} \leqslant d(y_0, Y^t_1)^p \mathbf{1}_{\{d(y_0, Y^t_1) < K\}}.
\end{equation}

We now insert the Euler scheme approximation $\mathcal{E}_{(\mathcal{V})}(y_0, \delta_{t^H}\mathbf{B})$ for \( Y^t_1 \), which starts from the initial point \( y_0 \) and runs up to time $1$ (see \cite[Definition 10.1]{friz2010multidimensional}) to obtain that
\begin{equation}\label{(3)}
d(y_0, Y^t_1) \leqslant  d(y_0, y_0 + \mathcal{E}_{(\mathcal{V})}(y_0, \delta_{t^H}\mathbf{B})) + d(y_0 + \mathcal{E}_{(\mathcal{V})}(y_0, \delta_{t^H}\mathbf{B}), Y^t_1),
\end{equation}
where \( \mathbf{B} \) denotes the step-\( \bar{l} \) signature of the fBM \( (B_r)_{0 \leqslant r \leqslant 1} \) and \( \delta_{t^H} \) is the dilation operator. According to \cite[Corollary 10.15]{friz2010multidimensional}, one has
\[
\left| y_0 + \mathcal{E}_{(\mathcal{V})}(y_0, \delta_{t^H}\mathbf{B}) - Y^t_1 \right| 
\leqslant C_{\mathcal{V},l} \|\delta_{t^H}\mathbf{B}\|_{p\text{-var};[0,1]}^{\bar{l}+1} 
= C_{\mathcal{V},\bar{l}} \|\mathbf{B}\|_{p\text{-var};[0,1]}^{\bar{l}+1} t^{H(\bar{l}+1)}.
\]
To obtain an appropriate upper bound in \eqref{(3)}, we localise on the event
\[
L := \left\{ \|\mathbf{B}_{0,1}\|_{\mathrm{CC}} \cdot t^H \leqslant 1 \right\} \cap 
\left\{ C_{\mathcal{V},\bar{l}} \|\mathbf{B}\|_{p\text{-var};[0,1]}^{\bar{l}+1} \cdot t^{H(\bar{l}+1)} \leqslant 1 \right\}.
\]
On this event, by using \cite[Equation 3.21]{kusuoka1987applications} and \eqref{(3.22)} one finds that
\begin{equation}\label{(3).1}
d(y_0, y_0 + \mathcal{E}_{(\mathcal{V})}(y_0, \delta_{t^H}\mathbf{B})) 
\leqslant C_\mathcal{V} \|\delta_{t^H}\mathbf{B}_{0,1}\|_{\mathrm{CC}} = C_\mathcal{V} \|\mathbf{B}_{0,1}\|_{\mathrm{CC}} \cdot t^H,
\end{equation}
and
\begin{equation}\label{(3).2}
\begin{aligned}
d(y_0 + \mathcal{E}_{(\mathcal{V})}(y_0, \delta_{t^H}\mathbf{B}), Y^t_1) 
&\leqslant K \left| y_0 + \mathcal{E}_{(\mathcal{V})}(y_0, \delta_{t^H}\mathbf{B}) - Y^t_1 \right|^{1/\bar{l}} \\
&\leqslant K C_{\mathcal{V},\bar{l}}^{1/\bar{l}} \|\mathbf{B}\|_{p\text{-var};[0,1]}^{(\bar{l}+1)/\bar{l}} t^{H(\bar{l}+1)/\bar{l}}.
\end{aligned}
\end{equation}

Combining \eqref{(1)}, \eqref{(3)}, \eqref{(3).1} and \eqref{(3).2}, one obtains that
\begin{equation*}
\begin{aligned}
d(y_0, Y^t_1)^p \mathbf{1}_{\{|Y^t_1 - y_0| < 1\}} 
&\leqslant d(y_0, Y^t_1)^p \mathbf{1}_{\{d(y_0, Y^t_1) < K\}} \\
&\leqslant d(y_0, Y^t_1)^p \mathbf{1}_{\{d(y_0, Y^t_1) < K\} \cap L} + K^p \mathbf{1}_{L^c} \\
&\leqslant 2^{p/2} \left( C_\mathcal{V} \|\mathbf{B}_{0,1}\|_{\mathrm{CC}} t^H + K C_{\mathcal{V},\bar{l}}^{1/\bar{l}} \|\mathbf{B}\|_{p\text{-var};[0,1]}^{(\bar{l}+1)/\bar{l}} t^{H(\bar{l}+1)/\bar{l}} \right)^p \\
&\quad + K^p \left( \|\mathbf{B}_{0,1}\|_{\mathrm{CC}}^p t^{Hp} + C_{\mathcal{V},\bar{l}}^p \|\mathbf{B}\|_{p\text{-var};[0,1]}^{(\bar{l}+1)p} t^{H(\bar{l}+1)p} \right).
\end{aligned}
\end{equation*}
Since \( 0 \leqslant t \leqslant T \) and \( \|\mathbf{B}_{0,1}\|_{\mathrm{CC}} \leqslant \|\mathbf{B}\|_{p\text{-var};[0,1]} \), there exists a constant \( C_3 := C(H, \mathcal{V}, p, T) \)  such that
\begin{equation}\label{star2}
d(y_0, Y^t_1)^p \mathbf{1}_{\{\|Y^t_1 - y_0\|_{\mathbb{R}^N} < 1\}} 
\leqslant C_3 \left( 1 + \|\mathbf{B}\|_{p\text{-var};[0,1]}^{(\bar{l}+1)p} \right) t^{Hp}.
\end{equation}

Now one concludes from \eqref{star1} and \eqref{star2} that
\begin{equation*}
\begin{aligned}
d(y_0, Y^t_1)^p 
&= d(y_0, Y^t_1)^p \mathbf{1}_{\{|Y^t_1 - y_0| < 1\}} + d(y_0, Y^t_1)^p \mathbf{1}_{\{|Y^t_1 - y_0| \geqslant 1\}} \\
&\leqslant (2^{3/2} K)^p t^{\frac{3pH}{2}} \left| \int_0^1 V(Y^t_r) \, dB_r \right|^{\frac{3p}{2}} 
+ C_3 \left( 1 + \|\mathbf{B}\|_{p\text{-var};[0,1]}^{(\bar{l}+1)p} \right) t^{Hp}.
\end{aligned}
\end{equation*}
Clearly, \( \|\mathbf{B}\|_{p\text{-var};[0,1]} \) has finite moments of all orders.  
The rough integral \( \int_0^1 V(Y^t_r) \, dB_r \) also has finite moments of all orders (uniformly in $t\in[0,T]$) due to \cite[Theorem 10.36, Theorem 10.47]{friz2010multidimensional}. One therefore arrives at the following estimate
\[
\mathbb{E}\left[ d(y_0, Y_t)^p \right] = \mathbb{E}\left[ d(y_0, Y^t_1)^p \right] 
\leqslant C_{H, \mathcal{V}, p, T} \, t^{Hp}.
\]
This proves the first part of the lemma. The second part follows  from Kolmogorov's continuity theorem.
\end{proof}

We are now in a position to state and prove the main result of this section. 


\begin{theorem}\label{theoremupperbound}
Under the assumptions in Theorem \ref{theoremlowerbound}, let $0<a<b$ and $M>0$ be given fixed. For any \( \eta >0 \) (if $Q-1/H>0$, we further require $\eta<Q-1/H$), there exists a positive constant \( C = C(a, b, M, Q, H, \mathcal{V}, \eta) \) such that
\begin{equation}
\mathbb{P}\Big(Y([a,b])\cap A\neq\emptyset\Big)~\leqslant C\ \mathrm{Cap}_{Q-1/H-\eta}(A)
\end{equation}for all compact sets \( A \subseteq B_d(o, M) \).
\end{theorem}
\begin{proof}
According to Lemma \ref{8.7}, it suffices to verify that $(Y_t)_{0\leqslant t\leqslant T}$ satisfies Assumption $\mathbf{(A3)}$ for any fixed $\beta\in(0,Q)$. Fix \( z \in B_d(o, M) \). For any \( \varepsilon \in (0,1) \) and \( s, t \in [a,b] \) satisfying \( |t - s| = \varepsilon^{1/H} \), we define
\[
U := d(Y_s, z), \quad \text{and} \quad R := \sup_{r,v \in [s,t]} d(Y_r, Y_v).
\]One easily checks by using the triangle inequality that
\begin{equation}\label{UR}
\mathbb{P}\big(Y([s,t]) \cap B_d(z, \varepsilon) \neq \emptyset\big) \leqslant \mathbb{P}\left(U\leqslant \varepsilon + R\right).
\end{equation}
The right hand side is estimated as
\begin{equation}\label{split1}
\begin{aligned}
\mathbb{P}\left(U\leqslant \varepsilon + R\right)
&\leqslant \mathbb{P}\left(U\leqslant \varepsilon + R,\; U \geqslant 2\varepsilon \right) + \mathbb{P}\left(U< 2\varepsilon \right) \\
&\leqslant \mathbb{P}\left(U/2\leqslant R \right) + \mathbb{P}\left(U< 2\varepsilon \right),
\end{aligned}
\end{equation}
where the second inequality uses the observation that if \( U \leqslant \varepsilon + R \) and \( U \geqslant 2\varepsilon \), then \( \varepsilon \leqslant R \), which further implies \( U \leqslant 2R \) and hence \( U/2 \leqslant R \). Now for fix \( \alpha \in (0,1) \), one has \eqref{split1}:
\begin{equation}\label{split2}
\begin{aligned}
\mathbb{P}\left(U/2\leqslant R\right)
&\leqslant \mathbb{P}\left(U/2 \leqslant R,\; R \leqslant \varepsilon^\alpha \right) + \mathbb{P}\left(R > \varepsilon^\alpha\right) \\
&\leqslant \mathbb{P}\left(U/2\leqslant \varepsilon^\alpha \right) + \mathbb{P}\left(R > \varepsilon^\alpha\right).
\end{aligned}
\end{equation}
Combining \eqref{split1} and \eqref{split2}, one obtains the following estimate:
\begin{equation*}
\mathbb{P}\left(U \leqslant \varepsilon + R\right)
\leqslant \mathbb{P}\left(U/2 \leqslant \varepsilon^\alpha \right) + \mathbb{P}\left(R > \varepsilon^\alpha\right) + \mathbb{P}\left(U < 2\varepsilon \right).
\end{equation*}
We will estimate the above three terms separately.

First of all, recall that \( Y_s\) admits a smooth density for all $0 < s \leqslant T$. Then there exists a constant \( C_1 = C_{a,b,M} > 0 \) such that
\[
\sup_{x \in B_d(o, M)} |p_s(x)| \leqslant C_1
\]holds for any \( s \in [a,b] \).
By using Lemma \ref{lemmvolcompari}, one therefore obtains that
\begin{equation*}
\begin{aligned}
&\mathbb{P}\left(U/2 \leqslant \varepsilon^\alpha \right) \leqslant C_1 \, |B_d(z, 2 \varepsilon^\alpha)| \leqslant C_1 C_{2,\mathcal{V},M}\cdot 2^{Q} \varepsilon^{\alpha Q}, \\
&\mathbb{P}\left(U< 2\varepsilon \right) \leqslant C_1 \, |B_d(z, 2 \varepsilon)| \leqslant C_1 C_{2,\mathcal{V},M}\cdot 2^{Q}  \varepsilon^{Q},
\end{aligned}
\end{equation*}
where \( C_{2,\mathcal{V},M} \) is the constant appearing in the volume comparison \eqref{volcompari}.

In addition, let \( q \in (0, 1 - \alpha) \) and $p\geqslant 1$. By using Lemma \ref{controlmoment} and Chebyshev's inequality, one finds that
\begin{equation*}
\mathbb{P}\left(R > \varepsilon^\alpha\right)
\leqslant \mathbb{P}\left(\sup_{r,v \in [s,t]} \frac{d(Y_r, Y_v)}{|v - r|^{H(\alpha + q)}} > \varepsilon^{-q} \right)
\leqslant C_2 \, \varepsilon^{q p},
\end{equation*}
where \( |s - t| = \varepsilon^{1/H} \) and \( C_2 \) depends only on $H,\mathcal{V},p,b,\eta$. Now fix any \( \beta \in (0, Q) \). By choosing \( \alpha \) and \( p \) sufficiently large such that \( \alpha Q \geqslant \beta \) and \( q p \geqslant \beta \), one concludes that
\begin{equation*}
\begin{aligned}
\mathbb{P}\left(Y([s,t]) \cap B_d(z, \varepsilon) \neq \emptyset\right)
&\leqslant \mathbb{P}\left(U \leqslant \varepsilon + R\right) \\
&\leqslant \mathbb{P}\left(U/2 \leqslant \varepsilon^\alpha \right)
+ \mathbb{P}\left(R > \varepsilon^\alpha\right)
+ \mathbb{P}\left(U< 2\varepsilon\right) \\
&\leqslant 2^{Q}C_1 C_{2,\mathcal{V},M} (\varepsilon^{\alpha Q} + \varepsilon^Q) + C_2 \varepsilon^{q p} \\
&\leqslant C \, \varepsilon^\beta,
\end{aligned}
\end{equation*}
where \( C \) depends on \( C_1 \), \( C_2 \), and \( C_{2,\mathcal{V},M} \). This yields the desired result.
\end{proof}
\appendix

\renewcommand{\thetheorem}{\Alph{section}.\arabic{theorem}}

\section{Proof of Lemma \ref{logsig3}}

In this appendix, we prove Lemma \ref{logsig3}. 
As a starting point, we make the following observations.

\begin{itemize}
  \item Using the stationarity of increments proved by Lemma \ref{decom}, we may fix \( s = 0 \) without loss of generality. More precisely, Lemma \ref{decom}, combined with Lemma \ref{equnorm}, gives that
  \[
  \left\| \gamma(\mathbf{U}_{s,s+\varepsilon})^{-1} \right\|_{k,p,s}~\overset{M}{=}~\left\| \gamma(\mathbf{U}_{0,\varepsilon})^{-1} \right\|_{k,p,0}.
  \]
  \item By the scaling property of fBm, it suffices to consider \( t = 1 \) by introducing the rescaled dynamics over  $r\in[0,1]$:
\begin{equation}\label{rephrase}
d\mathbf{U}^{\varepsilon,x}_r = \sum_{i=1}^d A^\varepsilon_i(\mathbf{U}^{\varepsilon,x}_r)\, dB^i_r= \sum_{i=1}^d \varepsilon^H A_i(\mathbf{U}^{\varepsilon,x}_r)\, dB^i_r,\quad \mathbf{U}^{\varepsilon,x}_0=x\in\mathfrak{g}^{(l)}(\mathbb{R}^d).
\end{equation}
Here $A_i$ is the left invariant vector field on $\mathfrak{g}^{(l)}(\mathbb{R}^d)$ induced by the canonical basis vector $e_i$ and note that the vector fields $\{A_i\}_{i=1}^d$ satisfy H\"ormander's condition on $\mathfrak{g}^{(l)}(\mathbb{R}^d)$. The vector fields $A^\varepsilon_i$ are defined by $A^\varepsilon_i(u)=\varepsilon^H A_i(u)$. We use $o\in\mathfrak{g}^{(l)}(\mathbb{R}^d)$ to denote the origin and to ease notation, we will use $\mathbf{U}^{\varepsilon}_t$ instead of $\mathbf{U}^{\varepsilon,o}_t$ when the initial value is $o$. Note that \( \mathbf{U}^{\varepsilon}_1 =\delta_{\varepsilon^H} \mathbf{U}_{0,1}\overset{d}{=}\mathbf{U}_{0,\varepsilon}\).
This formulation allows one to reduce the analysis to the unit time interval \([0,1]\), where the dependence on \( \varepsilon \) encodes small-time behaviour in the original problem.
\end{itemize}

Based on Lemma~\ref{equnorm}, one can apply a similar argument as in the derivation of~\eqref{ImHeHmIm} to show that 
\[
\left\| \gamma(\mathbf{U}_{s,s+\varepsilon})^{-1} \right\|_{k,p,s}~\overset{M}{=}~\left\| \gamma(\mathbf{U}_{1}^\varepsilon)^{-1} \right\|_{k,p,0}
\]for any $s\geqslant 0,\varepsilon>0$, $k\geqslant 1$ and $p>0$.
Now the task is reduced to showing that
\begin{equation}\label{toshowapp1}
\left\|\gamma(\mathbf{U}_{1}^\varepsilon)^{-1} \right\|_{k,p,0} \leqslant C\varepsilon^{-\alpha} \Psi
\end{equation}
for all $0<\varepsilon\leqslant T$ with some smooth random variable $\Psi$.

First of all, it is well-known that for any $\varepsilon\in(0,1]$ and $t>0$, the map $x\to \mathbf{U}^{\varepsilon,x}_t$ is a flow of $C^\infty$ diffeomorphism. 
We denote the Jacobian of $\mathbf{U}^{\varepsilon,x}_t$ by 
$
\mathbf{J}^{\varepsilon,x}_t = \partial_x \mathbf{U}^{\varepsilon,x}_t$
and let $\beta_I^{J,\varepsilon}(t,x)$ ($I,J\in\mathcal{A}_1(l)$) satisfy
\begin{equation}\label{defJtep}
(\mathbf{J}^{\varepsilon,x}_t)^{-1}A^\varepsilon_I(\mathbf{U}^{\varepsilon,x}_t)= \sum_{J\in\mathcal{A}_1(l)}\beta_I^{J,\varepsilon}(t,x)A^\varepsilon_J(x),\quad 0\leqslant t\leqslant 1.
\end{equation}
The precise definition of $\{\beta_I^{J,\varepsilon}(\cdot,\cdot)\}_{I,J\in\mathcal{A}(l)}$ can be found in \cite[Equation (3.5)]{BOZ}. We define the Malliavin matrix $\mathbf{M}^\varepsilon(x)$ by setting
\begin{equation}\label{defMIJepx}
M^\varepsilon_{I,J}=(\mathbf{M}^\varepsilon(x))_{I,J}=\langle\mathscr{L}(\beta^{\varepsilon,I}(\cdot,x)),\mathscr{L}(\beta^{\varepsilon,J}(\cdot,x))\rangle_{\mathcal{H}_{W,0}}
\end{equation}for any $1\leqslant I,J\leqslant \mathcal{A}_1(l)$,
where $\mathscr{L}$ is defined in  \eqref{innerprodKtoL} and $\beta^{\varepsilon,I}(\cdot,x)$ means the column vector $(\beta^{\varepsilon,I}_1(\cdot,x),\cdots,\beta^{\varepsilon,I}_d(\cdot,x))$. 
The following result is taken from \cite[Lemma 3.9]{BOZ}.
\begin{lemma}\label{lamleq2l}
Let $\mathbf{U}^{\varepsilon,x}$, $\beta_I^{J,\varepsilon}$ and $M^\varepsilon_{I,J}$ defined by \eqref{rephrase}, \eqref{defJtep} and \eqref{defMIJepx} respectively. Then one has
$$
\lambda_{\max}(\gamma(\mathbf{U}^{\varepsilon,x}_1)^{-1})\leqslant C\varepsilon^{-2Hl} \lambda_{\max}((\mathbf{M}^\varepsilon(x))^{-1})\lambda_{\max}((\mathbf{J}^{\varepsilon,x}_1(\mathbf{J}^{\varepsilon,x}_1)^*)^{-1}),
$$
where $C$ is a positive constant independent of $\varepsilon$ and $\lambda_{\max}(\cdot)$ means taking the maximal eigenvalue of a matrix.
\end{lemma}

In view of the above Lemma, it remains to control the eigenvalues $\lambda_{\max}((\mathbf{M}^\varepsilon(o))^{-1})$ and $\lambda_{\max}((\mathbf{J}^{\varepsilon,o}_1(\mathbf{J}^{\varepsilon,o}_1)^*)^{-1})$ individually. The latter is easier which we shall first address.

\begin{proposition}\label{lamJJ}
   There exists a smooth random variable $\Psi$ such that
    $$
    \lambda_{\max}((\mathbf{J}^{\varepsilon,o}_1(\mathbf{J}^{\varepsilon,o}_1)^*)^{-1})\leqslant\Psi,
    $$
    for any $\varepsilon\in(0,1]$.
    \begin{proof}
     It is enough to prove that each entry of $(\mathbf{J}^{\varepsilon,o}_1)^{-1}$ is bounded by certain smooth random variable $\Psi$. 
     Note that 
     $
     \mathbf{U}^{\varepsilon,x}_t = x~\star~\mathbf{U}^{\varepsilon,o}_t
     $
     where $\star$ is the group law induced from $G^{(l)}(\mathbb{R}^d)$. Since the inverse of $x~\mapsto~x~\star~\mathbf{U}^{\varepsilon,o}_1$ is $y~\mapsto~y~\star~(\mathbf{U}^{\varepsilon,o}_1)^{-1}=: H(y)$, it follows that 
     $$
     (\mathbf{J}^{\varepsilon,o}_1)^{-1} = \frac{H(y)}{\partial y}\Big|_{y=\mathbf{U}^{\varepsilon,o}_1}
     $$
    Observe that every component of $\mathbf{U}^{\varepsilon,o}_1$ and $(\mathbf{U}^{\varepsilon,o}_1)^{-1}$ is a polynomial of
    $$
    \varepsilon^{Hm}\int_{0< t_1<\cdots< t_m< 1}d B^{i_1}_{t_1}\cdots d B^{i_d}_{t_m}.
    $$
    As in the argument of Lemma \ref{logsig1}, this can be bounded above by the random variable
\(\exp( \eta \left( \mathcal{N}^{0,1}_{\gamma,2q}(\mathbf{B}) \right)^{1/q})\) 
for some \(q \geqslant 1\), \(\eta \in (0, \eta_0]\) and \(\gamma < H\).
\end{proof}
\end{proposition}

The analysis of $\lambda_{\max}((\mathbf{M}^\varepsilon(o))^{-1})$ is more involved and we need several technical lemmas.

\begin{lemma}\label{indofepsi}
Let $I,J \in \mathcal{A}_1(l)$ with $|I|\leqslant |J|$. Then there is no $(\varepsilon,x)$-dependence in $\beta^{J,\varepsilon}_I$ and more specifically, one has
\[
\beta^{J,\varepsilon}_I(t,x)=
\begin{cases}
(-1)^{|K|}B^K_t, & \text{if } J=(I,K)\ \text{for some } K\in\mathcal{A}_1(l),\\[6pt]
0, & \text{otherwise}.
\end{cases}
\]
Here $B^K_t$ denotes the projection of $\mathbf{B}_{0,t}$ onto the tensor basis 
$e_{k_1}\otimes\cdots\otimes e_{k_m}$ for any word $K=(k_1,\ldots,k_m)$.
\end{lemma}
\begin{proof}
This is just Lemma 4.2 and Remark 4.3 of \cite{BOZ} applied to the special context of the truncated signature process. 
\end{proof}

We define the supreme norm 
$$
\|f\|_{\infty,[s,t]}:=\sup_{u\in[s,t]}\|f(u)\|_{\mathrm{HS}}
$$
for any continuous $f$ taking values in $T^{(l)}(\mathbb{R}^d)$.
\begin{lemma}\label{5.16}
For any $m \in\mathbb{N}$ and $p \geqslant 1$, there exists a smooth function $\Psi_{m,p}$ such that
\[
\sup_{\sum_{I\in \mathcal{A}(m)} a_I^2 = 1} 
\mathbb{P}\left(
\Big\|\sum_{I\in \mathcal{A}(m)} a_I B^I_t\Big\|_{\infty,[0,1]}<\varepsilon\,\middle|\, \mathcal{F}_0^W
\right)
\leqslant \varepsilon^p\ \Psi_{m,p}
\]for all $\varepsilon>0$.
\end{lemma}
\begin{proof}
When $m=0,1$, one can use the same argument as in \cite[Lemma 4.4]{BOZ}, with the small ball estimate replaced by the corresponding conditional version given by Lemma \ref{6.4}, to show that 
\[
\sup_{\sum_{I\in \mathcal{A}(m)} a_I^2 = 1} 
\mathbb{P}\left(
\Big\|\sum_{I\in \mathcal{A}(m)} a_I B^I_t\Big\|_{\infty,[0,1]}<\varepsilon\,\middle|\, \mathcal{F}_0^W
\right)
\leqslant C_{p,H}\,\varepsilon^p
\]
with some constant $C_{p,H}$ only depending on $p,H$.

We prove the general case by induction. Assume that the result holds for $k=0,1,\cdots,m$. We only consider the case $a_\emptyset=0$; the case $a_\emptyset\neq0$ can be handled similarly as in \cite[Lemma 4.4]{BOZ}. Let $f(t)=\sum_{I\in \mathcal{A}_1(m+1)}a_I B^I_t$ with $\sum_{I\in \mathcal{A}_1(m+1)}a_I^2=1$. One can write
    $$
    f(t)=\int_0^t A_s d B_s,
    $$
    where $B_t = (B^1_t,\cdots,B^d_t)$ and $A_t = (\sum_{J\in \mathcal{A}(m)}a_{J*1}B^J_t,\cdots,\sum_{J\in \mathcal{A}(m)}a_{J*d}B^J_t)$. We choose a $k$, $1\leqslant k\leqslant d$, such that $\sum_{J\in \mathcal{A}(m)}a^2_{J*k}\geqslant 1/d$. According to the pathwise Norris' lemma (\cite[Theorem 5.6]{CHLT}), 
    $$
    \Big\|\sum_{J\in \mathcal{A}(m)}a_{J*k}B^J_t\Big\|_{\infty,[0,1]}\leqslant C R^\eta \, \big\|f(t)\big\|^r_{\infty,[0,1]}
    $$
    where $C,\eta>0$, 
    $$
    R^\eta\lesssim 1 + L_\theta^{-\eta}+\exp\left( \zeta \left( \mathcal{N}^{0,1}_{\gamma,2q}(\mathbf{B}) \right)^{1/q} \right)
    $$
    for some $\zeta\in(0,\eta_0]$ (see \eqref{Besovnorminte} for $\eta_0$) and $L_\theta$ is the $\theta$-H\"older roughness of $B$ with $2\gamma=\theta\in(H,2H)$.
    It follows that 
\begin{equation*}
\begin{aligned}
&\mathbb{P}\left(
\big\| f(t) \big\|_{\infty,[0,1]} < \varepsilon 
\,\middle|\, \mathcal{F}_0^W
\right)\\
&\leqslant 
\mathbb{P}\left(
\frac{\big\|\sum_{J\in \mathcal{A}(m)} a_{J*k} B^J_t\big\|_{\infty,[0,1]}}{C R^q} 
< \varepsilon^r 
\,\middle|\, \mathcal{F}_0^W
\right)\\
&\leqslant 
\mathbb{P}\left(
\Big\|\sum_{J\in \mathcal{A}(m)} a_{J*k} B^J_t\Big\|_{\infty,[0,1]} 
< \varepsilon^{r/2} 
\,\middle|\, \mathcal{F}_0^W
\right)
+ 
\mathbb{P}\left(
C R^\eta \geqslant \varepsilon^{-r/2} 
\,\middle|\, \mathcal{F}_0^W
\right)\\
&\leqslant 
\mathbb{P}\left(
\Big\|\sum_{J\in \mathcal{A}(m)} \frac{a_{J*k}}{\sqrt{\sum a_{J*k}^2}} B^J_t\Big\|_{\infty,[0,1]} 
< \sqrt{d}\varepsilon^{r/2} 
\,\middle|\, \mathcal{F}_0^W
\right)
+ 
\mathbb{P}\left(
C R^\eta \geqslant \varepsilon^{-r/2} 
\,\middle|\, \mathcal{F}_0^W
\right).
\end{aligned}
\end{equation*}
The first term is bounded above by $\varepsilon^p\ \Psi_p$ due to the induction hypothesis.

For the second term, one has
\begin{equation*}
\begin{aligned}
&\mathbb{P}\left(
C R^\eta \geqslant \varepsilon^{-r/2} 
\,\middle|\, \mathcal{F}_0^W
\right)\lesssim \varepsilon^{p}\mathbb{E}\left[ R^{\frac{2p\eta}{r}}\middle|\, \mathcal{F}_0^W\right]\\
 & \lesssim ~\varepsilon^{p}\left(1 +\mathbb{E}\left[ L_\theta^{-\frac{2p\eta}{r}}\middle|\, \mathcal{F}_0^W\right]+\mathbb{E}\left[\exp\left( \frac{2p\zeta}{r} \left( \mathcal{N}^{0,1}_{\gamma,2q}(\mathbf{B}) \right)^{1/q} \right)\middle|\, \mathcal{F}_0^W\right]\right) \\
\end{aligned}
\end{equation*}
One can use Proposition $\ref{6.5}$ (in the same way as Lemma 5.8 and Corollary 5.10 in \cite{CHLT}) to conclude that
$$
\mathbb{P}\left(L_{\theta}<x \,\middle|\, \mathcal{F}_0^W\right)\leqslant C_{1,p,H}\exp\left(-C_{2,p,H}x^{-1/H}\right),\quad 0<x<1
$$
for some constants $C_{1,p,H},C_{2,p,H}>0$. This further implies that 
$$
\mathbb{E}\left[ L_\theta^{-\frac{2p\eta}{r}}\middle|\, \mathcal{F}_0^W\right]\leqslant C_{p,H}.
$$
Note that  $\mathbb{E}\left[\exp\left( \frac{2p\zeta}{r} \left( \mathcal{N}^{0,1}_{\gamma,2q}(\mathbf{B}) \right)^{1/q} \right)\middle|\, \mathcal{F}_0^W\right]$ is similar to $\Psi^1$ in Lemma \ref{logsig1},  so one can choose $0<\zeta\leq\frac{\eta_0r}{2p}$ and use the same argument as the one leading to  \eqref{checksmoothness} to show that it is a smooth random variable. Therefore, the conclusion holds for $k=m+1$, which completes the induction step.
\end{proof}

\begin{corollary}
For any $m \geqslant 0$ and $p > 1$, one has
\begin{equation}\label{eq:InvMom}
\mathbb{E}\left[\inf\left\{\int_0^1\left(\sum_{I\in\mathcal{A}(m)}a_I B^I_t\right)^2 dt \; ; \; \sum_{I\in\mathcal{A}(m)} a_I^2 = 1\right\}^{-p} \,\middle|\, \mathcal{F}^W_0\right] \leqslant \Psi_{m,p},
\end{equation}
where $\Psi_{m,p}$ is a smooth random variable depending only on $m,p$.
\end{corollary}
\begin{proof}


An unconditional version of this result was obtained in \cite[Corollary 4.5]{BOZ}. Here we use a similar argument but with all the estimates replaced by the conditional ones. We first show that
\begin{equation}\label{supprob}
\sup_{\sum_{I\in \mathcal{A}(m)} a_I^2 = 1} 
\mathbb{P}\left(\int_0^1\left(\sum_{I\in \mathcal{A}(m)} a_I B^I_t\right)^2dt<\varepsilon\,\middle|\, \mathcal{F}_0^W
\right)
\leqslant \varepsilon^p\ \Psi_{m,p}
\end{equation}for all $\varepsilon>0$ with some smooth random variable $\Psi_{m,p}$.
Setting $f(t) = \sum_{I\in A(m)} a_I B_t^I$, one has
\[
\mathbb{P}\left( \int_0^1 \!\!\left( \sum_{I\in A(m)} a_I B_t^I \right)^2 dt < \varepsilon \,\middle|\, \mathcal{F}_0^W
\right)
= \mathbb{P}\!\left( \|f\|_{L^2}^2 < \varepsilon \,\middle|\, \mathcal{F}_0^W
\right)
= \mathbb{P}\!\left( \|f\|_{L^2} < \sqrt{\varepsilon} \,\middle|\, \mathcal{F}_0^W
\right).
\]
By using the interpolation inequality in \cite[Equation (6.5)]{CHLT}, i.e.
\[
\|f\|_\infty \leqslant 2\max\!\left\{\|f\|_{L^2}, \|f\|_{L^2}^{\frac{2r}{2r+1}} \|f\|_{r}^{\frac{1}{2r+1}} \right\},
\]
one finds that 
\begin{align*}
\left\{ \|f\|_{L^2} < \sqrt{\varepsilon} \right\}
\subseteq& \left\{ \frac{\|f\|_\infty}{2} < \sqrt{\varepsilon},\, \|f\|_{L^2} > \|f\|_{r} \right\}\\
&\ \ \ \bigcup\left\{~\left(\frac{\|f\|_\infty}{2\|f\|_{r}^{\frac{1}{2r+1}}}\right)^{\frac{2r+1}{2r}}<\sqrt{\varepsilon},\, \|f\|_{L^2} < \|f\|_{r} \right\}.
\end{align*}
As a consequence, 
\begin{equation*}
\begin{aligned}
&~\mathbb{P}\left(\|f\|_{L^{2}}<\sqrt{\varepsilon}~\middle|~\mathcal{F}_0^W\right)\\
&\leqslant~\mathbb{P}\left(\|f\|_{\infty}<2\sqrt{\varepsilon}~\middle|~\mathcal{F}_0^W\right)
+ \mathbb{P}\left(\|f\|_{\infty}^{\frac{2r+1}{2r}}<\varepsilon^{1/4}~\middle|~\mathcal{F}_0^W\right)
\\&\ \ \ + \mathbb{P}\left(\left(2\|f\|_{r}^{\frac{1}{2r+1}}\right)^{\frac{2r+1}{2r}}>\varepsilon^{-1/4}~\middle|~\mathcal{F}_0^W\right)\\
&\leqslant~\mathbb{P}\left(\|f\|_{\infty}<2\sqrt{\varepsilon}~\middle|~\mathcal{F}_0^W\right)
+ \mathbb{P}\left(\|f\|_{\infty}<\varepsilon^{\frac{r}{4r+2}}~\middle|~\mathcal{F}_0^W\right)
+ \mathbb{P}\left(\|f\|_{r}>2^{-2r-1}\varepsilon^{-r/2}~\middle|~\mathcal{F}_0^W\right)
\end{aligned}
\end{equation*}
The first  and  second terms on the right hand side are estimated by Lemma \ref{5.16} and the third term is upper bounded by smooth random variable due to \eqref{Besovnormboun}.

To obtain the desired estimate \eqref{eq:InvMom} from \eqref{supprob}, one only needs to apply a standard compactness argument as in \cite[Lemma 2.3.1]{Nualart}.
\end{proof}

Now we can derive the key nondegeneracy estimate for the Malliavin matrix $\mathbf{M}^\varepsilon(x)$.
\begin{proposition}\label{lamMleqPsi}
For any $p \in (1,\infty)$, there exists a smooth function $\Psi_{\bar{l},p}$ depending only on $\bar{l},p$ such that
\[
\mathbb{E}\left[
\lambda_{\max}\big((\mathbf{M}^\varepsilon(x))^{-1}\big)^{p}
\,\middle|\,
\mathcal{F}^W_{0}
\right]
\leqslant \Psi_{\bar{l},p}
\]holds for all $\varepsilon\in(0,1)$.
\end{proposition}
\begin{proof}
Recalling the definition of $(\mathbf{M}^\varepsilon(x))_{I,J}$ in \eqref{defMIJepx} and the definition of $\beta_I^J(t,x)$ in Lemma \ref{indofepsi}, one has 
\begin{equation*}
\begin{aligned}
(\mathbf{M}^\varepsilon(x))_{I,J}~=&~\langle\mathscr{L}(\beta^{\varepsilon,I}(\cdot,x)),\mathscr{L}(\beta^{\varepsilon,J}(\cdot,x))\rangle_{\mathcal{H}_{W,0}}\\
=&~(-1)^{|I|+|J|}\delta_{i_1}^{j_1}\langle\mathscr{L}(B^I_{\cdot}),\mathscr{L}(B^J_{\cdot})\rangle_{\mathcal{H}_{W,0}}:=& M_{I,J}.
\end{aligned}
\end{equation*}

We first assume $1/4<H\leqslant1/2$. For any multi-index $a$ such that $\sum_{I\in\mathcal{A}_1(\bar{l})}a_I^2=1$,  one has
\begin{equation*}
\begin{aligned}
&\sum_{I,J\in \mathcal{A}_1(\bar{l})} (-1)^{|I|+|J|}a_I a_J M_{I,J}\\
&= \sum_{i=1}^d 
\Bigl\langle 
\mathscr{L}\Bigl(
\sum_{I\in \mathcal{A}(\bar{l}-1)} a_{(i,I)} B^{(i,I)}_{\cdot}
\Bigr),
\mathscr{L}\Bigl(
\sum_{J\in \mathcal{A}(\bar{l}-1)} a_{(i,J)} B^{(i,J)}_{\cdot}
\Bigr)
\Bigr\rangle_{\mathcal{H}_{W,0}} \\
&\geqslant C_H\sum_{i=1}^d 
\left\|
\sum_{J\in \mathcal{A}(\bar{l}-1)} a_{(i,J)} B^{(i,J)}_{\cdot}\right\|_{\mathcal{H}_B}
\geqslant C_H 
\sum_{i=1}^d 
\int_0^1 
\left(\sum_{J\in \mathcal{A}(\bar{l}-1)} a_{(i,J)} B^{(i,J)}_{t} 
\right)^2 
dt,
\end{aligned}
\end{equation*}
where the second line follows from Lemma \ref{6.7}, Lemma \ref{6.6} and the third line is implied by the  inequality $\|f\|_{\mathcal{H}_B}\geqslant C\|f\|_{L^2}$  for every $f\in\mathcal{H}_B$ (see \cite[P. 414]{BOZ}). Let $k$ be such that $\sum_{J\in \mathcal{A}(\bar{l}-1)} a_{(k,J)}^2\geqslant1/d$. Then one has 
$$
\sqrt{d}\sum_{I,J\in \mathcal{A}_1(\bar{l})} (-1)^{|I|+|J|}a_I a_J M_{I,J}\geqslant C_H\int_0^1 
\left(\sum_{J\in \mathcal{A}(\bar{l}-1)} \frac{a_{(k,J)}}{\sqrt{\sum_{J\in \mathcal{A}(\bar{l}-1)} a_{(k,J)}^2}} B^{(k,J)}_{t} 
\right)^2 
dt.
$$
It follows that 
\begin{equation*}
\begin{aligned}
&\mathbb{E}\left[
\lambda_{\max}\big((M^\varepsilon(x))^{-1}\big)^{p}
\,\middle|\,
\mathcal{F}^W_{0}
\right]\\
 &\lesssim \mathbb{E}\left[\inf\left\{\int_0^1\left(\sum_{I\in\mathcal{A}(\bar{l})}a_I B^I_t\right)^2 dt \; ; \; \sum_{I\in\mathcal{A}(\bar{l})} a_I^2 = 1\right\}^{-p} \,\middle|\, \mathcal{F}^W_0\right] \leqslant \Psi_{\bar{l},p}.
\end{aligned}
\end{equation*}

If $1/2<H<1$, one uses the interpolation inequality $\|f\|_{\mathcal{H}_B}\geqslant C\frac{\|f\|_{\infty}^{3+1/\alpha}}{\|f\|^{2+1/\gamma}_{\alpha}}$ instead (See  \cite[Lemma 4.4]{baudoin2007version}). 
The similar analysis will be omitted for conciseness.  
\end{proof}

\begin{proof}[Proof of Lemma \ref{logsig3}]The conclusion follows from Lemma \ref{lamleq2l}, Proposition \ref{lamJJ} and Proposition \ref{lamMleqPsi}.
\end{proof}
\section*{Acknowledgement}

XG gratefully acknowledges the support from ARC grant DE210101352. SW gratefully acknowledges the support from Melbourne Research Scholarship. 
\printbibliography
\end{document}